\documentclass[10pt]{amsart}
\usepackage{datetime}
\usepackage{amsfonts}
\usepackage{amsmath}
\usepackage{amssymb}
\usepackage{amsthm}
\usepackage{stmaryrd}
\usepackage{eucal}
\usepackage{libertine}
\usepackage{tensor}
\usepackage{booktabs}
\usepackage{graphicx}  
\graphicspath{{../figures/}{./figures/}}
\usepackage[usenames,dvipsnames,svgnames,table]{xcolor}
\definecolor{badgerred}{rgb}{0.715,0.004,0.004}
\definecolor{burntorange}{rgb}{0.801,0.332,0.0}
\usepackage[font=small,labelfont={bf,sf}, textfont={sf},margin=0em]{caption}
\usepackage{hyperref}
\hypersetup{colorlinks,
filecolor=black,
linkcolor=burntorange,
citecolor=burntorange,
urlcolor=badgerred,
bookmarksopen=true}
\hypersetup{bookmarksopenlevel=1}

\theoremstyle{plain}
\newtheorem{maintheorem}{Theorem}

\newtheorem{theorem}{Theorem}[subsection]

\newtheorem{lemma}[theorem]{Lemma}

\newcommand{\equalradii}{{\mc R}}
\newcommand{\solspace}{{\mathbb X}}
\renewcommand{\Re}{\mathop{\mathrm{Re}}}
\renewcommand{\Im}{\mathop{\mathrm{Im}}}
\newcommand{\C}{{\mb C}}
\newcommand{\N}{{\mb N}}
\newcommand{\R}{{\mb R}}

\newcommand{\mb}{\mathbb}
\newcommand{\mc}{\mathcal}
\newcommand{\mr}{\mathrm}
\newcommand{\mf}{\mathfrak}
\newcommand{\bs}{\boldsymbol}
\newcommand{\pd}{\partial}
\newcommand{\pdd}[1]{\frac{\partial}{\partial#1}}
\newcommand{\dd}{{\mr d}}
\newcommand{\ds}{\mr{d}s}
\newcommand{\pt}{\partial_t}
\newcommand{\pdt}{{\partial_{\mathfrak t}}}  
\newcommand{\vp}{\varphi}

\newcommand{\vV}{\bs V}  
\newcommand{\vE}{\bs E}  

\newcommand{\tsum}{\textstyle\sum}
\newcommand{\gf}{{\mathfrak g \mathfrak f}}
\newcommand{\rfc}{{\mathfrak r\mathfrak f\mathfrak c}}
\newcommand{\rfes}{{\mathfrak r\mathfrak f\mathfrak e\mathfrak s}}
\newcommand{\bgk}{{\mathfrak B}}
\newcommand{\wug}{{W^{u}(\gf)}}
\newcommand{\wsr}{{W^{s}(\rfc)}}
\newcommand{\wur}{{W^{u}(\rfc)}}

\newcommand{\fp}{{\mathfrak p}}
\newcommand{\fq}{{\mathfrak q}}

\newcommand{\fw}{{\mathfrak w}}
\newcommand{\cC}{{\mathcal C}}

\newcommand{\cM}{{\mathcal M}}
\newcommand{\cO}{{\mathcal O}}
\newcommand{\cQ}{{\mathcal Q}}
\newcommand{\cR}{{\mathcal R}}
\newcommand{\cU}{{\mathcal U}}

\DeclareMathOperator{\rc}{Rc}

\DeclareMathOperator{\grad}{grad}

\newcommand{\mat}{\begin{bmatrix}}
\newcommand{\rix}{\end{bmatrix}}
\newcommand{\tmat}{\left[\begin{smallmatrix}}
\newcommand{\trix}{\end{smallmatrix}\right]}
\newcommand{\isdef}{\, \raisebox{0.3pt}{:}\!\!= }
\newcommand{\transv}{\mathrel{\text{\tpitchfork}}}
\makeatletter
\newcommand{\tpitchfork} {%
  \vbox{
    \baselineskip\z@skip
    \lineskip-.52ex
    \lineskiplimit\maxdimen
    \m@th
    \ialign{##\crcr\hidewidth\smash{$-$}\hidewidth\crcr$\pitchfork$\crcr}
  }%
}
\makeatother

\DeclareMathOperator{\Rc}{Rc}

\begin{document}

\title{Ricci Solitons, Conical Singularities, and Nonuniqueness}

\author{Sigurd B.~Angenent} \address[Sigurd Angenent]{University of Wisconsin-Madison} \email{angenent@math.wisc.edu} \urladdr{http://www.math.wisc.edu/\symbol{126}angenent/}

\author{Dan Knopf} \address[Dan Knopf]{University of Texas} \email{danknopf@math.utexas.edu} \urladdr{http://www.ma.utexas.edu/users/danknopf}

\begin{abstract}
  In dimension $n=3$, there is a complete theory of weak solutions of Ricci flow ---
  the \emph{singular Ricci flows} introduced by Kleiner and Lott \cite{KL17, KL18}
  --- which Bamler and Kleiner~\cite{BK17} proved are unique across singularities.

  In this paper, we show that uniqueness should not be expected to hold for Ricci
  flow weak solutions in dimensions $n\geq5$.  Specifically, for any integers
  \(p_1,p_2\geq 2\) with \(p_1+p_2\leq 8\), and any $K\in\N$, we construct a complete
  shrinking soliton metric $g_K$ on $\mc S^{p_1}\times\R^{p_2+1}$ whose forward
  evolution \(g_K(t)\) by Ricci flow starting at $t=-1$ forms a singularity at time
  $t=0$.  As \(t\nearrow 0\), the metric \(g_K(t)\) converges to a conical metric on
  $\mc S^{p_1}\times\mc S^{p_2}\times(0,\infty)$.  Moreover there exist at least $K$
  distinct, non-isometric, forward continuations by Ricci flow expanding solitons on
  $\mc S^{p_1}\times\R^{p_2+1}$, and also at least $K$ non-isometric, forward
  continuations expanding solitons on $\R^{p_1+1}\times \mc S^{p_2}$.

  In short, there exist smooth complete initial metrics for Ricci flow whose forward
  evolutions after a first singularity forms are not unique, and whose topology may
  change at the singularity for some solutions but not for others.
\end{abstract}

\maketitle
\tableofcontents

\setlength{\parskip}{1ex plus 1ex}

\section{Introduction}
In this paper we analyze shrinking and expanding soliton solutions of Ricci flow and
in particular construct solutions that are defined for \(t<0\), form a singularity at
time \(t=0\), and then allow more than one smooth continuation for \(t>0\).  To state
the result more precisely, we use the notion of a \emph{Ricci Flow spacetime}
introduced by Kleiner and Lott~\cite{KL17}.  Their definition is as follows.

\noindent\textbf{Definition. } A \emph{Ricci flow spacetime} is a tuple
$(\mc M,\mf t,\pt,g)$ such that:\footnote{To obtain the full strength of their
  results, Kleiner and Lott impose three other conditions (Definition~1.6) on the
  singular Ricci flows they consider involving $3$-manifolds and PIC
  $4$-manifolds. We omit those conditions because they are not relevant to our work
  on higher-dimensional manifolds in this paper.}
\begin{itemize}

\item $\mc M$ is a smooth manifold-with-boundary;

\item $\mf t$, the \emph{time function,} is a submersion $\mf t : \mc M \to I\subseteq \mb R$, where $I$ is regarded as a time interval; 

\item the boundary of $\mc M$, if nonempty, corresponds to the endpoint(s) of the interval, that is, $\partial\mc M = \mf t^{-1}(\partial I)$;

\item $\pt$ is the \emph{time vector field,} which satisfies $\pt\mf t\equiv 1$; and

\item $g$ is a smooth inner product on the subbundle $\ker(\mr d\mf t)\subset T\mc M$ which defines a Ricci flow,
  \[
    \mc L_{\partial_t} [g] =-2\Rc[g].
  \]
\end{itemize}
Slightly modifying the notation of~\cite{KL17}, we denote the spacelike timeslice at time \(a\) by $\mc M_a = \mf t^{-1}(a)$, and similarly define $\mc M_{<a}=\mf t^{-1}(-\infty,a)$ and $\mc M_{>a}=\mf t^{-1}(a,\infty)$.

\noindent\textbf{Main Theorem. }\itshape%
Let \(p_1,p_2\geq 2\) be integers with \(p_1+p_2\leq 8\), and let \(n=p_1+p_2\).

For any integer \(K\in\N\), there exist Ricci flow spacetimes \(\mc M^1\), \dots, \(\mc M^K\) and $\tilde{\mc M}^1$, \dots, $\tilde{\mc M}^K$ with time functions \(\mf t:\mc M^k\to\R\), $\tilde{\mf t}:\tilde{\mc M}^k\to\R$, and evolving metrics \(g_k\) and $\tilde g_k$, respectively, such that for
\begin{itemize}
\item[\(\mf t<0\)~:] the sets $\mc M^k_{<0}$ and $\tilde{\mc M}_k$ and their metrics \(g_k|_{\mc M^k_{<0}}\) and \(\tilde g_k|_{\tilde{\mc M}^k_{<0}}\) all coincide and are given by a single shrinking gradient soliton --- all timeslices \(\mc M^k_t\) and $\tilde{\mc M}^k_t$ with \(t<0\) are diffeomorphic to \(\mc S^{p_1}\times\R^{p_2+1}\);

\item[\(\mf t=0\)~:] each timeslice $\mc M^k_0$ and $\tilde{\mc M}^k_0$ is incomplete and isometric to the same cone metric on \((0,\infty)\times\mc S^{p_1}\times \mc S^{p_2}\) with its conical singularity removed; and

\item[\(\mf t>0\)~:] each $\mc M^k_{>0}$ and $\tilde{\mc M}^k_{>0}$ is isometric to the flow of a \underline{distinct} expanding $n+1$-dimensional soliton --- that is, the spacetimes $\mc M^j_{>0}$ and $\tilde{\mc M}^k_{>0}$ are isometric if and only if $j=k$.

\end{itemize}

These Ricci flow spacetimes are maximal with respect to the partial ordering induced by inclusion:
\((\mc M, \mf t, \partial_t, g) \preceq (\mc M', \mf t', \partial_t', g')\) iff \(\mc M\subset\mc M'\), \(\mf t = \mf t'|_{\mc M}\),
\(\partial_t = \partial_t'|_{\mc M}\), and \(g=g'|_{\mc M}\).
There is only one singular time, \(\mf t=0\), \emph{i.e.,} \(\mc M^k_t\) and \(\tilde{\mc M}^k_t\) are complete for all
\(t\neq 0\).\footnote{ Completeness here is understood with respect to Kleiner and Lott's \emph{spacetime metric}
$g_{\mc M} = \hat g +(\mr d\mf t)^2$, where $\hat g$ is the extension of $g$ to a quadratic form on $T\mc M$ satisfying $\pt\in\ker(\hat g)$.}

For $t>0$, the timeslices \(\mc M_t^k\) are diffeomorphic to \(\R^{p_1+1}\times\mc S^{p_2}\), while the timeslices $\tilde{\mc M}^k_{t}$ are diffeomorphic to \(\mc S^{p_1}\times\R^{p_2+1}\).

\upshape

Put succinctly, in dimensions $n+1\geq5$, Ricci flow spacetimes may not be unique after their initial singularities: a given initial metric \(\mc M^k_{-1}\) may admit several distinct Ricci flow spacetimes, some of which change their topology while others do not.

The Main Theorem follows directly from Theorems \ref{main-A}, \ref{main-B}, and \ref{main-C} as stated below.

In the remainder of this introduction, we provide further exposition of these ideas.

\section{Background}

It is well known that Ricci flow solutions $\big(\mc M^n,g(t)\big)$ typically develop local singularities in finite time, after which the flow cannot be continued by classical means.  To deal with this phenomenon, Hamilton~\cite{Ham97} introduced and Perelman~\cite{Per02, Per03} further developed \emph{Ricci flows with surgery.} As implemented by Perelman, these depend on a fixed positive constant $\epsilon\ll1$ and three decreasing positive functions of time: a surgery parameter $\delta(t)$, a canonical neighborhood scale $r(t)$, and a non-collapsing parameter $\varkappa(t)$.  While Perelman's construction was brilliantly successful, it suffered from two limitations that he himself noted: the surgeries violate the \textsc{pde} where they occur, and they depend on arbitrary choices, hence are not canonical.  \emph{A priori,} the forward evolution of a solution after a surgery is not known to be independent of those choices.  In drawing attention to these issues, Perelman wrote: ``It is likely that by passing to the limit in this construction one would get a canonically defined Ricci flow through singularities, but at the moment I don't have a proof of that''~\cite{Per02}; and ``Our approach \dots is aimed at eventually constructing a canonical Ricci flow, defined on a largest possible subset of space-time --- a goal that has not been achieved yet in the present work''~\cite{Per03}.

In the intervening years, there have been a few rigorous examples of Ricci flow singularity recovery without intervening surgeries.  For $t<0$, the noncompact K\"ahler ``blowdown soliton'' discovered by Feldman, Ilmanen, and one of the authors is a shrinking gradient soliton with the topology of $\mb C^N$ blown up at the origin; as $t\nearrow0$, it converges to a cone on $\mb C^N\setminus\{0\}$; and for $t>0$, it desingularizes into an expanding gradient soliton discovered by Cao~\cite{Cao97} on complete $\mb C^N$; see \cite{FIK03} for details.  A \textsc{pde} regularization scheme, closer in spirit to what Perelman suggested, was used to recover from nondegenerate neck pinches by Caputo and the authors of this paper~\cite{ACK12}.  Similar techniques were employed by Carson to recover from degenerate neck pinches~\cite{Car16} as well as flows from more general singular initial metrics that need not be warped products globally \cite{Car18}.  There has also been significant progress on flowing through singularities in the K\"ahler setting (where the flow reduces to a strictly parabolic equation for a scalar function); see, \emph{e.g.,} Song--Tian~\cite{ST17} and Eyssidieux--Guedj--Zeriahi~\cite{EGZ16}.  All these examples can be thought of heuristically as ``weak'' or ``generalized'' solutions of Ricci flow.  We note that several other authors have studied existence, uniqueness, and regularity of Ricci flow solutions originating from non-smooth initial data.  In chronological order, see \cite{Sim02}, \cite{Sim09}, \cite{Top10}, \cite{CTZ11}, \cite{GT11}, \cite{KL12}, \cite{Sim12}, \cite{Top12}, \cite{GT13}, and \cite{Top15}.

For dimension $n=2$, where Ricci flow is conformal, the \emph{instantaneously complete Ricci flows} studied by Topping and collaborators provide a well-posed solution of the Ricci flow initial value problem starting from a completely general initial surface --- see~\cite{Top15} and the references therein.

For dimension $n=3$, a complete theory of Ricci flow weak solutions has now been developed.  Kleiner and Lott define and analyze \emph{singular Ricci flows}, constructed by regularization of compact $3$-dimensional (and $4$-dimensional PIC) solutions \cite{KL17, KL18}.  Bamler and Kleiner subsequently prove that these \emph{singular Ricci flows} are unique for $n=3$~\cite{BK17}.  Together, these results elegantly realize Perelman's hope for a canonically-defined solution of the Ricci flow initial value problem.  \smallskip

In higher dimensions, far less is currently known.  ``Super Ricci flows'' have been studied by McCann--Topping~\cite{MT10} and Sturm~\cite{Stu17} using techniques from optimal transport.  An alternate approach to constructing weak solutions, using stochastic analysis, has been pioneered by Naber--Haslhofer \cite{HN18}.

Our main result in this paper demonstrates that, whichever definition(s) of Ricci flow ``weak solutions'' emerge(s), one should not expect uniqueness to hold in dimensions five and above.  (Dimension four remains, as is so often the case, a mystery.)  As stated above in our Main Theorem, we construct families of asymptotically conical gradient solitons that model the formation of and recovery from finite-time singularities that admit non-unique forward continuations, both geometrically and topologically.

In our construction, we consider cohomogeneity-one metrics on the manifold $\R_+\times \mc S^{p_1} \times \mc S^{p_2}$ having the form of a doubly-warped product,
\begin{equation}
  \label{eq-MetricAnsatz}
  g = (\ds)^2 + \frac{(p_1-1)s^2}{x_1(s)}\,g_{\mc S^{p_1}} + \frac{(p_2-1)s^2}{x_2(s)}\,g_{\mc S^{p_2}},
\end{equation}
where $s\in\R_+$, $4\leq p_1+p_2\leq8$, and where \(g_{\mc S^{p_\alpha}}\) is the round metric on the \(p_\alpha\)-dimensional unit sphere \(\mc S^{p_\alpha}\).  In parts of this paper, it is convenient to set $s=e^\tau$ and write the metric in the form
\begin{equation}
  \label{eq-MetricAnsatz-t}
  g = e^{2\tau}\left\{
    (\dd \tau)^2 +
    \frac{p_1-1}{x_1(\tau)} g_{\mc S^{p_1}} +
    \frac{p_2-2}{x_2(\tau)} g_{\mc S^{p_2}}
  \right\},
\end{equation}
regarding \(x_\alpha:\R\to\R_+\) as functions of \(\tau\).  We are interested in smooth metrics that compactify as $s\searrow0$ and that are complete as $s\nearrow\infty$, giving the topology of $\R^{p_1+1}\times\mc S^{p_2}$.
%

Metrics of the form \eqref{eq-MetricAnsatz} extend to complete metrics on \(\R^{p_1+1} \times \mc S^{p_2}\) if \(x_1\) and \(x_2\) satisfy
\begin{equation}
  \label{eq-gf-boundaryconditions}
  x_1(s) = p_1-1 + o(1)\qquad\mbox{ and }\qquad
  x_2(s) = Cs^2 + o(s^2)
\end{equation}
as \(s\to 0\), for some constant \(C>0\).

Metrics of the form \eqref{eq-MetricAnsatz} are asymptotically conical as \(s\to\infty\) if the limits
\[
  \bar x_\alpha = \lim_{s\to\infty} x_\alpha(s),\qquad \alpha\in\{1,2\},
\]
exist and are positive.  We call the constants \(\bar x_\alpha\) the \emph{asymptotic apertures} of the metric \(g\).

In the very special case in which
\begin{equation}
  \label{eq-rfc}
  x_1(s) = x_2(s) = n-1
\end{equation}
for all \(s>0\), the metric \eqref{eq-MetricAnsatz} is that of the unique \emph{Ricci flat cone} of the form \eqref{eq-MetricAnsatz}.  This cone metric is singular at \(s=0\): as we show in Appendix~\ref{DWP-geometries}, the norm of its Riemann tensor is unbounded as \(s\searrow 0\).

In this paper, we seek complete metrics having the structure~\eqref{eq-MetricAnsatz} that satisfy the Ricci soliton condition.
We write this in the form
\begin{subequations}
  \begin{equation}
    \label{eq-Ricci-Soliton-pde}
    -2\rc[g]  = 2\lambda g + \mc L_{\mf X}(g),
  \end{equation}
  where $\lambda\in\{-1,0,+1\}$ is the dilation rate that corresponds to shrinking, steady, and expanding solitons, respectively, and
  \[
    \mf X = f(s) \frac\pd {\pd s}
  \]
  is the soliton vector field, \emph{i.e.,} the vector field on $\R^{p_1+1}\times\mc S^{p_2}$ that generates the diffeomorphisms by which the soliton evolves under Ricci flow. If $\mf X = \grad F$ for some potential function $F$ (which is the case here), then~\eqref{eq-Ricci-Soliton-pde} is equivalent to
  \begin{equation}
    \Rc[g]+\lambda g + \nabla^2 F=0.
  \end{equation}
\end{subequations}
We show in \S~\ref{more_equations} below that the soliton equation~\eqref{eq-Ricci-Soliton-pde} applied to the \emph{Ansatz}~\eqref{eq-MetricAnsatz} gives rise to a system of \textsc{ode} on $\R^6$.  Generalizations of this system were investigated analytically and numerically by Dancer--Hall--Wang \cite{DancerHallWang}.

The Ricci-flat cone~\eqref{eq-rfc} is a stationary soliton whose soliton vector field is \(\mf X = 0\). But a consequence of system~\eqref{eq-LittleSolitonSystem} below is that any Ricci-flat metric may also be regarded as an expanding or shrinking soliton if one chooses the soliton vector field to be \(\mf X = -\lambda s\frac{\pd}{\pd s}\).

We are particularly interested in complementing pairs of shrinking solitons \((G^-, \mf X^-)\) and expanding solitons \((G^+, \mf X^+)\) that are asymptotic to the same conical metric,
\begin{equation} \label{eq-cone-metric}
  \bar G = (\mr ds)^2 + \frac{(p_1-1)s^2}{\bar x_1} g_{\mc S^{p_1}} + \frac{(p_2-1)s^2}{\bar x_2} g_{\mc S^{p_2}},
\end{equation}
with apertures \(\bar x_1\), \(\bar x_2\).  Given such a pair of solitons, we define the family of metrics
\[
  g(t) \isdef
  \begin{cases}
    -2t\bigl(\phi_{\mf X^-}^t\bigr)^* G^- & (t<0), \\
    \bar G  & (t=0),\\
    2t\,\bigl(\phi_{\mf X^+}^t\bigr)^* G^+ & (t<0),
  \end{cases}
\]
where \(\phi_{\mf X^\pm}^t\) denotes the flow generated by the vector field \(\mf X^\pm\).

In Section~\ref{app-diffeo}, we show that these choices $g(t)$ glue together by isometries to give a smooth Ricci flow spacetime.
For \(t<0\), the metric \(g(t)\) is the unique (see below) smooth shrinking soliton diffeomorphic to $\mb R^{p_1+1}\times\mc S^{p_2}$
that converges to its asymptotic cone as \(t\nearrow 0\).
For \(t>0\), the solution continues as a smooth expanding soliton diffeomorphic to either $\mb R^{p_1+1}\times\mc S^{p_2}$
or $\mc S^{p_1}\times\mb R^{p_2+1}$ with the same singular conical metric as initial data.
(This construction may be compared to that of Theorem~1.6 in~\cite{FIK03}.)

A theorem of Kotschwar~and~Wang~\cite{KW15} implies that there can be at most one shrinking soliton that is asymptotic to a given cone \(\bar G\).  There can however be many expanding solitons \(G^+\) asymptotic to any given cone \(\bar G\).  In fact, one of our results is that the Ricci-flat cone (see below) admits infinitely many smooth expanding solitons as forward evolutions. Moreover, for cones very close to the Ricci-flat cone, the number of expanding solitons can be arbitrarily large.  More precisely, we have the following existence result.

\begin{maintheorem}\label{main-A}
  Assume the dimensions \(p_\alpha\) satisfy \(p_1, p_2\geq 2\) and \(p_1+p_2\leq 8\).  Then there exists a two-parameter family of expanding solitons, \((G^+,\mathfrak X^+)(j,T)\), where the parameters \((j,T)\) take values in \([-\iota, +\iota]\times [T_0,\infty)\) for certain \(\iota>0\), \(T_0<\infty\), with the following properties.

  The expanding solitons \((G^+,\mf X^+)(j,T)\) as well as their asymptotic apertures \(x_1^+(j,T)\) and \(x_2^+(j,T)\) are real analytic functions of \(j,T\).  For any \(k\in\N\), there exists a neighborhood \(\mc U_k\subset\R^2\) of the point \((n-1, n-1)\in\R^2\) such that for each \((\bar x_1, \bar x_2) \in \mc U_k\), there exist at least \(k\) distinct expanding solitons \((G^+,\mathfrak X^+)(j_i, T_i)\), (\(i=1, \dots, k\)), with
  \[
    x_1^+(j_i, T_i) = \bar x_1, \quad x_2^+(j_i, T_i)=\bar x_2 \qquad (i=1, \dots k).
  \]

\end{maintheorem}

The metrics we find can be thought of as the result of gluing two simpler soliton metrics, each of which appears as a one-parameter family of solitons, the parameters being \(j\) for one family and \(T\) for the other.

To describe the first family, which is parameterized by $j$, we recall that the singular cone \(x_1=x_2=n-1\), whose metric we denote by
\[
  g_{\rfc} = (\mr ds)^2 + \frac{p_1-1}{n-1} s^2 g_{\mc S^{p_1}} + \frac{p_2-1}{n-1} s^2 g_{\mc S^{p_2}},
\]
is Ricci-flat.  As follows from system~\eqref{eq-LittleSolitonSystem} below, the choices of soliton vector field
\[
  \mf X = \lambda s \frac{\pd}{\pd s}
\]
make it into a shrinking or expanding soliton, respectively.  In the case of expanding solitons, it turns out that there is a one-parameter family of metrics \(g^j\) (with \(j\in[-\iota,\iota]\) for some small \(\iota>0\)) of the form
\[
  g^j = (\mr ds)^2 + \frac{p_1-1}{x_1(j;s)} s^2 g_{\mc S^{p_1}} + \frac{p_2-1}{x_2(j;s)} s^2 g_{\mc S^{p_2}},
\]
where \(x_\alpha(j;s)\) are real analytic functions of \(s^2\) with
\[
  x_1(j;0)=x_2(j;0)=n-1.
\]
Thus the metrics \(g^j\) are singular at \(s=0\), and the singularity is asymptotically like the Ricci-flat cone metric $g_{\rfc}$.  These metrics appear as solutions in the unstable manifold \(\wur\) of a fixed point $\rfc$ in the ODE system~\eqref{eq-BlowUpSystem-Polynomial} that we study.

To describe the second family of solutions, which is parameterized by $T$, we recall the Ricci-flat Einstein metric found by B\"ohm \cite{Bohm}, which is of the form
\[
  g_\bgk = (\mr ds)^2 + (p_1-1) \frac{s^2}{x_1^\bgk(s)} g_{\mc S^{p_1}} + (p_2-1) \frac{s^2}{x_2^\bgk(s)} g_{\mc S^{p_2}},
\]
where \(x_\alpha^{\bgk}(s)\) are again real analytic functions of \(s^2\), this time with
\[
  x^\bgk_1(0) = p_1-1 \quad\text{and}\quad x^\bgk_2(s) = s^2 + \cO(s^4), \qquad (s\to0).
\]
This metric extends smoothly to a metric on \(\mb R^{p_1+1}\times\mc S^{p_2}\).  As \(s\to\infty\), the metric grows asymptotically like a paraboloid, in that \(x_\alpha(s) = As+o(s)\) for \(s\to\infty\), and thus
\[
  g_{\bgk} \sim (\mr ds)^2 + s \left\{\frac{p_1-1}{A}g_{\mc S^{p_1}} + \frac{p_2-1}{A}g_{\mc S^{p_2}}\right\}, \qquad (s\to\infty).
\]
Any multiple of a Ricci-flat metric is again Ricci-flat, so we have a one-parameter family of Ricci-flat metrics given by \(e^{-2T}g_\bgk\), with \(T\in\R\).

Heuristically, to produce the metric \(G^+(j,T)\), one removes a neighborhood of size \(\mc O(e^{-T})\) of the singular point in the expanding soliton metric \(g^j\) on \((0, \infty)\times\mc S^{p_1}\times\mc S^{p_2}\) and replaces it by a piece of the same size of the rescaled B\"ohm metric \(e^{-2T}g_\bgk\).  We execute this gluing rigorously by analyzing the \textsc{ode} system that describes Ricci solitons.

More precisely, the B\"ohm metric \(g_\bgk\) and the singular metrics \(g^j\) appear as complete orbits of the \textsc{ode} system, which meet at a hyperbolic fixed point that represents the Ricci-flat cone metric~\(g_{\rfc}\).  To glue the two families of orbits, we use techniques from dynamical systems, notably Palis's \(\lambda\)-lemma.
Moreover, we find when linearizing the \textsc{ode} system near the $\rfc$ fixed point, at which the B\"ohm and \(g^j\) metrics are to be glued, that the differences $x_1-x_2$ and $s\frac{\dd}{\dd s}(x_1-x_2)$ decouple into a subsystem whose eigenvalues are complex when $n=p_1+p_2$ satisfies $n\in\{4,5,6,7,8\}$.  Existence of this oscillatory subsystem was earlier observed by Dancer--Hall--Wang in their analysis of winding numbers of cohomogeneity-one shrinking solitons \cite{DancerHallWang}.  These complex eigenvalues are responsible for the oscillatory dependence on the parameter \(T\) of the asymptotic apertures of the expanding solitons we construct, and thus are the main source of the nonuniqueness of smooth continuation by Ricci flow of the cone metrics that we find in this paper.  This nonuniqueness is analogous to the phenomenon of ``fattening'' for Mean Curvature Flow~\cite[Lecture~4]{Trieste}, \cite{fattening}.

For shrinking solitons, we have the following companion Theorem.

\begin{maintheorem}\label{main-B}
  Assume again that the dimensions \(p_\alpha\) satisfy \(p_1, p_2\geq 2\) and \(p_1+p_2\leq 8\).  Then there exists a sequence of smooth shrinking soliton metrics \(\{(G_i^-, \mf X_i^-) \mid i\in\N\}\) on \(\R^{p_1}\times \mc S^{p_2}\) having the form~\eqref{eq-MetricAnsatz}, \emph{i.e.,}
  \[
    G_i^- = (\mr ds)^2 + \frac{(p_1-1)s^2}{x^-_{i,1}(s)} g_{\mc S^{p_1}} + \frac{(p_2-1)s^2}{x^-_{i,2}(s)} g_{\mc S^{p_2}}, \qquad \mf X^-_i = f^-_i(s)\frac{\pd}{\pd s},
  \]
  whose asymptotic apertures \((\bar x^-_{i,1}, \bar x^-_{i, 2})\) satisfy
  \[
    \lim_{i\to\infty} \bar x^-_{i,1} = \lim_{i\to\infty} \bar x^-_{i,2} = n-1.
  \]
\end{maintheorem}
To the best of our knowledge, this construction (at least for $p_1+p_2$ even) gives the only known examples of complete, nontrivial, gradient shrinking solitons that are neither K\"ahler nor products of compact Einstein spaces with a Gaussian soliton \cite{PW09}.

Combining Theorems A and B with the isometric gluing construction in Section~\ref{app-diffeo},
we conclude that there exists sequences of smooth shrinking solitons \(\bigl(G_i^-, \mf X_i^-\bigr)\) for which the corresponding ancient solutions
\(-2t\bigl(\varphi_{\mf X_i^-}^t\bigr)^*G_i^-\) of Ricci flow each form a conical singularity at \(t=0\) that admits at least \(2k(i)\) distinct forward evolutions by expanding solitons.  More precisely, we have the following.

\begin{maintheorem}\label{main-C}
  Let \((G_i^-, \mf X_i^-)\) be a family of shrinking solitons constructed in Theorem~B on $\mb R^{p_1+1}\times\mc S^{p_2}$.
  Let \((G_i^+, \mf X_i^+)\) be a family of expanding solitons constructed in Theorem~A on either $\mb R^{p_1+1}\times\mc S^{p_2}$
  or $\mc S^{p_1}\times\mb R^{p_2+1}$ with the same asymptotic apertures $(\bar x_{i,1},\bar x_{i,2})$.
 
 These glue together by isometries to form maximal Ricci flow spacetimes. And the number \(k(i)\) of distinct pairs (of either topology) of this type becomes unbounded as \(i\to\infty\),
 \emph{i.e.,} the number of pairs \((j,T)\in [-\iota, \iota]\times[T_0, \infty)\) such that \(x_1^+(j,T) = \bar x_{i,1}\) and
 \(x_2^+(j,T) = \bar x_{i,2}\) satisfies
  \[
    \lim _{i\to\infty} k(i) = \infty.
  \]
\end{maintheorem}

The remainder of this paper provides additional analytic details, including the
solitons' asymptotic rates of convergence to the Ricci-flat cone at spatial infinity.

\smallskip\noindent\textbf{Remark. }  If $k$ denotes the number of zeroes of $x_1-x_2$, one
is led to wonder whether there exists a unique shrinker for each $k\in\N$.  Answering
this would require a global analysis of the \textsc{ode} system; our methods, which
are local near $\rfc\cup\rfes$, do not provide this information.

\subsubsection*{Acknowledgments}
DK thanks the NSF for support (DMS-1205270) during early work on this project.  Both authors thank the Mathematisches Forschungsinstitut Oberwolfach for its hospitality in the 2016 \emph{Geometrie} workshop, during which they made further progress on the project.

\section{Derivation of the soliton flow}
\label{more_equations} The requirement that a metric of the form \eqref{eq-MetricAnsatz} be a Ricci soliton is equivalent to a system of ordinary differential equations for the functions $(x_1(s), x_2(s), f(s))$, supplemented with boundary conditions that describe whether the metric closes up smoothly or is asymptotically conical at either end of the $s$ interval.  In this section, we rewrite this system of \textsc{ode} as an autonomous system on $\R^6$, which we call the \emph{soliton system}.

\subsection{Derivation of the second-order system}

Curvatures of doubly-warped product metrics are well studied. For brevity, we follow~\cite{Pet16} and merely outline the derivation in Appendix~\ref{DWP-geometries}, referring the reader to that text for further details.

So we begin by considering doubly-warped-product manifolds $\mb R_+\times\mc S^{p_1}\times\mc S^{p_2}$ with metrics
\begin{equation} \label{eq-Petersen} g = (\ds)^2 + \vp_1^2(s)\,g_{\mc S^{p_1}} + \vp_2^2(s)\,g_{\mc S^{p_2}},
\end{equation}
where $g_{\mc S^{p_1}}$ and $g_{\mc S^{p_2}}$ have constant unit sectional curvatures. Following the derivation in
Appendix~\ref{DWP-geometries}, we see that the Ricci tensor of $g$ is
\begin{align*}
  \Rc=
  &-\left\{p_1\frac{(\vp_1)_{ss}}{\vp_1}+p_2\frac{(\vp_2)_{ss}}{\vp_2}\right\}(\ds)^2\\
  &+ \left\{-\frac{(\vp_1)_{ss}}{\vp_1}+(p_1-1)\frac{1-(\vp_1)_s^2}{\vp_1^2}
    -p_2\frac{(\vp_1)_s(\vp_2)_s}{\vp_1\vp_2}\right\}\vp_1^2\,g_{\mc S^p}\\
  &+\left\{-\frac{(\vp_2)_{ss}}{\vp_2}+(p_2-1)\frac{1-(\vp_2)_s^2}{\vp_2^2}-p_1\frac{(\vp_1)_s(\vp_2)_s}{\vp_1\vp_2}\right\}\vp_2^2\,g_{\mc S^{p_2}},
\end{align*}
and the Lie derivative of $g$ with respect to a gradient vector field $\grad F=\mathfrak X=f(s)\frac{\partial}{\partial s}$ of a potential function $F(s)$ is
\[
  \mc L_{\mf X} g = 2f_s(\mr d s)^2+2f\frac{(\vp_1)_s}{\vp_1}\,\vp_1^2g_{\mc S^{p_1}} +2f\frac{(\vp_2)_s}{\vp_2}\,\vp_2^2g_{\mc S^{p_2}}.
\]
It follows that equation~\eqref{eq-Ricci-Soliton-pde} is equivalent to the second-order system of three differential equations
\begin{subequations} \label{eq-LittleSolitonSystem}
  \begin{align}
    f_s &= p_1\frac{(\vp_1)_{ss}}{\vp_1}+p_2\frac{(\vp_2)_{ss}}{\vp_2}-\lambda,\\
    \frac{(\vp_1)_{ss}}{\vp_1} &= (p_1-1)\frac{1-(\vp_1)_s^2}{\vp_1^2} -
                                 p_2\frac{(\vp_1)_s(\vp_2)_s}{\vp_1\vp_2}
                                 + \frac{(\vp_1)_s}{\vp_1}f+\lambda,\\
    \frac{(\vp_2)_{ss}}{\vp_2} &=(p_2-1)\frac{1-(\vp_2)_s^2}{\vp_2^2} -
                                 p_1\frac{(\vp_1)_s(\vp_2)_s}{\vp_1\vp_2} + \frac{(\vp_2)_s}{\vp_2}f +\lambda.
  \end{align}
\end{subequations}

We note the curious fact that system~\eqref{eq-LittleSolitonSystem} is equivalent to a mechanical system on $\mb R^3$.  We outline this correspondence in Appendix~\ref{MechanicalSystem}.

\subsection{Reduction to a system of first-order equations}
The differential equations~\eqref{eq-LittleSolitonSystem} have the disadvantage that the different possible boundary conditions at \(s=0\) lead to singularities.  If at \(s=0\) the metric should extend to a complete smooth metric on \(\mc D^{p_1+1}\times\mc S^{p_2}\), then one of the functions \(\varphi_\alpha\) must vanish at \(s=0\), which leads to a division by zero in the corresponding equation for \(\varphi_\alpha\) in \eqref{eq-LittleSolitonSystem}.  It turns out that the choice of variables below leads to equations that are equivalent to~\eqref{eq-LittleSolitonSystem} for \(s>0\), and that also capture all the possible boundary conditions at \(s=0\) in the form of hyperbolic fixed points of the corresponding flow.  We arrive at the new variables by writing the metric in the form~\eqref{eq-MetricAnsatz} rather than~\eqref{eq-Petersen}.  The variables \(x_\alpha\) and \(\varphi_\alpha\) are related by
\[
  x_\alpha=(p_\alpha-1)\frac{s^2}{\vp_\alpha^2},\qquad\alpha\in\{1,2\}.
\]
Making these substitutions in~\eqref{eq-LittleSolitonSystem} leads via further straightforward calculations to a second-order system for $x_\alpha$ and $f$.

To get a system of first-order equations, we introduce variables \(y_\alpha\) related to \(x_\alpha\) via
\begin{equation} \label{eq-define-yalpha} y_\alpha =-\frac{s}{2x_\alpha}\frac{\dd x_\alpha}{\dd s} = -1 +s\frac{(\varphi_\alpha)_s}{\varphi_\alpha} \,.
\end{equation}
Furthermore, we replace the function $f$ in the soliton vector field $\mf X$ by the new quantity
\begin{equation}\label{eq-define-Gamma}
  \Gamma = s f(s) + \lambda s^2 - \sum_{\alpha=1,2} p_\alpha \bigl(1+y_\alpha\bigr).
\end{equation}

In those parts of the soliton where \(s\) is small, it is advantageous to consider the quantity
\[
  \sigma = s^2
\]
rather than the distance \(s\) itself.  It is related to \(\tau\) via \(\sigma=e^{2\tau}\).

Putting these substitutions into~\eqref{eq-LittleSolitonSystem}, one finds after diligent computations that a metric of the form~\eqref{eq-MetricAnsatz} is a shrinking $(\lambda<0)$, steady $(\lambda=0)$, or expanding $(\lambda>0)$ gradient Ricci soliton if and only if $(x_1, x_2, y_1, y_2, \Gamma, \sigma)$ satisfy
\begin{subequations}
  \label{eq-BlowUpSystem-Polynomial}
  \begin{align}
    x_\alpha' & = -2x_\alpha y_\alpha \label{eq-Pa}, \\
    y_\alpha' & = x_\alpha + \bigl(\Gamma+1 - \lambda\sigma\bigr)y_\alpha + \Gamma + 1, \label{eq-Pb} \\
    \Gamma'   & =\Gamma+ \tsum_\alpha p_\alpha (1+y_\alpha)^2, \label{eq-PG}  \\
    \sigma'   & = 2\sigma.  \label{eq-Pd}
  \end{align}
\end{subequations}
We call this system the \emph{soliton system}.  It is the main system of differential equations that we study in the remainder of this paper.  Here the prime ${}'$ indicates differentiation with respect to $\tau$, \emph{i.e.,} for any differentiable quantity $\psi$,
\[
  \psi' = \frac{\dd \psi}{\dd \tau} = s \frac{\dd \psi}{\ds}.
\]
The system of \textsc{ode}~\eqref{eq-BlowUpSystem-Polynomial} defines a vector field $X$ on $\R^6$,
\begin{multline}\label{eq-soliton-vectorfield}
  X = \sum_\alpha \Bigl\{ -2x_\alpha y_\alpha \frac{\pd} {\pd x_\alpha} + \Bigl[x_\alpha + \bigl(\Gamma+1 - \lambda\sigma\bigr)y_\alpha + \Gamma + 1 \Bigr] \frac\pd{\pd y_\alpha}
  \Bigr\}\\
  + \Bigl\{\Gamma + \sum_\alpha p_\alpha (1+y_\alpha)^2 \Bigr\}\frac{\pd}{\pd\Gamma} + 2\sigma \frac{\pd} {\pd\sigma}.
\end{multline}
We denote the corresponding \emph{soliton flow} on $\R^6$ by $g^t:\R^6\to\R^6$.  It has the defining property that for any solution \((x_\alpha(\tau), y_\alpha(\tau), \Gamma(\tau), \sigma(\tau))\) of~\eqref{eq-BlowUpSystem-Polynomial} one has
\[
  (x_\alpha(\tau+t), y_\alpha(\tau+t), \Gamma(\tau+t), \sigma(\tau+t)) =g^t\bigl(x_\alpha(\tau), y_\alpha(\tau), \Gamma(\tau), \sigma(\tau)\bigr).
\]
If \(I\subset\R\) is an interval and \(E\subset\R^6\) is any subset, then we will find the following notation convenient:
\[
  g^I(\fp) = \{g^t\fp \mid t\in I\}\qquad \text{ and } \qquad g^I(E) = \{g^t\fp \mid t\in I, \fp\in E\}.
\]

The map \(g^t\) may not actually be defined on all of \(\R^6\).  Standard existence and uniqueness theorems for \textsc{ode} imply that for each \(\fp\in\R^6\), there exist \(T_-(\fp) <0< T_+(\fp)\) such that the solution \(g^t(\fp)\) of the soliton system~\eqref{eq-BlowUpSystem-Polynomial} is defined exactly on the interval \(T_-(\fp) < t < T_+(\fp)\).  The domain \(\{(\fp, t)\in \R^6\times\R \mid T_-(\fp) < t < T_+(\fp)\}\) is open in \(\R^7\), and the flow is a real analytic function from this domain to \(\R^6\).

The objects we seek are \(\R^3\)-valued functions \((x_1(s), x_2(s), \Gamma(s))\); these are in one-to-one correspondence with \(\R^5\)-valued functions \((x_1(s), x_2(s), y_1(s), y_2(s), \Gamma(s))\) satisfying \eqref{eq-define-yalpha}.  Instead of the unknown functions \((x_1,x_2,y_1,y_2, \Gamma) : (s_0, s_1) \to \R^5 \), we will frequently consider their graphs, which are differentiable curves in \(\R^6= \R\times\R^5\).  The condition that \((x_1, x_2, y_1, y_2, \Gamma)\) satisfy the first-order differential system is equivalent to the requirement that the graph of \((x_1,x_2,y_1,y_2, \Gamma)\) be an orbit of the flow of the vector field \(X\).

At times it will be more convenient to use $s$ instead of $\sigma$, especially when we let $\sigma\to\infty$.  In the region where $\sigma >0$, we may regard $(x_\alpha, y_\alpha, \Gamma, \sigma)$ and $(x_\alpha, y_\alpha, \Gamma, s)$ as equivalent sets of coordinates related by $\sigma=s^2$.  In the $s$ coordinate, the soliton flow is given by a nearly identical system, namely%
\[
  \begin{aligned}
    x_\alpha' & = -2x_\alpha y_\alpha, &
    \Gamma'   & =\Gamma+ \tsum_\alpha p_\alpha (1+y_\alpha)^2,                        \\
    y_\alpha' & = x_\alpha + \bigl(\Gamma+1 - \lambda s^2\bigr)y_\alpha + \Gamma + 1, & s' & = s.
  \end{aligned}
\]
In these coordinates, the vector field $X$ is given by
\begin{multline*}
  X = \sum_\alpha \Bigl\{ -2x_\alpha y_\alpha \frac{\pd} {\pd x_\alpha} + \Bigl(x_\alpha + \bigl(\Gamma+1 - \lambda s^2\bigr)y_\alpha + \Gamma + 1 \Bigr)
  \frac\pd{\pd y_\alpha} \Bigr\}\\
  + \Bigl\{\Gamma + {\sum_\alpha p_\alpha (1+y_\alpha)^2} \Bigr\}\frac{\pd}{\pd\Gamma} + s \frac{\pd} {\pd s}.
\end{multline*}

\subsection{Separation into averaged and difference variables}
It is useful (especially in the ``gluing region'' near the $\rfc$ fixed point) to consider the \emph{averaged variables}
\begin{equation}
  x = \sum_\alpha \frac{p_\alpha}{n} x_\alpha,
  \qquad
  y = \sum_\alpha \frac{p_\alpha}{n} y_\alpha,
\end{equation}
and the \emph{difference variables}
\[
  x_{12} = x_1-x_2, \qquad y_{12} = y_1-y_2.
\]
The averaged variables evolve by
\begin{subequations}\label{eq-average}
  \begin{align}
    x' &= -2 xy-2 \frac{p_1p_2}{n^2} x_{12}y_{12}, \\
    y' &= x + (\Gamma+1-\lambda s^2)y + \Gamma + 1,\\
    \Gamma' &= \Gamma + n (1+y)^2 + \frac{p_1p_2}{n} y_{12}^2,
  \end{align}
\end{subequations}
while the difference variables satisfy the \emph{difference} or \emph{oscillating system}
\begin{subequations}\label{eq-difference}
  \begin{align}
    x_{12}' &=-2(yx_{12}+xy_{12}) + 2\,\frac{p_1-p_2}{n}x_{12}y_{12},\\
    y_{12}' &= x_{12} + (\Gamma+1-\lambda s^2)y_{12}.
  \end{align}
\end{subequations}
This system of equations is equivalent to that found by Dancer, Hall, and Wang~\cite{DancerHallWang}.

\subsection{Coordinates near the Ricci-flat cone}
The Ricci-flat cone, given by \(x_1=x_2=n-1\), \(y_1=y_2=0\), and \(\Gamma=-n\) is both an expanding and a shrinking soliton, and plays a central role in our construction.  In the second half of this paper, we will mostly be analyzing orbits of the soliton flow near this special solution.  Because of this, it is convenient to consider new variables $\xi_\alpha$ and $\gamma$ defined by
\begin{equation}
  \label{eq-xi-y-gamma-def}
  x_\alpha = \xi_\alpha+n-1 \quad\mbox{ and }\quad \Gamma = -n+\gamma,
\end{equation}
respectively.  In these variables, the soliton flow is given by
\begin{equation}
  \label{eq-xi-y-gamma-system}
  \begin{aligned}
    \xi_\alpha' &= -2(n-1+\xi_\alpha)y_\alpha, &
    \gamma' &= \gamma + \textstyle{\sum_\alpha} p_\alpha (2y_\alpha+y_\alpha^2)\\
    y_\alpha ' &= \xi_\alpha + \bigl(\gamma-n+1-\lambda s^2\bigr)y_\alpha+\gamma, & s' &= s.
  \end{aligned}
\end{equation}

\section{Special invariant subsets for the soliton flow}
To explore solutions of the soliton flow, we begin by identifying a number of its invariant subsets.

The only solutions of the soliton system~\eqref{eq-BlowUpSystem-Polynomial} that correspond directly to metrics of the form~\eqref{eq-MetricAnsatz} are those for which \(\sigma, x_1, x_2\) all are positive.  Nevertheless the local flow \(g^t\) is defined on all \(\R^6\) and some of the solutions with \(\sigma=0\) or \(x_\alpha=0\) for some \(\alpha\) are still relevant to the problem of finding solitons.

\subsection{The hyperplane $\sigma=0$}

It follows directly from~\eqref{eq-Pd} that the hyperplane $\{\sigma\equiv0\} = \R^5\times\{0\}$ is an invariant subset for the soliton flow.  As we just observed, none of the solutions in this hyperplane correspond directly to soliton metrics.  Nonetheless, they do generate stationary soliton metrics in an indirect way.  Namely, if \((x_\alpha(\tau), y_\alpha(\tau), \Gamma(\tau), 0)\) is a solution of~\eqref{eq-BlowUpSystem-Polynomial} for some value of the parameter \(\lambda\), then for any constant \(\sigma_0>0\), one finds that
\begin{equation}
  \label{eq-stationary-soliton}
  \mf p(\tau) = \bigl(x_\alpha(\tau), y_\alpha(\tau), \Gamma(\tau), \sigma_0 e^{2\tau}\bigr)
\end{equation}
is a solution of~\eqref{eq-BlowUpSystem-Polynomial} with \(\lambda=0\).  So any solution of the soliton system~\eqref{eq-BlowUpSystem-Polynomial} with arbitrary \(\lambda\in\R\) that lies in the \(\sigma=0\) hyperplane generates solutions to the stationary soliton equations~\eqref{eq-BlowUpSystem-Polynomial} in which \(\lambda=0\).  The free parameter \(\sigma_0>0\) appears because the equation for stationary solitons is homogeneous: if \(g\) is a stationary soliton, then so is \(\sigma_0 g\) for any \(\sigma_0>0\).

Conversely, any solution of the stationary soliton equations (\emph{i.e.,}~\eqref{eq-BlowUpSystem-Polynomial} with \(\lambda=0\)) is of the form~\eqref{eq-stationary-soliton} for some \(\sigma_0>0\).  Given such a solution, the projection \((x_\alpha(\tau), y_\alpha(\tau), \Gamma(\tau), 0)\) onto the \(\sigma=0\) hyperplane is a solution of the soliton equations~\eqref{eq-BlowUpSystem-Polynomial} for any choice of \(\lambda\in\R\).

Consequently, complete solutions of the reduced system in the hyperplane $\{\sigma=0\}$ are in one-to-one correspondence with steady soliton metrics on $\R\times \mc S^{p_1}\times \mc S^{p_2}$.

\subsection{The region \(\sigma<0\)}
While it is clear from the definition \(\sigma=s^2\) that solutions to the soliton equations~\eqref{eq-BlowUpSystem-Polynomial} with \(\sigma<0\) do not correspond to metrics of the form~\eqref{eq-MetricAnsatz}, it is also trivially true that for any solution \((x_\alpha(\tau), y_\alpha(\tau), \Gamma(\tau), \sigma(\tau))\) of \eqref{eq-BlowUpSystem-Polynomial} with \(\sigma(\tau)<0\), the \(\R^6\)-valued function
\[
  \mf p(\tau) =\bigl(x_\alpha(\tau), y_\alpha(\tau), \Gamma(\tau), -\sigma(\tau)\bigr)
\]
satisfies~\eqref{eq-BlowUpSystem-Polynomial} with \(\lambda\) replaced by \(-\lambda\).

Thus instead of studying~\eqref{eq-BlowUpSystem-Polynomial} for the three different cases \textcolor{badgerred}{\(\lambda\in\{0, \pm 1\}\)}, one could in principle only consider the case \textcolor{badgerred}{\(\lambda=+1\)}.  Then expanding solitons correspond to solutions with \(\sigma>0\); shrinkers are solutions with \(\sigma<0\); and stationary solitons correspond to solutions with \(\sigma=0\).  We find it easier to consider the three cases separately, and so assume that \(\sigma\geq 0\) always holds.

\subsection{Invariant subsets determined by the sign of $x_\alpha$}
It follows directly from the equation (\ref{eq-BlowUpSystem-Polynomial}a) for $x_\alpha$ that for any $\alpha\in\{1,2\}$, the three subsets of $\R^6$ defined by $\{x_\alpha =0\}$, $\{x_\alpha >0\}$, and $\{x_\alpha <0\}$ are invariant under the soliton flow $g^t$.

The region $x_\alpha < 0$ is of no interest to us, because it does not lead to metrics of the form~\eqref{eq-MetricAnsatz}.

\subsection{Metrics with equal radii}

\label{sec-equal-radii} Even though the soliton equations \eqref{eq-BlowUpSystem-Polynomial} are not invariant under exchange of $(x_1, y_1)$ and $(x_2, y_2)$ unless $p_1=p_2$, it is true that the four-dimensional subspace of $\R^6$ defined by
\[
  \equalradii \isdef \bigl\{ (x_1, y_1, x_2, y_2, \Gamma, \sigma) \mid x_1=x_2 \text{ and }y_1=y_2\bigr\}
\]
is invariant under the soliton flow \eqref{eq-BlowUpSystem-Polynomial}.

Indeed, if $(x_1, y_1, x_2, y_2, \Gamma, \sigma)$ solves~\eqref{eq-BlowUpSystem-Polynomial}, then $(x_1(t), y_1(t))$ and $(x_2(t),$ $y_2(t))$ are both solutions of~\eqref{eq-Pa}--\eqref{eq-Pb} with the same $\Gamma(t)$.  The uniqueness theorem for \textsc{ode} now implies that $(x_1(t), y_1(t))$ and $(x_2(t), y_2(t))$ either coincide for all $t$, or else are different for all $t$.

Alternatively, one could simply observe that \(x_{12}=y_{12}=0\) is a solution of the difference equations~\eqref{eq-difference}, no matter what the averaged solutions \((x, y, \Gamma)\) are.

\section{Special orbits of the soliton flow \texorpdfstring{$g^t$}{gt}}
\label{SpecialOrbits} Fixed points of the flow $g^t$ correspond to $t$-independent solutions of \eqref{eq-BlowUpSystem-Polynomial}.  Since $ \sigma' =2 \sigma $, fixed points can only occur in the hyperplane $\sigma=0$.  At a stationary solution in that hyperplane, system~\eqref{eq-BlowUpSystem-Polynomial} reduces to
\begin{equation} \label{eq-soliton-fixed-points} x_\alpha y_\alpha = 0, \quad x_\alpha = -(\Gamma+1)(y_\alpha+1), \quad \Gamma = -\tsum_\alpha p_\alpha (1+y_\alpha)^2.
\end{equation}
We can classify solutions according to the number of $\alpha$ for which $x_\alpha=0$.

\subsection{Ricci-flat Cone $(\rfc)$}
If $(x_\alpha, y_\alpha, \Gamma, \sigma=0)$ is a fixed point for which neither of the $x_\alpha$ vanish, then $y_\alpha= 0$ for each $\alpha$, and thus $\Gamma=-\tsum_\alpha p_\alpha = -n$.  The second equation in \eqref{eq-soliton-fixed-points} then implies that $x_\alpha = -\Gamma-1 = n-1$ for both $\alpha$.  Thus we have
\[
  (x_1, y_1, x_2, y_2, \Gamma, \sigma) = \bigl(n-1, 0,\, n-1, 0,\, -n, 0\bigr)=\rfc.
\]
This fixed point corresponds to a Ricci-flat cone that will play a central role in our subsequent analysis.

\subsection{Good Fills $(\gf)$}
If we have a fixed point at which $x_\alpha=0$ for exactly one $\alpha\in\{1,2\}$, then for $\beta\neq\alpha$ we have $x_\beta\neq0$, so that $x_\beta y_\beta=0$ implies $y_\beta=0$, while $x_\beta = -(\Gamma+1)(y_\beta+1)=-\Gamma-1$ implies $\Gamma+1\neq0$.  On the other hand, $x_\alpha=0$ implies $x_\alpha = -(\Gamma+1)(y_\alpha+1)= 0$, so that $y_\alpha=-1$.

For each $\alpha\in\{1,2\}$, we therefore get a ``Good Fill'' that smoothly fills one sphere with a disc.  For $\alpha=1$, we get
\[
  (x_1, y_1, x_2, y_2, \Gamma, \sigma) = \bigl(0, -1, \, p_2-1, 0,\, -p_2, 0\bigr),
\]
while $\alpha=2$ yields
\begin{equation} \label{eq-BestEnd} \gf = (x_1, y_1, x_2, y_2, \Gamma, \sigma) = \bigl(p_1-1, 0, \, 0, -1,\, -p_1, 0\bigr).
\end{equation}
In what follows, we study orbits of the soliton flow that converge to a Good Fill of one sphere.  We will always assume that the Good Fill is the one with $\alpha=2$, \emph{i.e.,} the solution $\gf$ given in~\eqref{eq-BestEnd}.  The corresponding metrics then satisfy the boundary conditions \eqref{eq-gf-boundaryconditions}.  In Lemma~\ref{lem-Wu-gf-analytic}, we verify that all solutions \(\bigl(x_1(\tau),y_1(\tau), x_2(\tau), y_2(\tau), \Gamma(\tau), \sigma(\tau)\bigr)\) of~\eqref{eq-BlowUpSystem-Polynomial} attaining these boundary conditions actually are real analytic functions of \(\sigma=s^2\), so that the corresponding metrics are complete as $\sigma\searrow0$ and locally diffeomorphic to $\mc D^{p_1+1}\times\mc S^{p_2}$.

\subsection{Non-Cones}
If $x_\alpha=0$ for both $\alpha$, then we have two possibilities: either $\Gamma=-1$, or else $y_\alpha=-1$ for both $\alpha$.

In the first case, where $\Gamma=-1$, we find a continuum of fixed points given by
\[
  (x_1, y_1, x_2, y_2, \Gamma, \sigma) = (0, y_1^*,\, 0, y_2^*,\, -1, 0),
\]
where the two parameters $y_\alpha^*$ must satisfy the constraint
\[
  \tsum_\alpha p_\alpha (1+y_\alpha^*)^2 = 1.
\]
To describe these, we rewrite the metric~\eqref{eq-MetricAnsatz} in the equivalent form~\eqref{eq-Petersen}.
Unwrapping variables, one finds that near a Non-Cone fixed point with $\Gamma=-1$, the functions $\vp_\alpha$ exhibit the asymptotic behavior
\[
  \vp_\alpha\sim C_\alpha s^{1+y_\alpha^*},
\]
where $p_\alpha\geq 2$ implies $|1+y_\alpha^*|\leq \frac12\sqrt2$.  These points therefore cannot produce Good Fills as $s\searrow0$.  Hence we do not study them further.

The remaining case is that in which $x_\alpha = 0$ and $y_\alpha = -1$ for all $\alpha$, and thus $\Gamma = 0$.  These correspond to solutions for which the functions \(\varphi_\alpha = (p_\alpha-1)s^2/x_\alpha\) have nonzero limits as \(s\searrow 0\).  For such solutions, the system~\eqref{eq-LittleSolitonSystem} is not singular, and hence the solutions extend into the region $\{s<0\}$.  One such solution is the \emph{generalized shrinking cylinder soliton} on \(\R\times\mc S^{p_1}\times\mc S^{p_2}\), which occurs for \(\lambda=-1\) and in which for all \(s\in\R\) one has
\[
  x_\alpha = s^2, \quad y_\alpha = -1, \quad \Gamma = C_0s.
\]
The metric and soliton field for this solution are
\[
  g = (\mr ds)^2 + (p_1-1)g_{\mc S^{p_1}} + (p_2-1)g_{\mc S^{p_2}} \quad\mbox{ and }\quad \mf X = (C_0+s)\frac{\pd}{\pd s}.
\]
In this paper, we will not investigate whether other solitons of this type exist.

\subsection{The Ricci-flat Cone regarded as an expander or shrinker ($\rfes$)}
There exists a solution of~\eqref{eq-BlowUpSystem-Polynomial} given by $x_\alpha=n-1$, $y_\alpha=0$, $\Gamma=-n$, and $0<\sigma<\infty$, namely
\[
  (x_1, y_1, x_2, y_2, \Gamma, \sigma) = \bigl(n-1, 0,\, n-1, 0,\, -n, \sigma\bigr),\qquad 0<\sigma<\infty .
\]
We refer to this as the $\rfes$, and discuss it further in \S~\ref{sec-ConvergeToRFC} below.

\subsection{The invariant submanifold $\mc E$}
\label{sec-Ivey-and-co}
We consider the following quantities:
\begin{subequations}\label{eq-FJdefined}
  \begin{align}
    J & \stackrel{\rm def}{=} \tfrac12\tsum_\alpha p_\alpha \bigl\{x_\alpha+(1+y_\alpha)^2\bigr\}
        -  \tfrac12\Gamma^2,   \\
    F & \stackrel{\rm def}{=} \Gamma + \tsum_\alpha p_\alpha(1+y_\alpha)
        = \Gamma + n + \tsum_\alpha p_\alpha y_\alpha.
  \end{align}
\end{subequations}
The quantity $J$ is related to the invariant Ivey used~\cite{Ivey94} to construct the steady solitons he found.  The quantity $F$ is related to the vector field $\mf X=f(s)\frac{\pd} {\pd s}$ that generates the diffeomorphisms by which a soliton flows. Indeed, it follows immediately from equation~\eqref{eq-define-Gamma} that $F=sf+\lambda s^2$.

Direct substitution shows that both quantities $J$ and $F$ vanish at the $\gf$ and $\rfc$ fixed points.  The differential equations for $(x_\alpha, y_\alpha, \Gamma, \sigma)$ imply that
\begin{subequations}\label{eq-ivey-F-evolution}%
  \begin{align}
    J' & = 2J -\lambda \sigma\tsum_\alpha p_\alpha y_\alpha(1+y_\alpha),       \\
    F' & = (\Gamma+1) F + 2 J - \lambda \sigma \tsum_\alpha p_\alpha y_\alpha.
  \end{align}
\end{subequations}

The joint zero-set of $F$, $J$, and $\sigma$ is of interest, so we define
\[
  \mc E \isdef \bigl\{(x_\alpha, y_\alpha, \Gamma, \sigma)\in\R^6 \mid J=F=\sigma=0\bigr\}.
\]
The definitions of $J$ and $F$ imply that $\mc E$ is a three-dimensional submanifold of $\R^5\times\{0\}$ given by
\[
  \tsum_\alpha p_\alpha x_\alpha = \tsum_\alpha p_\alpha (1+y_\alpha)^2 - \Bigl(\tsum_\alpha p_\alpha (1+y_\alpha)\Bigr)^2 \qquad\text{and}\qquad \Gamma = - \tsum_\alpha p_\alpha (1+y_\alpha).
\]
The manifold $\mc E$ contains both fixed points $\gf$ and $\rfc$ as well as the unique heteroclinic orbit between them discovered by B\"ohm.  We recall Gastel and Kronz' proof \cite{GK04} of this fact in \S~\ref{thm-BGK}.

Since $\mc E$ is invariant under the flow $g^t$, the tangent space $T_{\fp}\mc E$ at any point $\fp$ of the flow is invariant under the linearization of the flow at $\fp$.  In the next section, this will help us organize the eigenvalues of the linearization.

\section{Linearization at fixed points}

\subsection{Linearization at the $\gf$}

\label{sec-Lin-GF}
Here we study the fixed point
\[
  \gf = (x_1, y_1, x_2, y_2, \Gamma, \sigma) = \bigl(p_1-1, 0, \, 0, -1,\, -p_1, 0\bigr).
\]
Recall that $X$ given in~\eqref{eq-soliton-vectorfield} denotes the vector field on $\R^6$ defined by the system~\eqref{eq-BlowUpSystem-Polynomial}.  At the $\gf$ fixed point, the linearization is given by
\[
  \dd X_{\gf} =
  \begin{bmatrix}
    0 & -2(p_1-1) & 0 & 0        & 0 & 0       \\
    1 & -(p_1-1)  & 0 & 0        & 1 & 0       \\
    0 & 0         & 2 & 0        & 0 & 0       \\
    0 & 0         & 1 & -(p_1-1) & 0 & \lambda \\
    0 & 2p_1      & 0 & 0        & 1 & 0       \\
    0 & 0 & 0 & 0 & 0 & 2
  \end{bmatrix}.
\]
The $\gf$ linearization splits into two independent subsystems.  The $(x_2, y_2, \sigma)$ subsystem has the matrix
\[
  \begin{bmatrix}
    2 & 0        & 0       \\
    1 & -(p_1-1) & \lambda \\
    0 & 0 & 2
  \end{bmatrix},
\]
whose eigenvalues are $\{-(p_1-1), +2, +2\}$.  The $(x_1, y_1, \Gamma)$ subsystem has the matrix
\[
  \begin{bmatrix}
    0 & -2(p_1-1) & 0 \\
    1 & -(p_1-1)  & 1 \\
    0 & 2p_1 & 1
  \end{bmatrix},
\]
whose eigenvalues are $\{-(p_1-1), -1, +2\}$.

Thus we see \(\gf\) is a hyperbolic fixed point with a three-dimensional unstable
manifold.  There is only one unstable eigenvalue, $\mu=2$, but it has multiplicity
three.  The eigenspace corresponding to the eigenvalue $2$ is spanned by the vectors
\begin{equation}\label{eq-GF-unstable-eigenvectors}
  \begin{gathered}
    \vE_1 = -(p_1-1)\frac{\pd}{\pd x_1} + \frac{\pd}{\pd y_1} + 2p_1 \frac{\pd}{\pd \Gamma},
    \\
    \vE_{2} = (p_1+1)\frac{\pd}{\pd x_2} + \frac{\pd}{\pd y_2},\quad \text{ and }\quad \vE_\sigma = -\lambda\frac{\pd}{\pd x_2} + \frac{\pd}{\pd\sigma}.
  \end{gathered}
\end{equation}

There is no resonance among the unstable eigenvalues, in the sense that none of the
eigenvalues is a nontrivial integer combination of the others (see the chapter on
normal forms in \cite{Arnold88}). Thus by Poincar\'e's theorem~\cite{Arnold88}, the
fixed point $\gf\in\R^6$ has a three-dimensional real analytic unstable manifold, and
there is a real analytic parametrization of $\wug$ that conjugates the flow of $ X$
to the flow of its linearization, \emph{i.e.,} $\dot {\bs x} = 2{\bs x}$ for
$\bs x\in T_{\gf}\wug$.  All of this directly implies:
\begin{lemma}\label{lem-Wu-gf-analytic}
  Solutions $t\mapsto (x_1, y_1, x_2, y_2, \Gamma, \sigma)$ of
  \eqref{eq-BlowUpSystem-Polynomial} that lie in $\wug$ for which $\sigma>0$ are
  graphs over the $\sigma$ axis in which the remaining variables $x_1$, $y_1$, $x_2$,
  $y_2$, and $\Gamma$ are real analytic functions of $\sigma = s^2$.
\end{lemma}

Indeed, solutions of the linearized flow in $T_{\gf}\wug$ are of the form $\bs x(t)= e^{2t}\bs x(0)$.  The flow on $\wug$ is analytically conjugate to the linear flow on the tangent space, so all solutions are analytic functions of $e^{2t}$.  If $\sigma\neq0$ on a solution, then~\eqref{eq-Pd} implies that $\sigma(t) = \sigma(0)e^{2t}$, \emph{i.e.,} $e^{2t} = \sigma/\sigma(0)$.  Hence solutions are convergent power series in $\sigma/\sigma(0)$.  \medskip

Recall that in \S~\ref{sec-Ivey-and-co}, we defined $\mc E$ to be the three-dimensional submanifold of $\R^5\times\{0\}$ defined by the equations $J=F=0$.  This submanifold is invariant under the flow $g^t$, and therefore the tangent space $T_\gf\mc E$ is an invariant subspace for the linearization $\dd X_\gf$.
\begin{lemma}\label{lem-GF-eigen-data}
  The eigenvalues of $\dd X_\gf$ restricted to $T_\gf\mc E$ are $\{ -(p_1-1),-1,+2\}$.  The unique unstable eigenvalue has eigenvector
  \[
    p_2\vE_1 - 3p_1\vE_2.
  \]
\end{lemma}
\begin{proof}
  The tangent space $T_\gf\mc E$ is the combined null space of $\dd J_\gf$ and $\dd F_\gf$.  We compute that
  \begin{align*}
    \dd J & = \tfrac12\tsum_\alpha p_\alpha \dd x_\alpha
            + \tsum_\alpha p_\alpha(1+y_\alpha) \dd y_\alpha - \Gamma \dd\Gamma, \\
    \dd F & = \tsum_\alpha p_\alpha \dd y_\alpha + \dd\Gamma.
  \end{align*}%
  Thus at the $\gf$ fixed point,
  \begin{align*}
    \dd J_\gf & = \tfrac12 p_1 \dd x_1 + \tfrac12 p_2 \dd x_2
                +  p_1 \dd y_1  + p_1 \dd\Gamma,                          \\
    \dd F_\gf & = p_1 \dd y_1 + p_2 \dd y_2 + \dd\Gamma.
  \end{align*}
  Unstable eigenvectors are linear combinations of $\vE_1$ and $\vE_2$.  Because\footnote{Throughout this paper, we denote the result of a one-form $\varrho:\R^6\to T^*\R^6$ acting on a vector $V\in T_{\fp}\R^6$ by $\varrho_{\fp}\cdot V$, or just $\varrho\cdot V$, if the point $\fp\in\R^6$ of evaluation can be deduced from the context.}
  \begin{align*}
    \dd J_\gf\cdot \vE_1 & = -\tfrac12 p_1(p_1-1) + p_1 +2p_1^2
                           = \tfrac32 p_1(p_1+1),                                      \\
    \dd J_\gf\cdot \vE_2 & = \tfrac12 p_2(p_1+1),               \\
    \dd F_\gf\cdot \vE_1 & = p_1 + 2p_1 = 3p_1,                 \\
    \dd F_\gf\cdot \vE_2 & = p_2,
  \end{align*}
  we see that $p_2\vE_1-3p_1\vE_2$ belongs to the kernels of $\dd J_\gf$ and of $\dd F_\gf$.

  To finish the proof, we use similar reasoning to see that
  \[
    2p_2\frac{\pd}{\pd x_1} + p_2\frac{\pd}{\pd y_1} -(p_1-2)\frac{\pd}{\pd y_2} - 2p_2\frac{\pd}{\pd\Gamma} \quad\text{ and } \quad 2(p_1-1)\frac{\pd}{\pd x_1} + \frac{\pd}{\pd y_1}-p_1\frac{\pd}{\pd\Gamma}
  \]
  belong to the subspace $T_\gf\mc E$ and are eigenvectors of $\dd X_{\gf}$ for the eigenvalues $-(p_1-1)$ and $-1$, respectively.
\end{proof}

\subsection{Linearization at the $\rfc$}

\label{sec-Lin-RFC} Recall that the fixed point $\rfc$ corresponds to $(x_1,y_1,x_2,y_2,\Gamma,\sigma) = (n-1,0,n-1,0,-n,0)$.  In these coordinates, the matrix of the linearization of~\eqref{eq-BlowUpSystem-Polynomial} at the $\rfc$ is
\[
  \dd X_\rfc =
  \begin{bmatrix}
    0 & -2(n-1) & 0 & 0       & 0 & 0 \\
    1 & -(n-1)  & 0 & 0       & 1 & 0 \\
    0 & 0       & 0 & -2(n-1) & 0 & 0 \\
    0 & 0       & 1 & -(n-1)  & 1 & 0 \\
    0 & 2p_1    & 0 & 2p_2    & 1 & 0 \\
    0 & 0 & 0 & 0 & 0 & 2
  \end{bmatrix}.
\]
It is immediately clear that $\frac{\pd}{\pd\sigma}$ is an eigenvector with eigenvalue $ +2 $.

To understand the dynamics of the soliton system near the $\rfc$, it is useful to use the \emph{averaged} and \emph{difference} variables \((x, y, x_{12}, y_{12})\), which satisfy equations~\eqref{eq-average} and \eqref{eq-difference}, respectively.  Recall that in \eqref{eq-xi-y-gamma-def}, we defined the perturbations \(\xi_\alpha = x_\alpha-n+1\) and \(\gamma= \Gamma+n\), both of which vanish at the \(\rfc\).  We also introduce the averaged \(\xi_\alpha\), namely
\begin{equation} \label{eq-xi-average-def} \xi = \sum_\alpha \frac{p_\alpha}{n} \bigl(x_\alpha-n+1\bigr).
\end{equation}
Observe that there is no need to consider the difference of the \(\xi_\alpha\) separately, because
\[
  x_{12} = x_1-x_2 = \xi_1-\xi_2.
\]
In coordinates $(\xi, y,\gamma, x_{12},y_{12})$, the matrix of the linearization of~\eqref{eq-BlowUpSystem-Polynomial} at the $\rfc$, restricted to the subspace $\sigma=0$, admits the block decomposition
\[
  \dd X_\rfc\big|_{\sigma=0}=
  \begin{bmatrix}
    0 & -2(n-1) & 0 &   &         \\
    1 & -(n-1)  & 1 &   &         \\
    0 & 2n      & 1 &   &         \\
    &         &   & 0 & -2(n-1) \\
    & & & 1 & -(n-1)
  \end{bmatrix}.
\]

On the $3$-dimensional subspace corresponding to the averaged variables~\((\xi, y, \gamma)\), $\dd X_\rfc$ has eigenvalues $\mu\in\{-(n-1), -1, 2\}$ with eigenvectors $\vV_\mu$ given in the basis \(\{\frac{\partial}{\partial\xi},\frac{\partial}{\partial y}, \frac{\partial}{\partial\gamma}\}\) by
\begin{equation}\label{eq-rfc-equal-radii-eigenvectors}
  \vV_2 = \mat -(n-1) \\ 1 \\ 2n \rix,\qquad
  \vV_{-1} = \mat 2(n-1) \\ 1 \\ -n \rix,\qquad
  \vV_{-(n-1)} = \mat 2 \\ 1 \\ -2 \rix.
\end{equation}

If the dimension lies in the range $n\in\{2, \dots, 8\}$, the eigenvalues of $\dd X_\rfc$ acting on the $2$-dimensional subspace on which \(\xi=y=\gamma=\sigma=0\) (spanned by \(\{\frac{\pd}{\pd x_{12}},\,\frac{\pd}{\pd y_{12}}\}\)) are complex.  They are $\{\omega, \bar\omega\}$, where
\begin{equation}
  \label{eq-A-pm-iOmega-defined}
  \omega = -A + i\Omega, \qquad\mbox{with}\qquad A=\frac{n-1} {2} \quad\text{ and }\quad \Omega = \frac{\sqrt{(n-1)(9-n)}}{2}.
\end{equation}
Since we assume that $p_\alpha \geq 2$ for $\alpha=1,2$, the dimension in our case is always bounded from below by $n\geq 4$.  For future use, we note that the following identity holds:
\begin{equation}
  \label{eq-A2plusOmega2-identity}
  A^2+\Omega^2 = 2(n-1).
\end{equation}

\subsection*{Summary}
The $2$-dimensional unstable manifold of the $\rfc$ is spanned by eigenvectors \(\{\pdd\sigma , \vV_2\}\) with the same eigenvalue $+2$.  The $4$-dimensional stable manifold is spanned by \(\{\vV_{-1}, \vV_{-(n-1)}\}\) and the two complex eigenvectors corresponding to the complex eigenvalues $A\pm i\Omega$.

\section{Linearization at the \texorpdfstring{$\rfes$}{RFES} soliton}
\label{sec-ConvergeToRFC}

\subsection{A linear non autonomous system}
To analyze the soliton flow near the \(\rfes\), we use the modified average/difference coordinates \((\xi, y, \gamma, x_{12}, y_{12})\).  According to \eqref{eq-average} and \eqref{eq-difference}, these satisfy
\begin{equation}\label{eq-rfes-average-nonlin}
  \left\{
    \begin{aligned} 
      \xi' & = -2(n-1) y - 2\xi y -2 \frac{p_1p_2}{n^2} x_{12}y_{12},\\
      y' & = \xi - (n-1 + \lambda s^2)y + \gamma +y\gamma,\\
      \gamma' &= \gamma + 2ny + ny^2 + \frac{p_1p_2}{n} y_{12}^2,
    \end{aligned}
  \right.
\end{equation}
and
\begin{equation}\label{eq-rfes-difference-nonlin}
  \left\{
    \begin{aligned}
      x_{12}' &= -2(n-1)y_{12} -2 yx_{12} -2\xi y_{12} + 2\frac{p_1-p_2}{n}x_{12}y_{12},\\
      y_{12}' &= x_{12} - (n-1-\gamma+\lambda s^2)y_{12}.
    \end{aligned}
  \right.
\end{equation}
Here, as always, \('\) stands for \(\frac{\dd}{\dd t} = s\frac{\dd}{\dd s}\).

This system can be regarded as a linear system in the difference variables whose coefficients depend on the difference and averaged variables.  We write this in matrix form as
\begin{equation}
  \label{eq-difference-system-matrix-form}
  \mat x_{12}' \\ y_{12}' \rix=
  \mat
  -2y- \frac2n (p_2-p_1)y_{12} & -2(n-1)+\xi \\
  1 & -n+1+\gamma-\lambda s^2
  \rix
  \mat x_{12}\\ y_{12}\rix.
\end{equation}

If we discard the terms in \eqref{eq-rfes-difference-nonlin} that are quadratic in \((\xi, y, \gamma, x_{12}, y_{12})\), then we are left with two non-autonomous systems of linear equations,
\begin{equation}\label{eq-three-by-three-subsystem}
  \left\{
    \begin{aligned}
      \xi'     & = -2(n-1) y,                           \\
      y'       & = \xi - (n-1+\lambda s^2) y + \gamma,  \\
      \gamma ' & = 2n y + \gamma,
    \end{aligned}
  \right.
\end{equation}
and
\begin{equation}
  \label{eq-two-by-two-subsystem}
  \left\{
    \begin{aligned}
      x_{12}' & =  -2(n-1)y_{12},                             \\
      y_{12}' & = x_{12} - \bigl(n-1+\lambda s^2\bigr)y_{12}.
    \end{aligned}
  \right.
\end{equation}

We see that the linearization of the flow around the \(\rfes\) decouples into two smaller systems of equations.  We group variables accordingly and define
\begin{equation} \label{eq-Phi-zeta-introduced} \Phi = \mat \xi \\ y\\ \gamma\rix \qquad \text{ and }\qquad \zeta = (A+i\Omega)x_{12} - (A^2+\Omega^2)y_{12}.
\end{equation}
As noted in~\eqref{eq-A2plusOmega2-identity}, $A$ and $\Omega$ satisfy \(A^2+\Omega^2 = 2(n-1)\).



\subsection{Reduction to higher-order scalar differential equations}

\begin{lemma}\label{lem-reduction-to-scalar-eqns}
  The linear system~\eqref{eq-three-by-three-subsystem} is equivalent to the third-order scalar \textsc{ode}
  \begin{equation}\label{eq-xi-bar-third-order}
    \xi''' + (n-2)\xi'' - (n+1)\xi' -2(n-1)\xi +\lambda
    s^2(\xi'' + \xi')
    =0.
  \end{equation}
  One can recover $\Phi$ from any solution $\xi$ of \eqref{eq-xi-bar-third-order} via
  \begin{equation}
    \Phi =
    \mat
    \xi \\ y \\ \gamma
    \rix
    = \frac{-1}{2(n-1)}
    \mat
    -2(n-1)\xi \\
    \xi' \\
    \xi'' + (n-1+\lambda s^2)\xi' + 2(n-1)\xi
    \rix.
  \end{equation}

  Similarly, the linear system \eqref{eq-two-by-two-subsystem} is equivalent to the second-order scalar \textsc{ode}
  \begin{equation}\label{eq-chi-second-order}
    \chi'' + (n-1+\lambda s^2)\chi'+2(n-1)\chi=0,
  \end{equation}
  where one can recover $x_{12}$ and $y_{12}$ from a given solution of \eqref{eq-chi-second-order} via
  \begin{equation}\label{eq-zeta-from-chi}
    x_{12} = \chi, \qquad y_{12} = -\frac{\chi'}{2(n-1)} ,\qquad
    \zeta = (A+i\Omega)\chi + \chi'.
  \end{equation}
\end{lemma}
\begin{proof}
  One can use the first equation in \eqref{eq-three-by-three-subsystem} to write $ y$ in terms of $\xi'$; then one can use the second equation in \eqref{eq-three-by-three-subsystem} to express $\gamma$ in terms of $\xi$, $\xi'$, and $\xi''$.  Substituting all this into the third equation of~\eqref{eq-three-by-three-subsystem} then leads to \eqref{eq-xi-bar-third-order}.  The derivation of \eqref{eq-chi-second-order} proceeds along the same lines.

  The expression for $\Phi$ follows directly from~\eqref{eq-three-by-three-subsystem}, while $\zeta = \chi'+(A+i\Omega)\chi$ follows from \eqref{eq-Phi-zeta-introduced} by using~\eqref{eq-A2plusOmega2-identity}, \emph{i.e.,} $A^2+\Omega^2 = 4A = 2(n-1)$.
\end{proof}

\subsection{Fundamental solutions of (\ref{eq-three-by-three-subsystem}) and (\ref{eq-two-by-two-subsystem})}
\label{sec-fundamental-sols-of-linearization}
Here we take a closer look at the \textsc{ode} for $\xi$ and $\chi$.

Equation~\eqref{eq-xi-bar-third-order} for $\xi$ can be simplified by observing that if
\[
  \psi = \xi'+\xi,
\]
then~\eqref{eq-xi-bar-third-order} is equivalent to
\[
  \psi'' +(n-3)\psi' - 2(n-1) \psi + \lambda s^2\psi' =0.
\]
Recalling that ${}' = \frac{\dd}{\dd t} = s\frac{\dd}{\ds}$, one can rewrite this equation as
\begin{equation} \label{eq-psi-s} \psi_{ss} + \Bigl(\frac{n-2}{s} + {\lambda s}\Bigr) \psi_s - \frac{2(n-1)}{s^2}\psi = 0.
\end{equation}
Equation~\eqref{eq-chi-second-order} for $\chi$ is equivalent to
\begin{equation}\label{eq-chi-s}
  \chi_{ss}
  + \left(\frac ns+\lambda s\right) \chi_s
  + \frac{2(n-1)}{s^2} \chi = 0.
\end{equation}
Both of these equations are of confluent hypergeometric type.  We need to classify their solutions in terms of their asymptotic behaviors at $s=0$ and at $s=\infty$.  Here we have to distinguish between $\lambda=+1$ and $\lambda=-1$.

\begin{lemma}
  The \textsc{ode}~\eqref{eq-xi-bar-third-order} and~\eqref{eq-chi-second-order} have linearly independent solutions $\{\xi_{0}, \xi_{1}^\pm, \xi_{2}^\pm\}$ and $\{\chi_1^\pm, \chi_2^\pm\}$, respectively, whose asymptotic behaviors are displayed in Table~\ref{tab-expander-asymptotics} for $\lambda>0$ and in Table~~\ref{tab-shrinker-asymptotics} for $\lambda<0$.
\end{lemma}

\begin{proof}

  Since $\psi=\xi'+\xi = s\xi_s + \xi = (s\xi)_s$, we see that $\xi_{0} \isdef s^{-1}$ is an exact solution of~\eqref{eq-xi-bar-third-order} for either value of $\lambda$.  And it is easy to verify that an exact solution of~\eqref{eq-psi-s} is
  \[
    \psi_1=s^{-(n-1)}e^{-\lambda s^2/2}.
  \]
  Using $\psi_1$ and reduction of order, will will derive the claimed behaviors of $\xi_{1}^\pm$ and $\xi_2^\pm$.  We deal with $\chi_1^\pm$ and $\chi_2^\pm$ below.

  \begin{table}[t]\centering

    \begin{tabular}{c@{\hspace{2em}}c@{\hspace{2em}}c}
      \toprule
      $\lambda=+1$
      & $s\to0$
      & $s\to\infty$
      \\ \midrule
      $\xi_{0} $
      & $1/s$
      & $1/s$
      \\[6pt]
      $\xi_{1}^+ $
      & $\dfrac{1+o(1)}{(n-2) s^{n-1}}$
      & $\big(1+o(1)\big)\,s^{-(n+1)}e^{-s^2/2} $
      \\[8pt]
      $\xi_{2}^+ $
      & $\dfrac{s^{2}}{3(n+1)} + \cO(s^{4})$
      & $1 - \dfrac{C_2^+}{s}+ \cO(s^{-2})$
      \\[6pt] \midrule
      $\chi_1^+  $
      & $\Im\bigl[(k^++\cO(s^2))\,s^{-A+i\Omega}\bigr]$
      & $ 1 + \cO(s^{-2})$
      \\[6pt]
      $\chi_2^+ $
      & $\Im \bigl[(ik^++\cO(s^2))\,s^{-A+i\Omega}\bigr]$
      & $ \bigl(C_\chi^++\cO(s^{-2})\bigr)\, s ^{-(n+1)}e^{- s^2/2} $
      \\[3pt]  \bottomrule
    \end{tabular}
    \medskip
    \caption{Asymptotics of the solutions of~\eqref{eq-xi-bar-third-order}, \eqref{eq-chi-second-order} for $\lambda=+1$.  Here $C_2^+$, $C_\chi^+\in \R$, and $k^+\in\C$ are nonzero constants.  }
    \label{tab-expander-asymptotics}
  \end{table}

  \begin{table}[t]
    \centering
    \begin{tabular}{c@{\hspace{2em}}c@{\hspace{2em}}c}
      \toprule
      $\lambda=-1$
      & $s\to0$
      & $s\to\infty$
      \\ \midrule
      $\xi_{0} $
      & $1/s$
      & $1/s$
      \\[6pt]

      $\xi_{1}^- $
      & $C_1^-s^{-(n-1)}+\cO\bigl(s^{-(n-3)}\bigr)$
      & $1+(n-1)s^{-2}+\cO(s^{-4}) $
      \\[6pt]
      $\xi_{2}^- $
      & $\dfrac{s^2}{3(n+1)}+\cO(s^4)$
      & $\big(C_2^-+\cO(s^{-2})\big)\,s^{-(n+1)}e^{s^2/2} $
      \\[8pt] \midrule
      $\chi_1^- $
      & $\Im \bigl[(k^-+\cO(s^2))\,s^{-A+i\Omega}\bigr]$
      & $ 1+\cO(s^{-2})$
      \\[6pt]
      $\chi_2^- $
      & $\Im \bigl[(ik^-+\cO(s^2))\,s^{-A+i\Omega}\bigr]$
      & $ \bigl(C_\chi^-+\cO(s^{-2})\bigr)\, s ^{-(n+1)}e^{ s^2/2}$
      \\[3pt]\bottomrule
    \end{tabular}\medskip
    \caption{Asymptotics of the solutions of~\eqref{eq-xi-bar-third-order}, \eqref{eq-chi-second-order} for $\lambda=-1$.  Here $C_\chi^-\in \R$ and $k^-\in \C$ are nonzero constants, while $C_2^- = (n-2)C_1^- = 2^{(n-1)/2}\Gamma\big(\frac{n+1}{2}\big)$.  }
    \label{tab-shrinker-asymptotics}
  \end{table}

  \smallskip

  \textbf{The case $\lambda=+1$.  } If $\lambda>0$, then $\psi_1^+(s)\isdef\psi_1(s)|_{\lambda>0}= s^{-(n-1)}e^{-s^2/2}$.  The corresponding solution $\xi_{1}^+$ is given by
  \[
    \xi_{1}^+(s) = \frac 1s \int_s^\infty \psi_1^+(r)\,\mr dr.
  \]
  Clearly \(\xi_1^+(s)\) is positive, and its asymptotic behaviors at $s=0$ and $s=\infty$ in Table~\ref{tab-expander-asymptotics} are easily verified.

  Using reduction of order, one finds that a second linearly independent solution of~\eqref{eq-psi-s} is
  \begin{equation}
    \label{eq-psi2-expander}
    \psi_2^+(s) = \psi_1^+(s)  \int_0^s r^n e^{r^2/2}\mr dr.
  \end{equation}
  At $s=0$, we get
  \[
    \psi_2^+(s) = \frac{s^2}{n+1} + \cO(s^4), \qquad (s\to0).
  \]
  To find an expansion at large $s$, we integrate by parts,
  \[
    \int_0^s r^ne^{r^2/2}\mr d r = s^{n-1}e^{s^2/2} - (n-1)\int_0^s r^{n-2}e^{r^2/2} \mr d r,
  \]
  which then leads to
  \[
    \psi_2^+(s) = 1 - s^{-(n-1)}e^{-s^2/2}\int_0^s r^{n-2}e^{r^2/2}\mr d r.
  \]
  Integrating by parts again yields an asymptotic expansion whose first few terms are
  \[
    \psi_2^+(s) = 1 - \frac{n-1}{s^2} + \cO(s^{-4}), \qquad (s\to\infty).
  \]
  The corresponding solution $\xi_{2}^+$ is
  \[
    \xi_{2}^+(s) = \frac 1s \int_0^s \psi_2^+(r)\,\mr dr.
  \]
  For $s\approx 0$, we therefore get
  \[
    \xi_2^+(s) = \frac{s^2}{3(n+1)} + \cO(s^4).
  \]
  For $s\to\infty$, we get
  \[
    \xi_2^+(s) = 1 - \frac{C_2^+}{s} + \cO (s^{-2}),
  \]
  where
  \[
    C_2^+ = \int_0^\infty \bigl(1-\psi_2^+(s)\bigr) \mr d s >0.
  \]

  \smallskip \textbf{The case $\lambda=-1$.  } Here $\xi_0 = s^{-1}$ is again a solution.

  We will use reduction of order to obtain two linearly independent solutions of~\eqref{eq-psi-s} from the exact solution $\psi_1(s)|_{\lambda=-1}= s^{-(n-1)}e^{s^2/2}$, with one bounded as $s\to\infty$ and the other bounded as $s\to0$.

  The special solution $\psi_1(s)|_{\lambda=-1}$ in itself does not lead to a useful solution $\xi_1$ of \eqref{eq-xi-bar-third-order}, but by using reduction of order, we can construct two other solutions $\psi_1^-$ and $\psi_2^-$ of~\eqref{eq-psi-s} for which the corresponding $\xi$ functions are relevant.

  To obtain $\xi_1^-$, we choose
  \[
    \psi_1^-(s)\isdef s^{-(n-1)}e^{s^2/2}\int_s^\infty r^n e^{-r^2/2}\,\mr dr.
  \]
  Integration by parts shows that
  \[
    \int_s^\infty r^n e^{-r^2/2}\,\mr dr =e^{-s^2/2}\big\{s^{n-1}+(n-1)s^{n-3}+\mc O(s^{n-5})\big\}, \qquad (s\to\infty).
  \]
  Therefore,
  \[
    \psi_1^-(s)=1+\frac{n-1}{s^2} + \frac{(n-1)(n-3)}{s^4} +\cO(s^{-6}),\qquad(s\to\infty).
  \]
  Repeated integration by parts shows that for odd values of \(n\), \(\psi_1^-(s)\) is a polynomial in \(s^{-2}\).  For even values of \(n\), one obtains an asymptotic expansion in arbitrarily high powers of \(s^{-2}\).

  We define \(\xi_1^-(s)\) by solving \((s\xi)_s = \psi\).  By writing the equation as \(\bigl(s(\xi-1)\bigr)_s = (s\xi)_s-1 = \psi-1\), and taking into account that \(\psi_1^-(s)-1 = \cO(s^{-2})\) as \(s\to\infty\), so that \(\psi_1^-(s)-1\) is integrable, we arrive at
  \begin{equation}
    \label{eq-xi1-shrinker-defined}
    \xi_1^-(s) \isdef
    1 - \frac{1}{s} \int_s^\infty \bigl(\psi_1^-(\varsigma)-1\bigr)\,\dd\varsigma.
  \end{equation}
  From the expansion of \(\psi_1^-(s)\), we then find the following expansion for \(\xi_1^-\),
  \[
    \xi_1^-(s)=1+\frac{n-1}{s^2} + \frac{(n-1)(n-3)}{3s^4} + \frac{(n-1)(n-3)(n-5)}{5s^6} +\cdots,\qquad(s\to\infty),
  \]
  which implies the large-\(s\) asymptotics in Table~\ref{tab-shrinker-asymptotics}.

  To verify the asymptotics at small \(s\), we consider that for $0<s\ll1$, one has
  \[
    \psi_1^-(s) = (n-2)C_1^-s^{-(n-1)} + \cO\bigl(s^{-(n-3)}\bigr),
  \]
  where $(n-2)C_1^-=\int_0^\infty r^n e^{-r^2/2}\,\dd r=2^{(n-1)/2}\Gamma\big(\tfrac{n+1}{2}\big)$.  Integration shows that \(\xi_1^-(s)\) as defined in~\eqref{eq-xi1-shrinker-defined} has the asymptotic behavior claimed in Table~\ref{tab-shrinker-asymptotics}.

  Our second linearly independent solution $\psi_2^-$ of~\eqref{eq-psi-s} is also obtained from the special solution $s^{-(n-1)}e^{s^2}$ through reduction of order,
  \begin{equation}
    \label{eq-psi2-shrinker}
    \psi_2^-(s) = s^{-(n-1)}e^{s^2/2} \int_0^s r^n e^{-r^2/2}\,\mr dr
    = (n-2)C_1^- s^{-(n-1)}e^{s^2/2} - \psi_1^-(s).
  \end{equation}
  For $0<s\ll1$, one sees that
  \[
    \psi_2^-(s)=\frac{s^2}{n+1}+\cO(s^4), \qquad (s\to0),
  \]
  and hence that $\xi_2^-(s) = \frac1s \int_0^s \psi_2^-(\tilde s)\,\dd\tilde s$ satisfies
  \[
    \xi_{2}^-=\frac{s^2}{3(n+1)}+\cO(s^4), \qquad (s\to0).
  \]
  On the other hand, for $s\gg1$, one has
  \[
    \psi_2^-(s) = \bigl((n-2)C_1^- +\cO(s^2)\bigr)s^{-(n-1)}e^{s^2/2}, \qquad (s\to\infty),
  \]
  with $C_1^-$ as above.  Integration by parts then leads to the asymptotic expression for $\xi_2^-$ in Table~\ref{tab-shrinker-asymptotics}.  \medskip

  Now we deal with solutions $\chi(s)$ of~\eqref{eq-chi-s}.  Classical theory \cite{Olv91} shows that a pair of linearly independent solutions is given by $\{\chi,\bar\chi\}$, where
  \[
    \chi(s) = s^{-A + i\Omega}\, \tensor*[_1]{\mr F}{_1}\Bigl( \frac12(-A + i\Omega),\,1 + i\Omega; \,-\frac{\lambda}{2}s^2 \Bigr),
  \]
  where $\tensor*[_1]{\mr F}{_1}$ denotes the Kummer confluent hypergeometric function,
  \[
    \tensor*[_1]{\mr F}{_1}(a,b;q) =\sum_{m=0}^\infty\frac{a(a+1)\cdots(a+m-1) q^m}{b(b+1)\cdots(b+m-1) m!}.
  \]
  As $|q|\rightarrow\infty$, with $-3\pi/2 < \arg(q) < \pi/2$, one has the asymptotic behaviors
  \[
    \tensor*[_1]{\mr F}{_1}(a,b;q)\sim\Gamma(b)\left\{\frac{e^q q^{a-b}}{\Gamma(a)}+\frac{(-q)^{-a}}{\Gamma(b-a)}\right\}.
  \]
  For a solution in which the first term above dominates, $\Re\big(\chi(s)\big)\sim e^{-\lambda s^2/2}s^{-(n+1)}$ as $s\rightarrow\infty$.  For a solution in which the second term dominates, $\Re\big(\chi(s)\big)\sim1$ as $s\rightarrow\infty$.  We choose real solutions $\chi_1,\,\chi_2$ so that $\chi_1\sim e^{-\lambda s^2/2}s^{-(n+1)}$ and $\chi_2=\cO(1)$ as $s\to\infty$.  We define $\chi_1^\pm=\chi_1|_{\lambda=\pm1}$ and $\chi_1^\pm=\chi_2|_{\lambda=\pm1}$.  Then these functions behave as claimed in the tables. The reader may refer to Appendix~\ref{appendix-chi} for a self-contained justification of these claims.
\end{proof}


\subsection{Asymptotics of \texorpdfstring{$\Phi$}{Phi} at $s=0$}
\label{Phi-small-s}
For each of the solutions $\{\xi_0, \xi_1^\pm, \xi_2^\pm\}$ we choose for $\lambda=\pm1$, we get a solution $\Phi^\pm$ of the system~\eqref{eq-three-by-three-subsystem}.  Similarly, each of the two pairs of solutions $\{\chi_1^\pm, \chi_2^\pm\}$ corresponds to solutions $\zeta^\pm$ of the homogeneous system~\eqref{eq-two-by-two-subsystem}.  In this and the next few sections, we translate the asymptotic behaviors of $\xi_j^\pm$ and $\chi_j^\pm$ from Tables~\ref{tab-expander-asymptotics} and~\ref{tab-shrinker-asymptotics} into asymptotic expansions for $\Phi^\pm$ and $\zeta^\pm$ at either end of the interval $s\in(0,\infty)$.

For any given solution $\xi$ of \eqref{eq-xi-bar-third-order}, it follows from Lemma~\ref{lem-reduction-to-scalar-eqns} that
\begin{equation}
  \label{eq:HomogeneousSolution}
  \Phi = c \mat
  -2(n-1)\xi\\
  \xi' \\
  \xi'' + (n-1+\lambda s^2)\xi' +2(n-1)\xi\\
  \rix
\end{equation}
is a solution of the homogeneous system \eqref{eq-three-by-three-subsystem}, for any choice of $c\neq0$.  For each of the five fundamental solutions listed below, we choose the constant $c$ so as to simplify the coefficients in the asymptotic expansions of the solutions.  We will describe the asymptotic behavior using the eigenvectors $\vV_2,\vV_{-1},\vV_{-(n-1)}$ that appear in~\eqref{eq-rfc-equal-radii-eigenvectors}.

Applying~\eqref{eq:HomogeneousSolution} to $\xi_0 = s^{-1}$, recalling that ${}^\prime=s\frac{\dd}{\ds}$, and choosing a convenient value of $c$, we get
\[
  \Phi_0 = s^{-1} \vV_{-1} + \mat 0 \\ 0 \\ \lambda s \rix.
\]

We do not have simple explicit expressions for $\xi_1^\pm$ or $\xi_2^\pm$, but after choosing convenient values for $c$, we can compute the asymptotic expansions of the corresponding $\Phi_j^\pm$ as $s\to 0$.

We define $\Phi_1^\pm$ for $\lambda=\pm1$ using $\xi_1^\pm$.  From Table~\ref{tab-expander-asymptotics} and Table~\ref{tab-shrinker-asymptotics}, we see that $\xi_1^\pm$ have the same leading terms at $s=0$, up to a constant.  So by making suitable choices $c^\pm$, we get
\[
  \Phi_1^\pm = s^{-(n-1)}\vV_{-(n-1)} + o\bigl(s^{-(n-1)}\bigr), \qquad (s\to0).
\]

Both Tables~\ref{tab-expander-asymptotics} and \ref{tab-shrinker-asymptotics} list the same asymptotic expansion for $\xi_2^\pm$ at $s=0$, so that for both values $\lambda=\pm1$, choosing $c=3(n+1)$ gives
\[
  \Phi_2^\pm = s^2 \vV_2 + \cO(s^4), \qquad (s\to0).
\]

\subsection{Asymptotics of \texorpdfstring{$\Phi$}{Phi} as \texorpdfstring{$s\to\infty$}{s=infty}}
\label{Phi-big-s}

For $\Phi_0$, the explicit solution we found above is valid for all $s$,
\[
  \Phi_0 = s^{-1} \vV_{-1} + \mat 0 \\ 0 \\ \lambda s \rix.
\]
We note that for this solution, $\gamma\sim\lambda s$ as $s\to\infty$.

If $\lambda=+1$, using Table~\ref{tab-expander-asymptotics} and recalling equation~\eqref{eq:HomogeneousSolution}, one sees that $\Phi_1^+$ (defined using $\xi_1^+$) satisfies
\[
  \Phi_1^+ = C_{1,\infty}^+ s^{-(n+1)} e^{- s^2/2} \mat
  2(n-1) +\cO(s^{-2}) \\
  s^2+(n+1)+\cO(s^{-2}) \\
  -4n+\cO(s^{-2}) \rix, \qquad (s\rightarrow\infty).
\]
Using the fact that $\xi_2=1-\lambda(n-1)s^{-2}+\cO(s^{-4})$, one finds that $\Phi_2^+$ has the asymptotic behavior
\[
  \Phi_2^+= C_{2,\infty}^+  \mat 1 + \cO(s^{-2}) \\ \cO(s^{-2}) \\
  \cO(s^{-2}) \rix, \qquad (s\to\infty).
\]
We note that our choice of $c=3(n+1)$ above forces $C_{2,\infty}^+=-6(n-1)(n+1)$.

If $\lambda=-1$, using Table~\ref{tab-shrinker-asymptotics} and equation~\eqref{eq:HomogeneousSolution}, one sees that $\Phi_1^-$ (defined with $\xi_1^-$) satisfies
\[
  \Phi_1^-= C_{1,\infty}^-  \mat 1+\cO(s^{-2}) \\ \cO(s^{-2}) \\
  \cO(s^{-2}) \rix, \qquad (s\to\infty).
\]
For $\Phi_2^-$, we have
\begin{equation}
  \label{eq-Phi2-shrink-at-infty}
  \Phi_2^- = C_{2,\infty}^- s^{-(n+1)} e^{+s^2/2}
  \mat
  -2(n-1)+\cO(s^{-2}) \\
  s^2 -(n+1)+\cO(s^{-2}) \\
  4n+\cO(s^{-2})
  \rix,
  \qquad (s\rightarrow\infty),
\end{equation}
where $C_{2,\infty}^-=2^{(n-1)/2}\Gamma(\tfrac{n+1}{2})$.

\subsection{Asymptotics of \texorpdfstring{$\zeta$}{zeta} at $s=0$}
\label{sec-zeta-small-s}

To recover the asymptotic expansion of $\zeta$ from $\chi$, we use~\eqref{eq-zeta-from-chi}, \emph{i.e.,} $\zeta = \chi'+(A+i\Omega)\chi$.

A short computation shows that if $f(s) = \Im (ks^{-A+i\Omega})$ with $k$ complex, then
\[
  f'+Af = \Im\bigl(i\Omega k s^{-A+i\Omega}\bigr) = \Omega\Re\bigl(ks^{-A+i\Omega}\bigr),
\]
so that $f'(s)+(A+i\Omega)f(s) = \Omega ks^{-A+i\Omega}$.

Referring back to Tables~\ref{tab-expander-asymptotics} and \ref{tab-shrinker-asymptotics}, this implies that as $s\to0$,
\begin{equation}
  \label{eq-zeta-small-s}
  \zeta_1^\pm(s) =\big(\Omega k^\pm+ \cO(s^2)\big) s^{-A+i\Omega}
  \quad\mbox{and}\quad
  \zeta_2^\pm(s) = \big(i\Omega k^\pm+ \cO(s^2)\big) s^{-A+i\Omega}.
\end{equation}

We remind the reader that the equation for $\zeta^\pm_{1,2}$ is only \emph{real} linear, so the constants \(\Omega k_1^\pm\) and \(i\Omega k_2^\pm\) above can only be adjusted by a real multiple.\medskip

We do not derive the asymptotics of the complex function $\zeta$ as $s\to\infty$, because we find it more convenient to deal directly with the real functions $\chi^\pm$ in that region.

\section{The B\"ohm stationary soliton}

As a step in their construction of expanding solitons, Gastel and Kronz prove existence of a complete steady ($\lambda=0$) soliton $\bgk(t)$ solving~\eqref{eq-BlowUpSystem-Polynomial}.  This Ricci-flat soliton, whose existence follows from a more general result obtained earlier by B\"ohm \cite{Bohm}, plays a central role in their construction as well as ours, so we outline the proof of~\cite{GK04} here using our notation.

\subsection{The Gastel--Kronz construction of the B\"ohm stationary soliton}
\label{thm-BGK}\itshape
The unstable manifold $\wug$ and the stable manifold $\wsr$ intersect in exactly one orbit \(\{\bgk(t)\mid t\in\R\}\) of the soliton flow $g^t$.  The metric corresponding to this connecting orbit from $\gf$ to $\rfc$ is Ricci-flat.  It is thus a steady soliton and an Einstein metric.  \upshape

\begin{proof}[Sketch of the construction]
  Since $\sigma' = 2\sigma$, all connecting orbits between fixed points of the soliton flow must lie in the hyperplane $\sigma=0$.  In that subspace, we have $J'=2J$, so that $J$ also must vanish along connecting orbits.  All connecting orbits are therefore contained in the submanifold of $\R^6$ defined by $J=\sigma=0$.  On this submanifold, it follows from (\ref{eq-ivey-F-evolution}b) that the quantity $F$ defined in \eqref{eq-FJdefined} satisfies $F' =(\Gamma+1)F$.

  As $t\searrow-\infty$, one has $F\rightarrow0$ and $\Gamma+1\rightarrow1-p_\alpha<0$, which forces $F=0$ everywhere, \emph{i.e.,} that any connecting orbit lies in the invariant submanifold $\mc E$, a fact that we use below.  Furthermore, we also note that $J=0$ implies that
  \[
    (\Gamma+1)' =\tsum_\alpha p_\alpha(1+y_\alpha)^2 + \Gamma =\Gamma^2 + \Gamma - \tsum_\alpha p_\alpha x_\alpha \leq\Gamma(\Gamma+1).
  \]
  Since $\Gamma+1=-p_\alpha+1 < 0$ at $t=-\infty$, it follows that $\Gamma+1<0$ on the entire orbit.

  Thus far, we have shown that any connecting orbit from $\gf$ to $\rfc$ must lie in the portion of the submanifold $\mc E$ of $\R^6$ on which $\Gamma < -1$.  We now show that this region contains exactly one such orbit.

  The unstable manifold $\wug$ is three-dimensional, but the part of this manifold that lies in $\mc E$ is one-dimensional.  There are therefore two orbits that emanate from $\gf$ and that can converge to $\rfc$ as $t\to\infty$.

  According to Lemma~\ref{lem-GF-eigen-data}, the unstable eigendirection in $\mc E$ at $\gf$ is given by
  \[
    p_2\vE_1 - 3p_1\vE_2 = -p_2(p_1-1)\frac{\pd}{\pd x_1} + p_2\frac{\pd}{\pd y_1} -3p_1(p_1+1)\frac{\pd}{\pd x_2} -3p_1\frac{\pd}{\pd y_2} + 2p_1p_2\frac{\pd}{\pd \Gamma}.
  \]
  In particular, it has a nonzero component in the $x_2$ direction.  Since $x_2=0$ at $\gf$, we see that only one of the two orbits on the unstable manifold of \(\gf\) lies in the region \(x_2 > 0\).  This is the orbit we now consider.

  To show convergence to the $\rfc$ fixed point as $t\nearrow\infty$, we follow Gastel and Kronz by considering the Lyapunov function
  \[
    W=\tfrac12\tsum_\alpha p_\alpha\big\{x_\alpha-(n-1)\log(x_\alpha)+y_\alpha^2\big\}.
  \]
  Along any orbit of the soliton flow, we have
  \[
    W'=\tsum_\alpha p_\alpha\big\{(\Gamma+n)y_\alpha+(\Gamma+1)y_\alpha^2\big\} \leq(\Gamma+n)\tsum_\alpha p_\alpha y_\alpha =-(\Gamma+n)^2 < 0.
  \]
  In the final step, we used the fact that the quantity $F$ defined in~\eqref{eq-FJdefined} satisfies $F=0$ along the orbit.  It follows that $W$ is monotonically decreasing along the orbit and achieves its minimum at the stationary solution $\rfc=(n-1, 0, n-1, 0, -n ,0)$.

  To explain the claim that these metrics are in fact Ricci-flat, we recall that if
  \[
    \bgk(\tau) = (x_\alpha(\tau), y_\alpha(\tau), \Gamma(\tau), 0)
  \]
  is the B\"ohm soliton that results from the construction above, then according to~\eqref{eq-stationary-soliton},
  \[
    \mf p(\tau) = (x_\alpha(\tau), y_\alpha(\tau), \Gamma(\tau), e^{2\tau})
  \]
  is a solution of~\eqref{eq-BlowUpSystem-Polynomial} for \(\lambda=0\).  For this solution, we also have \(F(\tau) = \Gamma(\tau)+\sum_\alpha p_\alpha (1+y_\alpha(\tau)) \equiv 0\) for all \(\tau\in\R\).  The soliton vector field $\mathfrak X=f\pd/\pd s$ (which generates the diffeomorphisms by which the soliton metric flows) satisfies $sf = F-\lambda\sigma \equiv 0$, as follows easily from~\eqref{eq-define-Gamma}.  Hence we have both \(\mf X=0\) and \(\lambda=0\), so that the B\"ohm soliton \(\bgk\) is Ricci-flat.
\end{proof}

\section{Transversality of stable and unstable manifolds}

\label{sec-Transverse}

We have shown so far that the B\"ohm steady soliton $\bgk$ is the unique orbit in the intersection of the unstable and stable manifolds of the $\gf$ and $\rfc$, respectively.  This section is devoted to a proof of the fact that the intersection is transverse, \emph{i.e.,}
\begin{equation}\label{eq-Wu-Ws-transverse}
  \wug \transv \wsr .
\end{equation}
\subsection{Defining equations for $\wsr$}
\itshape If $\fp\in \wsr$, then $J=\sigma=0$ at $\fp$, and there is a small neighborhood $\mc N\subset \R^6$ of $\fp$ such that
\begin{equation}\label{eq-WsRFC-locally-defined}
  \wsr \cap\mc N
  = \bigl\{\tilde{\fp}\in\mc N : \sigma(\tilde{\fp}) = J(\tilde{\fp}) = 0\bigr\}.
\end{equation}
Moreover, the differentials $\dd \sigma$ and $\dd J$ are linearly independent along $\wsr$.\upshape

\begin{proof}
  Consider any point $\fp_0\in \wsr$, and let $\fp(t)$ be the orbit of the soliton flow starting at $\fp(0) = \fp_0$.  By definition, $\fp(t) \to\rfc$ as $t\to\infty$, so that $\sigma' = 2\sigma$ implies that we must have $\sigma\big(\fp(t)\big)=0$ for all $t\in\R$.  Since $\sigma$ vanishes along the orbit $\fp(t)$, it follows from equations~\eqref{eq-ivey-F-evolution} that $J' = 2J$ along the trajectory $\fp(t)$; since $J\big(\fp(t)\big)$ remains bounded as $t\to\infty$, this then implies that $J(\fp(t))\equiv 0$.  It follows that $\wsr$ is contained in the joint zero-set of $\sigma$ and $J$.  By definition of $J$, this joint zero-set is given by the equations
  \[
    \sigma=0,\qquad \tsum_\alpha p_\alpha x_\alpha = \Gamma^2 - 2\tsum p_\alpha y_\alpha(1+y_\alpha),
  \]
  from which it is clear that $\sigma^{-1}(0) \cap J^{-1}(0)\subset\R^6$ is a smooth embedded submanifold.

  In \S~\ref{sec-Lin-RFC}, we found that the $\rfc$ is a hyperbolic fixed point of the soliton flow with four stable eigenvalues, so general \textsc{ode} theory implies that the stable manifold $\wsr$ is a real analytic immersed submanifold of $\R^6$ of codimension two.  As it is contained in the four dimensional embedded submanifold $\sigma^{-1}(0) \cap J^{-1}(0)$, it must be a relatively open subset of $\sigma^{-1}(0) \cap J^{-1}(0)$, which we claimed in~\eqref{eq-WsRFC-locally-defined}.

  To complete the proof, we note that
  \[
    \dd J = \tsum_\alpha \tfrac12 p_\alpha\dd x_\alpha + \tsum_\alpha p_\alpha (1+2y_\alpha)\dd y_\alpha - \Gamma\dd\Gamma.
  \]
  The $\dd x_\alpha$ components of $\dd J$ are thus everywhere nonzero, so that $\dd\sigma$ and $\dd J$ are indeed linearly independent everywhere.
\end{proof}

\subsection{Proof of transversality}

\label{sec-transversality-proof}

The intersection $\wug \cap \wsr$ contains exactly one orbit, namely the B\"ohm steady soliton $\bgk$.  The tangent spaces to invariant manifolds for the soliton flow $g^t$ are themselves invariant under the tangent flow $\dd g^t$, so we only have to prove that $\wug$ and $\wsr$ meet transversally at one point $\bgk(t)$ along the orbit.  Since we have a global description of $\wsr$ as an open subset of $\sigma^{-1}(0)\cap J^{-1}(0)$, while we only have local description of $\wug$, it will be to our advantage if we consider a point $\bgk (t)$ that is very close to $\wug$.

We found in \S~\ref{sec-Lin-GF} --- see \eqref{eq-GF-unstable-eigenvectors} --- that the tangent space to $\wug$ at $\gf$ is spanned by the vectors
\begin{gather*}
  \vE_1 = -(p_1-1)\frac{\pd}{\pd x_1} + \frac{\pd}{\pd y_1} + 2p_1 \frac{\pd}{\pd \Gamma},
  \\
  \vE_{2} = (p_1+1)\frac{\pd}{\pd x_2} + \frac{\pd}{\pd y_2},\quad \text{ and }\quad \vE_\sigma = -\lambda\frac{\pd}{\pd x_2} + \frac{\pd}{\pd\sigma}\,.
\end{gather*}
At the $\gf$ fixed point, we have
\begin{align*}
  \dd\sigma\cdot \vE_\sigma & = 1,
  & \dd\sigma\cdot \vE_1    & =0,       \\
  \dd J\cdot \vE_\sigma     & = -\frac{\lambda} {2} p_2,
  & \dd J\cdot \vE_1        & = \tfrac32 p_1(p_1+1)\,.
\end{align*}
This implies that at the fixed point $\gf$, the unstable manifold $\wug$ intersects the submanifold $\sigma^{-1}(0)\cap J^{-1}(0)$ transversally.  By continuity of the tangent spaces, $\wug\transv \sigma^{-1}(0)\cap J^{-1}(0)$ holds in a neighborhood $\mc U$ of $\gf$.  If we choose $t\in\R$ sufficiently close to $-\infty$, then $\bgk(t)\in \mc U$, and it follows that $\wug\transv\wsr$ at $\bgk(t)$, because near $\bgk(t)$, the submanifolds $\wsr$ and $\sigma^{-1}(0)\cap J^{-1}(0)$ coincide.  This completes the proof of \eqref{eq-Wu-Ws-transverse}.

\section{Behavior of smooth solutions near $\bgk \cap \rfes = \rfc$ }

\subsection{The unstable manifold of the $\rfc$}
\label{sec-WuRFC} In \S~\ref{sec-Lin-RFC}, we saw that the eigenvalues of the linearization $\dd X_\rfc$ are $\{-(n-1), -A\pm i\Omega,-1, +2, +2\}$.  So the $\rfc$ is hyperbolic and has a two-dimensional unstable manifold.  Just as in the case of the $\gf$ fixed points, there are again no resonances between the unstable eigenvalues, so the flow $g^t|\wur$ is analytically conjugate to the linear flow $\bs x\mapsto e^{2t}\bs x$ on $\R^2$.

Those orbits in the unstable manifold $\wur$ that extend to $0<\sigma<\infty$ correspond to metrics on $(0, \infty) \times \mc S^{p_1}\times \mc S^{p_2}$ that define expanding or shrinking Ricci flow solitons with a singularity at $\sigma=0$ that is asymptotically like the Ricci-flat cone.

\begin{lemma}
  The unstable manifold $\wur$ is contained in the equal radii subspace $\equalradii$, \emph{i.e.,} the difference variables $x_{12}, y_{12}$ vanish on $\wur$.
\end{lemma}
\begin{proof}
  As we observed in \S~\ref{sec-equal-radii}, the equal radii subspace $\equalradii \subset\R^6$ consisting of all points with $x_{12}=y_{12}=0$ is invariant under the soliton flow $g^t$.  This subspace contains the fixed point $\rfc$, and the unstable eigenvectors of the linearization of the flow at $\rfc$ are tangential to $\equalradii$.  Therefore the flow $g^t|\equalradii$ has a smooth two-dimensional unstable manifold at $\rfc$.  Since the unstable manifold at $\rfc$ of the flow $g^t$ on the whole space $\R^6$ also is a two-dimensional immersed submanifold, it must coincide with the unstable manifold of $g^t|\equalradii$ at $\rfc$.  Hence $\wur\subset\equalradii$.
\end{proof}

\subsection{The tangent space \(T_{\rfes}\wur\)}
In a small neighborhood of the \(\rfc\), the variables \((\sigma, J)\) form analytic coordinates on \(\wur\), so that we have a parametrization \(\fp=\fw^{u}(\sigma, j)\) of a small neighborhood of \(\rfc\) in \(\wur\).

The tangent space to the \(\rfes\) is therefore the range of \(\dd\fw^u\), which is spanned by the partial derivatives \(\fw^{u}_{\sigma}(\sigma, 0)= \frac{\pd\fw^u}{\pd\sigma}\) and \(\fw^{u}_{j}(\sigma,0)= \frac{\pd\fw^u}{\pd j}\).  Of these, the first derivative is
\[
  \fw^u_\sigma = \frac{1}{2s} \frac{\pd}{\pd s},
\]
which is a multiple of the tangent vector \(\frac\pd{\pd s}\) to the \(\rfes\).

To get a second tangent vector to the \(\rfes\), we choose a particular small value of \(s\), say \(s=a\), and let \(W(a) = \fw^u_j(a^2,0)\).  Since \(\wur\) is invariant under the soliton flow, the vector \(\dd g^t\cdot W(a)\) is tangential to \(\wur\) at \(g^t(\fw^u(a^2,0)) = \fw^u\bigl((e^{t}a)^2, 0\bigr)\).  We therefore define for any \(s>0\),
\[
  W(s) = \dd g^{\log(s/a)}\cdot W(a).
\]
Then \(W(s) \in T_\rfes\wur\), \(W(s)\neq0\), and \(\dd s\cdot W(s) = 0\) for all \(s>0\).  This implies that \(\{\frac\pd{\pd s}, W(s)\} \) is a basis for \(T_\rfes\wur\) at each point.  Moreover, \(W(s)\) is a solution of the linearization \eqref{eq-three-by-three-subsystem} of the system for the averaged variables \((\xi, y, \gamma)\).  Thus
\[
  W(s) = c_0\Phi_0(s) + c_1^\pm\Phi_1^\pm(s) + c_2^\pm\Phi_2^\pm(s)
\]
for some constants \(c_j^\pm\).

We also know that \(W(s)\to0\) as \(s\to0\), because \(\dd g^t\) decays exponentially on \(\wur\) as \(t\to-\infty\).  Both \(\Phi_0(s)\) and \(\Phi_1^\pm(s)\) become unbounded as \(s\to0\), so we must have \(c_0=c_1^\pm=0\).  Hence \(W(s)\) is a nonzero multiple of \(\Phi_2^\pm(s)\).

A consequence of this is that
\begin{equation}
  \label{eq-wu-expansion-at-a}
  \fw^u(a^2, j) = c_2^\pm j \Phi_2^\pm(a) + \cO(j^2),\qquad
  (j\to 0).
\end{equation}
Therefore, \(\fw^u_j(a^2,0) = c_2^\pm \Phi_2^\pm(a)\).  Moreover, because \(\fw^u\) is a real analytic function, we have
\begin{equation}
  \label{eq-wu-expansion-at-a-deriv}
  \fw^u_j(a^2, j) = c_2^\pm  \Phi_2^\pm(a) + \cO(j),\qquad
  (j\to 0).  
\end{equation}

\begin{figure}[h]
  \centering\sffamily \includegraphics[width=0.7\textwidth]{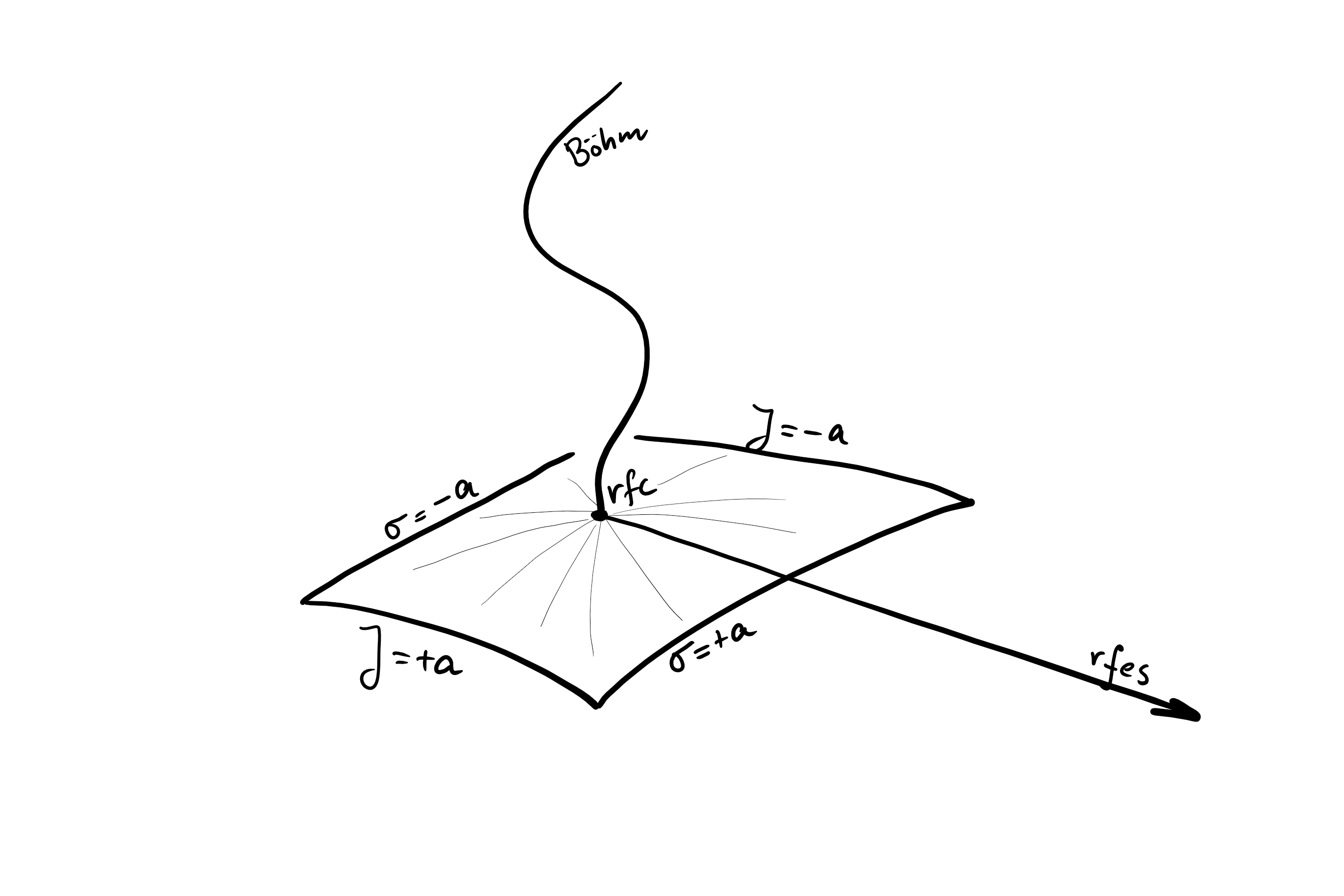}
  \caption{\(\wur\) near \(\rfc\)}
  \label{fig:wur-near-rfc}
\end{figure}

\subsection{An isolating block for \(\rfc\)}
Since the \(\rfc\) is a hyperbolic fixed point, we can write \(\R^6=E^u\oplus E^s\), where \(E^u\) and \(E^s\) are the unstable and stable subspaces of \(\R^6\), respectively, for the linearization \(\dd X_\rfc\).  Let \(\pi^s:\R^6\to E^s\), \(\pi^u:\R^6\to E^u\) be the projections onto those invariant subspaces, with null spaces \(N(\pi^u) = E^s\) and \(N(\pi^s) = E^u\).  Since the stable eigenvalues of \(\dd X_\rfc\) are \(\{-1,-A\pm i\Omega, -(n-1)\}\), one can find\footnote{ If \(V_1\), \(V_2\pm iV_3\), and \(V_4\) are eigenvectors of \(dX_\rfc\) corresponding to the eigenvalues \(-1\), \(-A\pm i\Omega\), and \(-n+1\), then define \(\langle\cdot,\cdot\rangle_s\) by declaring \(\{V_1, V_2, V_3, V_4\}\) to be orthonormal.}  an inner product \(\langle \cdot,\cdot\rangle_s \) on \(E^s\) such that
\[
  \langle v, \dd X_\rfc\cdot v\rangle_s \leq -\langle v, v\rangle_s
\]
for all \(v\in E^s\).

By the Grobman--Hartman theorem \cite{Gro62,Ha60}, there is a neighborhood \(U\subset\R^6\) of the \(\rfc\) on which the flow \(g^t\) is topologically conjugate to the linearized flow near the origin.  In particular, we can choose the neighborhood \(U\) so that any point \(\fp\in U\) with \(g^{[0,\infty)}(\fp) \subset U\) belongs to \(\wsr\).  Similarly, any \(\fp\in U\) with \(g^{(-\infty,0]}(\fp)\subset U\) belongs to \(\wur\).

For small \(a>0\), we now consider the compact neighborhood of the \(\rfc\) defined by
\begin{equation}
  \cQ_a = \left\{
    \fp\in\R^6 \mid
    |\sigma(\fp)|\leq a^2,\  |J(\fp)|\leq a^2, \ \|\pi^s (\fp - \rfc)\|_s\leq a^2
  \right\},
\end{equation}
where \(\|v\|_s = \sqrt{\langle v, v\rangle_s }\).

For sufficiently small \(a>0\), \(\cQ_a\subset U\), and \(\cQ_a\) is an \emph{isolating block} in the sense of Conley and Easton \cite{CE71,conley1978isolated}: for any \(\fp\in\pd \cQ_a\), one has
\begin{align}
  |\sigma(\fp)| = a^2 &\implies \frac{d}{dt}|\sigma(g^t\fp)|\Big|_{t=0} > 0,\\
  |J(\fp)| = a^2 &\implies \frac{d}{dt}|J(g^t\fp)|\Big|_{t=0} > 0,\\
  \|\pi^s(\fp - \rfc)\|_s = a^2
                      &\implies \frac{d}{dt}\left\|\pi^s\bigl(g^t(\fp) - \rfc\bigr)\right\|_s\Big|_{t=0}<0.
\end{align}
Thus the flow has no internal tangencies to the boundary \(\pd \cQ_a\), \emph{i.e.,} for every \(\fq\in\pd Q_a\), there exists \(\tau_0>0\) such that \(g^{(0,\tau_0)}(\fq) \subset \R^6\setminus \cQ_a\) or \(g^{(-\tau_0,0)}(\fq) \subset \R^6\setminus \cQ_a\) (or both).  It follows that for every \(\fp\in \cQ_a\), the exit time
\[
  T_a(\fp) \stackrel{\rm def}{=} \inf \{t>0 \mid g^t(\fp) \not\in \cQ_a\}
\]
is well defined.  If \(g^t(\fp)\in \cQ_a\) for all \(t\geq0\), then we define \(T_a(\fp) = +\infty\).  We note that the exit time \(T_a : \cQ_a\to [0,\infty]\) is a continuous function.

\subsection{A slice within \texorpdfstring{\(\wug\)}{Wu(GF)} transverse to the \texorpdfstring{\(\bgk\)}{B\"ohm} soliton}
\label{sec-slice-chosen}
Let \(\bgk(t_a)\) be a point on the B\"ohm soliton with \(t_a\) so large that \(\bgk(t_a)\) lies in the interior of \(\cQ_a\).  Since \(\wug\transv\wsr\), there is a two-dimensional slice \(\Sigma_a:\mathbb D^2\hookrightarrow \wug\cap\cQ_a\) with \(\Sigma_a(0) = \bgk(t_a)\), and such that \(\Sigma_a\transv \wsr\).  We will abuse notation and identify the map \(\Sigma_a\) with its image, and write \(\Sigma_a\) for both.

We choose \(\Sigma_a\) so small that the only intersection of \(\Sigma_a\) and \(\wsr\) is the point \(\bgk(t_a)\) on the B\"ohm soliton.  It follows that \(T_a(\fp) < \infty\) for all \(\fp\in\Sigma_a\setminus\{\bgk(t_a)\}\).

Orbits in \(\cQ_a\) exit either through the sides \(\sigma=\pm a^2\), the sides \(J=\pm a^2\), or through the edges on which \(\sigma=\pm a^2\) and \(J=\pm a^2\) both hold.  If the orbit starting at \(\fp\in\Sigma_a\) exits through a point on \(\pd\cQ_a\) with \(\sigma=+a^2\), then we have a simple expression for the exit time in terms of the \(\sigma\) coordinate of \(\fp\).  Namely, since \(a^2 = \sigma\big(g^{T_a(\fp)}\fp\big) = e^{2T_a(\fp)}\sigma(\fp)\), we have
\[
  T_a(\fp) = \log a - \frac12 \log \sigma(\fp).
\]

\begin{figure}[h]
  \centering \includegraphics[width=0.7\textwidth]{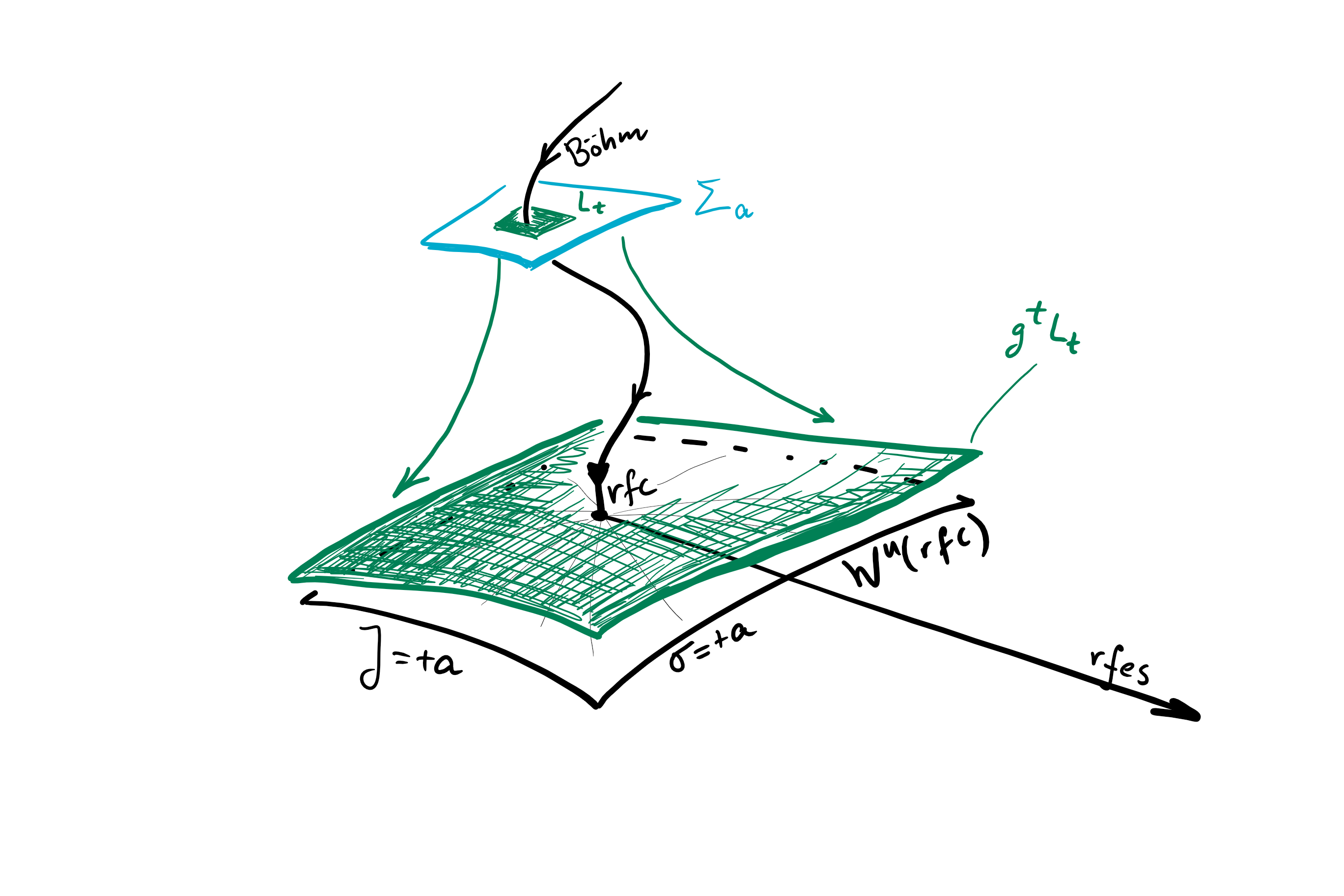}
  \caption{Application of the \(\lambda\)-lemma to a slice \(\Sigma_a\) transverse to \(\bgk\).  For all large enough \(t\), there is a subset \(L_t\subset\Sigma_a\) such that \(g^tL_t\) is \(C^1\) close to the given rectangle \(\wur\cap\{|\sigma|\leq a^2, |J|\leq a^2\}\).}
  \label{fig-apply-lambda}
\end{figure}

\subsection{Application of the \texorpdfstring{\(\lambda\)}{lambda}-lemma}
\label{sec-apply-lambda}
In general, if \(g^t\) is a smooth local flow on a manifold \(\mc M\) with a hyperbolic fixed point \(p\in \mc M\), then the flow has smooth immersed stable and unstable manifolds \(W^s(p)\) and \(W^u(p)\), respectively.  Jacob Palis' \(\lambda\)-lemma \cite{Pal68} states that if a smooth submanifold \(\Sigma_a\subset \mc M\) with the same dimension as \(W^u(p)\) is transverse to \(W^s(p)\), and if \(K\subset W^u(p)\) is compact, then for any \(\eta>0\) and all sufficiently large \(t>0\), there is a compact subset \(L_t\subset \Sigma_a\) such that \(g^t(L_t)\) is \(\eta\)-close to \(K\) in \(C^1\).

We intend to apply this lemma to the slice \(\Sigma_a\) and the fixed point \(\rfc\).  The \(\lambda\)-lemma as we have just stated it is a local result, so we may consider only those parts of \(\wur\) and \(g^t\Sigma_a\) that remain in the neighborhood \(\cQ_a\).

For small \(a>0\), the intersection of \(\wur\) with \(\cQ_a\) is
\[
  \mc R_a^u \stackrel{\rm def}{=} \wur \cap \cQ_a = \{\fp\in\wur : |\sigma(\fp)|\leq a^2, \ |J(\fp)|\leq a^2\},
\]
which we will take as our compact set $K$ and whose boundary is diffeomorphic to a rectangle
\[
  \wur\cap\pd\cQ_a = \{\fp\in\wur : \sigma=\pm a^2, |J|\leq a^2\} \cup \{\fp\in\wur : J=\pm a^2, |\sigma|\leq a^2\}.
\]
Using our analytic local parametrization \(\fw^u\) of \(\wur\), we parametrize the edge of \(\wur\cap\pd\cQ_a\) on which \(\sigma=+a^2\) by an analytic curve
\[
  \wur\cap\pd\cQ_a\cap\{\sigma=a^2\} = \{\fw^u(a^2, j) \mid -a^2\leq j\leq a^2\}.
\]

For large \(T\) and small \(a>0\), we consider the intersection \(g^T\Sigma_a \cap \partial \cQ_a\), or more precisely,
\[
  \cC_{a,T} = \{g^T\fp \mid \fp\in\Sigma_a\text{ and } T=T_a(\fp)\}.
\]
The \(\lambda\)-lemma implies that for any \(\eta>0\), there is a \(t_\eta>0\) such that \(\mc C_{a,T}\) is \(\eta\)-close in \(C^1\) to \(\cR_a^u\) for all \(T\geq t_\eta\).  In fact, \(\cC_{a,T}\) is a graph of a \(C^1\) function over \(\cR_a^u\).

\begin{lemma}\label{lem-qjT-defined-lambda-applied}
  For all \(j\in[-a^2,a^2]\) and sufficiently large \(T>0\), there exists a unique \(\fq(j, T) \in\wug\cap\pd\cQ_a\) with
  \[
    J(\fq(j, T)) = j, \qquad \sigma(\fq(j, T)) = a^2, \qquad g^{-T}\fq(j, T) \in \Sigma_a.
  \]
  The map \((j, T)\mapsto \fq(j, T)\) is smooth.

  As \(T\to\infty\), one has
  \[
    \fq(j, T)\to \fw^u(a^2, j) \qquad\text{ and }\qquad \frac{\pd\fq(j,T)}{\pd j} \to \frac{\pd\fw^u(a^2,j)}{\pd j}
  \]
  uniformly for \(|j|\leq a^2\).
\end{lemma}
\begin{proof}
  This follows from the \(\lambda\)-lemma.
\end{proof}
The map \((j,T)\mapsto \fq(j,T)\) is well-defined and smooth for \(|j|\leq a^2\) and \(T_0\leq T<\infty\), if \(T_0\) is sufficiently large.  If we now extend \(\fq\) by setting
\[
  \fq(j,\infty) = \fw^u(a^2, j),
\]
then \(\fq:[-a^2, a^2]\times [T_0,\infty]\to\R^6\) is a continuous map.

This lemma provides an approximation for the orbits in \(\wug\) that pass close to the \(\rfc\) as they exit the isolating block \(\cQ_a\).  The description can be improved by considering the \((x_{12}, y_{12})\)-component of \(\fq(j,T)\).  While it is much smaller than \(\fq(j, T) - \fw^u(a^2, T)\), it satisfies a set of homogeneous linear equations, and this allows us to give a precise estimate for \((x_{12}, y_{12})\) at exit points in \(\wug\cap\pd \cQ_a\).

We therefore consider orbit segments \(\{g^t\fp\}\) of the flow that spend a long time in the block \(\cQ_a\).  Our first observation is that such orbits must be close to the stable and unstable manifolds of the fixed point \(\rfc\).

\begin{lemma}
  \label{lem-close-to-invariant-mfds}
  Let \(\fp_i\in \Sigma_a\) be a sequence of points with \(\fp_i\to\bgk(t_a)\).  Define \(T_i = T_a(\fp_i)\) and assume that \(\fq_i=g^{T_i}(\fp_i)\) converges to \( \fq\in \pd \cQ_a\) .  Then the orbit segments starting at \(\fp_i\) and ending at the exit point \(\fq_i\) converge in the Hausdorff metric:
  \[
    g^{[0,T_i]}(\fp_i) \longrightarrow \bgk\left([t_a, \infty)\right) \cup \{\rfc\} \cup g^{(-\infty, 0]}(\fq).
  \]
\end{lemma}
\begin{proof}
  Let \(\mf m_i = g^{t_i}\fp_i\) be a subsequence of points on the orbit \(g^{[0,T_i]}(\fp_i)\) for which \(\mf m = \lim \mf m_i\) exists.  We will show that \(\mf m\in \bgk\left([t_a, \infty)\right) \cup \{\rfc\} \cup g^{(-\infty, 0]}(\fq)\).

  If \(t_i\) is bounded, then we may assume that \(t_i\to t_*\) for some \(t_*\geq 0\), and thus \(\mf m_i = g^{t_i}\fp_i \to g^{t_*}\bgk(t_a)\).  In this case we have \(\mf m\in\bgk([t_a, \infty))\).

  If on the other hand \(t_i' = T_i-t_i\) is bounded, then we may assume that \(-t_i'\to t_*\geq 0\), and we find that \(g^{t_i}\fp_i = g^{-t_i'}\fq_i \to g^{-t_*}\fq\).  In this case \(\mf m\in g^{(-\infty, 0]}\fq\).

  The remaining possibility is that both \(t_i\to\infty\) and \(T_i-t_i\to\infty\).  In this situation we have \(g^{[-t_i, T-t_i]}(\mf m_i) \subset \cQ_a \) for all \(i\geq1\) so that in the limit we find that the entire orbit \(\{g^t\mf m\mid t\in\R\}\) lies in \(\cQ_a\).  The only \(\mf m\in\cQ_a\) with this property is \(\mf m=\rfc\).
\end{proof}

\begin{lemma}
  \label{lem-zetajT-at-exit-of-Qa}
  There exist a constant \(C_*>0\) and a continuous complex-valued function \(\epsilon_a(j, T)\) such that
  \begin{equation}
    \label{eq-zetajT-at-exit-of-Qa}
    \zeta\bigl(\fq(j, T)\bigr) = 
    (1+\epsilon_a(j,T)) e^{-(A-i\Omega)T} \zeta\bigl(\bgk(t_a)\bigr),
  \end{equation}
  where
  \[
    |\epsilon_a(j,T)|\leq C_*a^2
  \]
  for \(j\in[-a^2, a^2]\), all sufficiently large \(T<\infty\), and all small enough \(a>0\).
\end{lemma}
\begin{proof}
  Equation \eqref{eq-zetajT-at-exit-of-Qa} effectively defines the function \(\epsilon_a(j,T)\), so we really only have to prove that \(|\epsilon_a|\leq C_* a^2\) for some constant \(C_*\)

  Consider \(\fp(j,T) = g^{-T}(\fq(j,T)) \in \Sigma_a\).  Then the orbit segment \(g^{[0,T]}(\fp(j,T))\) is contained in \(\cQ_a\), and if \(T\) is sufficiently large, then Lemma~\ref{lem-close-to-invariant-mfds} implies that \(\fp(j, T)\) will be close to \(\bgk(t_a)\).

  By rewriting the difference system \eqref{eq-difference-system-matrix-form} in terms of the complex quantity \(\zeta\), we find that along the orbit segment \(g^{[0,T]}(\fp(j, T))\), one has
  \[
    \zeta' = M\bigl(g^t\fp(j,T)\bigr) \zeta + N\bigl(g^t\fp(j,T)\bigr)\bar\zeta,
  \]
  where \(M, N\) are complex-valued functions that satisfy
  \[
    M(\fp) = -A+i\Omega + \cO(\|\fp-\rfc\|) \qquad\mbox{ and }\qquad N(\fp) =\cO\bigl(\|\fp-\rfc\|\big).
  \]
  Since \(\fp(j,T)\) is close to \(\bgk(t_a)\), the \(\zeta\) component \(\zeta\bigl(\fp(j, T)\bigr)\) is close to \(\zeta(\bgk(t_a))\).

  For \(t\geq 0\), we then have
  \begin{equation} \label{eq-log-zeta-growth} \frac{d}{dt}\log \zeta\bigl(g^t\fp(j,T)\bigr) = -A+i\Omega + \cO(\|g^t\fp(j,T) - \rfc\|).
  \end{equation}
  The Analysis Lemma~\ref{lem-AnalysisLemma} in Appendix~\ref{AnalysisAppendix} implies that there is a constant \(C>0\) that does not depend on \(a\) such that
  \[
    \int_0^T \left\|g^t\fp(j,T) - \rfc \right\|\,\mr dt \leq C \sup_{0\le t\le T} \left\|g^t\fp(j,T) - \rfc \right\| \leq Ca^2,
  \]
  for all sufficiently small \(a\), all \(j\in[-a^2, a^2]\), and all sufficiently large \(T\) (depending on \(a\)).

  Upon integrating \eqref{eq-log-zeta-growth}, we find that
  \[
    \log\frac{\zeta\bigl(\fq(j, T)\bigr)}{\zeta\bigl(\fp(j, T)\bigr)} = -(A-i\Omega)T + \cO(a^2).
  \]
  For large enough \(T\), the initial point \(\fp(j, T)\) is arbitrarily close to \(\bgk(t_a)\). So there exists \(T_a<\infty\) such that for \(T\geq T_a\), we have
  \[
    \left| \log\frac{\zeta\bigl(\fp(j, T)\bigr)} {\zeta\bigl(\bgk(t_a)\bigr)} \right| \leq a^2.
  \]
  Combining the last two estimates then shows us that
  \[
    \left| \log\frac{\zeta\bigl(\fq(j, T)\bigr)} {\zeta\bigl(\bgk(t_a)\bigr)} +(A-i\Omega)T \right| \leq Ca^2,
  \]
  which implies that \(|\epsilon(j,T)|\leq C_*a^2\) for some constant \(C_*\).
\end{proof}

\section{Continuation of the soliton flow from \(s=a\) to \(s\sim \ell\)}
\label{sec:s-goes-to-ell}
\subsection{Definition of \(\fq(s;j,T)\)}
We consider orbits \(\{g^t\bigl(\fq(j, T)\bigr) \mid t\geq 0 \}\) of the soliton flow for sufficiently small \(j\) and sufficiently large \(T\), and follow them until they reach the hyperplane \(s=\ell\) for some fixed large \(\ell\gg a\).  For \(t>0\), the \(s\)-coordinate of \(g^t\bigl(\fq(j,T)\bigr)\) is \(s=e^t a\), so we consider the orbit \(g^t\bigl(\fq(j, T)\bigr)\) for \(t\in[0, \log\frac{\ell}{a}]\).  We simplify notation by writing
\begin{equation}
  \label{eq-qsjT-defined}
  \fq(s; j,T) \isdef g^{\log(s/a)}\bigl(\fq(j,T)\bigr).
\end{equation}
Similarly we write
\[
  \fq(s;j,\infty) \isdef g^{\log(s/a)} \bigl(\fq(j,\infty)\bigr)
\]
for the result of flowing \(\fq(j,\infty) = \fw^u(a^2,j)\) by \(\log s/a\).

\begin{lemma} \label{lem-qsjt-exists} There exist \(\iota_{a,\ell}>0\) and \(T_{a,\ell}<\infty\) such that \(\fq(s;j,T)\) is defined for all \(j\in(-\iota_{a,\ell}, \iota_{a,\ell})\), \(T\in[T_{a,\ell},\infty]\), and \(s\in[a,\ell]\).  Furthermore, the limit
  \[
    \fq(s;j,\infty)=\lim_{T\to\infty} \fq(s;j, T)
  \]
  exists $C^1$-uniformly for \(|j|\leq \iota_{a,\ell}\) and \(a\leq s\leq \ell\).
\end{lemma}

\begin{proof}
  For \(j=0\) and \(T=\infty\), we have
  \begin{equation}
    \label{eq-jzeroTinfty-is-rfec}
    \fq(s; 0,\infty) = (s, 0,0,0,0,0) = \rfes(s)
  \end{equation}
  in \((s, \xi, y, \gamma, x_{12}, y_{12})\) coordinates.  In particular, \(g^{\log s/a}(\fq(0,\infty))\) is defined for all \(s\geq a\).  Continuity of the local flow \(g^t\) implies that \(g^{\log s/a}(\fq(j,T))\) is defined for all \(s\in[a, \ell]\) for any fixed \(\ell>a\), provided \(T\) is large and \(|j|\) is small enough, say for \(|j|\leq \iota_{a,\ell}\) and \(T\geq T_{a,\ell}\).

  As \(T\to\infty\), we saw in Lemma~\ref{lem-qjT-defined-lambda-applied} that \(\fq(j, T) \to \fq(j, \infty) = \fw^u(a^2, j)\).  Since the flow \(g^t\) is smooth, it follows that \(\fq(s;j,T) \to \fq(s; j,\infty)\) as \(T\to\infty\).
\end{proof}

For \(|j|\leq \iota_{a,\ell}\), \(T\geq T_{a,\ell}\), and all \(s\geq\ell\), we write \(\Phi_*(s; j, T)\) and \( \zeta_*(s;j, T)\) for the averaged and difference components of \(\fq(s; j,T)\); in other words, we have
\begin{equation}
  \label{eq-between-a-and-ell}
  \fq(s; j,T) = \bigl(s, \Phi_*(s;j, T), \zeta_*(s;j, T)\bigr)
\end{equation}
in \((s, \xi, y,\gamma,\zeta)\) coordinates.  The coordinate functions \(\Phi_*:[a, \ell]\to\R^3\) and \(\zeta_*:[a, \ell]\to\C\) are solutions of the nonlinear equations~\eqref{eq-rfes-average-nonlin} and~\eqref{eq-rfes-difference-nonlin}, respectively.

\begin{lemma}\label{lem-Phi-at-ell}
  The limits
  \begin{align*}
    \lim_{T\to\infty}\Phi_*(\ell; j, T) &= \Phi\bigl( \fq(\ell;j,\infty) \bigr), \\
    \lim_{T\to\infty} \zeta_*(\ell; j, T) &= 0,
  \end{align*}
  exist \(C^1\)-uniformly for \(|j|\leq\iota_{a,\ell}\).  In particular,
  \[
    \lim_{T\to\infty}\Phi_*(\ell; j, T) = c_2^\pm j\Phi_2^\pm(\ell) + \cO(j^2), \qquad(j\to 0),
  \]
  where \(c_2^\pm\) are as in~\eqref{eq-wu-expansion-at-a}.
\end{lemma}
\begin{proof}
  We know that \(\fq(j, T)\to\fq(j, \infty) = \fw^u(a^2, j)\) as \(T\to\infty\).  Since \(g^t\) is continuous, the limiting values \(\Phi_*(s;j,\infty)\) and \(\zeta_*(s;j,\infty)\) follow.

  The second assertion follows from \eqref{eq-wu-expansion-at-a}.
\end{proof}

\subsection{Asymptotics of the difference variables}
Lemma~\ref{lem-Phi-at-ell} gives us a good approximation of \(\Phi_*(\ell;j,T)\) as \(T\to\infty\) for small \(j\).  But for \(\zeta_*\), it only says that \(\zeta_*\to0\), without providing more details.  To get more information, we use the matrix form~\eqref{eq-difference-system-matrix-form} of the homogeneous equation \eqref{eq-rfes-difference-nonlin} satisfied by \(\{x_{12}, y_{12}\}\),
\[
  s\frac{\dd}{\dd s}\mat x_{12} \\ y_{12}\rix= \cM(s; j,T) \mat x_{12}\\ y_{12}\rix,
\]
where
\[
  \cM(s; j,T) = \mat
  y+(p_2-p_1)y_{12} & -2(n-1)+\xi \\
  1 & -n+1+\gamma-\lambda s^2 \rix,
\]
and where \((\xi, y, \gamma, x_{12}, y_{12})\) are to be evaluated at \(\fq(s;j,T)\).

The coefficient matrix \(\cM(s; j,T)\) depends continuously on \((s, j, T) \in [a,\ell]\times[-\iota_{a,\ell},\iota_{a,\ell}]\times[T_{a,\ell},\infty]\), and thus as \(T\to\infty\) and \(|j|\to 0\), we get
\[
  \cM(s;j, T) \to \cM(s; 0,\infty) = \mat
  0 & -2(n-1) \\
  1 & -n+1+\gamma-\lambda s^2 \rix.
\]
We denote the fundamental solution starting at \(s=a\) of the linear system by \(\cU(s;j,T)\).  So by definition, \(\cU\) satisfies
\[
  s\frac{\dd}{\dd s} \cU(s; j, T) = \cM(s; j,T)\cU(s; j, T), \qquad \cU(a;j, T) = \mat 1 & 0 \\ 0 & 1 \rix,
\]
and the solution \(\tmat x_{12}(s;j,T)\\y_{12}(s;j,T)\trix\) is given by
\[
  \mat x_{12}(s;j,T)\\y_{12}(s;j,T)\rix = \cU(s;j,T) \mat x_{12}(a;j,T)\\y_{12}(a;j,T)\rix.
\]
Continuous dependence of solutions of differential equations on their parameters implies that as \(|j|\to 0\) and \(T\to\infty\), the fundamental solution \(\cU\) also converges, namely,
\[
  \cU(s; j, T) \to \cU(s ; 0, \infty),
\]
where \(\cU(s; 0, \infty)\) satisfies
\[
  s\frac{\dd}{\dd s} \cU(s; 0, \infty) = \cM(s; 0, \infty)\cU(s; 0, \infty), \qquad \cU(a;0, \infty) = \mat 1 & 0 \\ 0 & 1 \rix.
\]
This last system of equations for \(\tmat x_{12} \\ y_{12}\trix\) is equivalent to the second order equation~\eqref{eq-chi-second-order} for \(\chi=x_{12}\), for which we have found fundamental solutions \(\{\chi_1^\pm, \chi_2^\pm\}\) described in Tables~\ref{tab-expander-asymptotics},~\ref{tab-shrinker-asymptotics}.

We now have
\[
  \mat x_{12}(\ell;j,T)\\y_{12}(\ell;j,T)\rix = \cU(\ell;0,\infty) \cdot \cU(\ell;0,\infty)^{-1}\cU(\ell;j,T)\cdot \mat x_{12}(a;j,T)\\y_{12}(a;j,T)\rix.
\]
The solution at \(s=\ell\) is therefore given by
\[
  \mat x_{12}(\ell;j,T)\\y_{12}(\ell;j,T)\rix = \cU(\ell;0,\infty)\cdot \mat \tilde x_{12}(j,T) \\ \tilde y_{12}(j,T) \rix,
\]
where
\[
  \mat \tilde x_{12}(j,T) \\ \tilde y_{12}(j,T) \rix = \cU(\ell;0,\infty)^{-1}\cU(\ell;j,T)\cdot \mat x_{12}(a;j,T)\\y_{12}(a;j,T)\rix.
\]

In other words, the solution of the nonlinear differential system is given by the solution of the linearization at the \(\rfes\) with a slightly modified initial value at \(s=a\), namely \(\tmat \tilde x_{12}(j,T) \\ \tilde y_{12}(j,T) \trix\).  This modified initial value is obtained by multiplying the actual value of \(\tmat x_{12}(a;j,T) \\ y_{12}(a;j,T)\trix\) with the matrix \(\cU(\ell;0,\infty)^{-1}\cU(\ell;j,T)\), which converges to the identity as \(T\to\infty\) and \(j\to0\), \emph{i.e.,}
\[
  \lim_{j\to0} \lim_{T\to\infty} \cU(\ell;0,\infty)^{-1}\cU(\ell;j,T) = \mat 1&0\\0&1\rix.
\]

We next derive an asymptotic description of the actual initial value \(\tmat x_{12}(a;j,T)\\y_{12}(a;j,T)\trix\) from expression~\eqref{eq-zetajT-at-exit-of-Qa} for \(\zeta(\fq(j,T))\).  In doing so, we have to go back and forth between the real variables \((x_{12}, y_{12})\) and the complex variable \(\zeta\), which will be easier if we define a real-linear map \(\cC:\R^2\to \C\) and its inverse \(\cC^{-1}:\C\to\R^2\) by
\[
  \cC\mat x\\y \rix =(A+i\Omega)x - (A^2+\Omega^2)y\qquad\text{and}\qquad \cC^{-1}\zeta = \frac1\Omega \mat \Im \zeta \\ \Im (\zeta/(A+i\Omega)) \rix.
\]
We have then just shown that the \(\zeta\) component of the solution at \(s=\ell\) coincides with the solution to the linearized equation for \(\zeta\) with initial data at \(s=a\) given by
\[
  \tilde\zeta(j,T) = \cC \mat \tilde x_{12}(j,T) \\ \tilde y_{12}(j,T) \rix.
\]
Recall from equation~\eqref{eq-zeta-from-chi} that \(\zeta_1^\pm=(\chi_1^\pm)^\prime+(A+i\Omega)\chi_1^\pm\) and \(\zeta_2^\pm=(\chi_2^\pm)^\prime+(A+i\Omega)\chi_2^\pm\) are a pair of linearly independent solutions of the linearized \(\zeta\) equations, where the asymptotic behaviors of $\chi^\pm_{1,2}$ are given in Tables~\ref{tab-expander-asymptotics} and~\ref{tab-shrinker-asymptotics}. Thus we have
\begin{equation}
  \label{eq-zeta-is-linear-combination}
  \zeta(\ell;j,T) = c_1^\pm(j,T)\zeta_1^\pm(\ell) 
  + c_2^\pm(j,T)\zeta_2^\pm(\ell),
\end{equation}
where the \emph{real} constants \(c_i^\pm(j,T)\) should satisfy
\begin{align*}
  c_1^\pm(j,T)\zeta_1^\pm(a) 
  &+ c_2^\pm(j,T)\zeta_2^\pm(a)
    = \cC \mat \tilde x_{12}(j,T) \\ \tilde y_{12}(j,T) \rix \\
  &=\cC \cU(\ell;0,\infty)^{-1}\cU(\ell;j,T)\cdot
    \mat x_{12}(a;j,T)\\y_{12}(a;j,T)\rix   \\
  &=\cC \cU(\ell;0,\infty)^{-1}\cU(\ell;j,T)\cdot
    \cC^{-1} \zeta\bigl(\fq(j,T)\bigr)\\
  &= \bigl(I + o(1)\bigr) \cdot \zeta\bigl(\fq(j,T)\bigr).
\end{align*}
Here \(I\) is the \(2\times2\) identity matrix, and \(o(1)\) is a real-linear transformation of \(\C\).  It follows that there is a complex number \(\epsilon(j,T)\) with \(|\epsilon(j,T)| = o(1)\) such that
\begin{equation}
  \label{eq-c1c2computation-1}
  c_1^\pm(j,T)\zeta_1^\pm(a) 
  + c_2^\pm(j,T)\zeta_2^\pm(a)
  = \bigl(1 + \epsilon(j,T)\bigr) \cdot \zeta\bigl(\fq(j,T)\bigr).
\end{equation}

Recall that according to Lemma~\ref{lem-zetajT-at-exit-of-Qa}, for large \(T\) and small \(j\), we have
\begin{equation}
  \label{eq-c1c2computation-2}
  \zeta(\fq(j,T)) = \bigl(1+\epsilon_a(j,T)\bigr)\zeta\bigl(\bgk(t_a)\bigr)
  e^{-(A-i\Omega)T},
\end{equation}
with \(|\epsilon_a(j,T)|\leq Ca^2\).  According to~\eqref{eq-zeta-small-s}, we also have
\begin{align}
  \label{eq-c1c2computation-3}
  c_1^\pm\zeta_1^\pm(a) + c_2^\pm\zeta_2(a) &= 
                                              \Omega k^\pm \bigl(c_1^\pm + ic_2^\pm\bigr) a^{-A+i\Omega}
                                              +(c_1^\pm\delta_1^\pm(a)+c_2^\pm\delta_2(a)) a^{-A+i\Omega} \\
                                            &= \bigl(1+\delta_3^\pm(a)\bigr) 
                                              \Omega k^\pm \bigl(c_1^\pm + ic_2^\pm\bigr) a^{-A+i\Omega},
                                              \notag
\end{align}
in which \(\delta_k^\pm\) are complex functions with \(\delta_k^\pm(a) = \cO(a^2)\).  Combining \eqref{eq-zeta-is-linear-combination}, \eqref{eq-c1c2computation-1}, \eqref{eq-c1c2computation-2}, and \eqref{eq-c1c2computation-3}, we find that
\begin{align*}
  c^\pm_1 +i c^\pm_2 &=
                       \frac{(1+\epsilon(j,T))(1+\epsilon_a(j,T))}{(1+\delta_3^\pm(a))}
                       \frac{\zeta(\bgk(t_a))}{\Omega k^\pm}
                       a^{A-i\Omega}e^{(-A+i\Omega)T} \\
                     &=\bigl(1+\epsilon_2(a;j,T)\bigr)
                       \frac{\zeta(\bgk(t_a))}{\Omega k^\pm}
                       a^{A-i\Omega}e^{(-A+i\Omega)T},
\end{align*}
where \(|\epsilon_2(a;j,T)|\leq \delta_4(a)\) for sufficiently small \(j\) and large \(T\), and \(\delta_4(a)\to0\) as \(a\to0\).

We now define \(R(a,j,T)\) and \(\phi(a,j,T)\) by
\[
  R(a,j,T)e^{i\phi(a,j,T)} \isdef \bigl(1+\epsilon_2(a;j,T)\bigr) \frac{\zeta(\bgk(t_a))}{\Omega k^\pm} \, a^{A-i\Omega}.
\]
We then have
\begin{equation}
  \label{eq-R-defined}
  R(a,j,T)  = a^A \, \frac{\left|\zeta\bigl(\bgk(t_a)\bigr)\right|}{\Omega|k^\pm|}\,
  \big|1+\epsilon_2(a;j,T)\big| 
\end{equation}
and
\begin{equation}
  \label{eq-phi-defined}
  \phi(a,j,T) = \arg \frac{\zeta(\bgk(t_a))}{k^\pm} -\Omega \log a
  + \cO\bigl(|\epsilon_2(a;j,T)|\bigr),
\end{equation}
which imply that
\[
  c^\pm_1 +i c^\pm_2 = R(a,j,T) e^{(-A+i\Omega)T + i\phi(a,j,T)},
\]
and hence
\begin{equation}
  \left\{
    \begin{aligned}
      c_1^\pm(a, j, T)
      &= R(a,j,T)e^{-AT} \cos\bigl(\Omega T +\phi(a,j,T)\bigr),\\
      c_2^\pm(a,j,T) &= R(a,j,T)e^{-AT}\sin\bigl(\Omega T +\phi(a,j,T)\bigr).
    \end{aligned}
  \right.
\end{equation}
This leads to
\begin{equation}
  \label{eq-zeta-at-ell}
  \zeta(\ell;j,T)
  = Re^{-AT} \cos(\Omega T+\phi)\zeta_1^\pm(\ell) +
  Re^{-AT} \sin(\Omega T+\phi)\zeta_2^\pm(\ell),
\end{equation}
where \(\phi=\phi(a,j,T)\) and \(R=R(a,j,T)\).  We note that the \(\zeta_{1,2}^\pm(\ell)\) that appear in this expression are the exact solutions of the linearized equation, and that all error terms have been absorbed in the coefficients \(R(a,j,T)\) and \(\phi(a,j,T)\). Thus we have the following.

\begin{lemma}\label{lem-about-R-and-phi}
  For any \(\delta>0\), there exist \(\iota_\delta>0\), \(T_\delta<\infty\), and \(a_\delta>0\) such that for all \(a<a_\delta\), \(|j|\leq \iota_\delta\), and \(T\geq T_\delta\), one has
  \[
    R(a,j,T) = R_a (1+\epsilon_3(a,j,T)) \qquad\text{and}\qquad \phi(a,j,T) = \phi_a + \epsilon_4(a,j,T),
  \]
  where
  \[
    |\epsilon_3(a,j,T)| + |\epsilon_4(a,j,T)| \leq \delta,
  \]
  and where the constants \(R_a\) and \(\phi_a\) are given by
  \[
    R_a = a^A \frac{|\zeta\bigl(\bgk(t_a)\bigr)|}{\Omega |k^{\pm}|} \qquad\text{and}\qquad \phi_a = \arg \frac{\zeta\bigl(\bgk(t_a)\bigr)}{k^\pm} - \Omega \log a.
  \]
\end{lemma}

Our next step is to translate the complex variable \(\zeta\) at \(s=\ell\) back to the difference variables \(x_{12}\) and \(y_{12}\).  Applying \(\cC^{-1}\) to both sides of \eqref{eq-zeta-at-ell}, we get
\[
  \mat x_{12}(\ell;j,T) \\ y_{12}(\ell;j,T) \rix = Re^{-AT} \cos(\Omega T+\phi) \mat x_{12,1}^\pm(\ell) \\ y_{12,1}^\pm(\ell) \rix + Re^{-AT} \sin(\Omega T+\phi) \mat x_{12,2}^\pm(\ell) \\ y_{12,2}^\pm(\ell) \rix,
\]
where we have defined
\[
  x_{12}(s;j,T) = x_{12}\bigl(\fq(s;j,T)\bigr),\qquad y_{12}(s;j,T) = y_{12}\bigl(\fq(s;j,T)\bigr),
\]
and where \(x_{12,i}^\pm(s)\) and \(y_{12,i}^\pm(s)\) are the solutions of the linearized equations \eqref{eq-two-by-two-subsystem} corresponding to \(\chi_i^\pm(s)\) from Tables~\ref{tab-expander-asymptotics} and~\ref{tab-shrinker-asymptotics}.  Using \eqref{eq-zeta-from-chi}, we can write this as
\begin{multline*}
  \mat x_{12}(\ell;j,T) \\ y_{12}(\ell;j,T) \rix = Re^{-AT}
  \cos(\Omega T+\phi) \mat \chi_1^\pm(\ell) \\ -(\chi_1^\pm)'(\ell)/2(n-1) \rix \\
  + Re^{-AT} \sin(\Omega T+\phi) \mat \chi_2^\pm(\ell) \\ -(\chi_2^\pm)'(\ell)/2(n-1) \rix.
\end{multline*}
Our expansion of \(\chi_i^\pm(s)\) as \(s\to\infty\) implies that as \(\ell\to\infty\),
\[
  \chi_1^\pm(\ell) = 1+\cO\bigl(\ell^{-2}\bigr), \qquad \bigl(\chi_1^\pm\bigr)'(\ell) = \cO\bigl(\ell^{-2}\bigr),
\]
and
\[
  \chi_2^\pm(\ell) = \bigl(C_\chi^\pm +\cO\bigl(\ell^{-2}\bigr)\bigr)\ell^{-n-1}e^{\mp\ell^2/2}, \qquad \bigl(\chi_2^\pm\bigr)'(\ell) = \bigl(C_\chi^\pm +\cO\bigl(\ell^{-2}\bigr)\bigr)\ell^{-n+1}e^{\mp\ell^2/2}.
\]
We then get these expressions for \(\tmat x_{12}\\ y_{12} \trix\):
\begin{align}\label{eq-x12-y12-at-ell}
  \mat x_{12}(\ell;j,T) \\ y_{12}(\ell;j,T) \rix 
  &= Re^{-AT}
    \cos(\Omega T+\phi)
    \mat 1+\cO\bigl(\ell^{-2}\bigr) \\ \cO\bigl(\ell^{-2}\bigr) \rix
  \\
  &\qquad\qquad+ Re^{-AT}
    \sin(\Omega T+\phi)
    \ell^{-n+1}e^{\mp\ell^2/2}
    \mat \cO\bigl(\ell^{-2}\bigr) \\
  C_\chi^\pm+\cO\bigl(\ell^{-2}\bigr)\rix.
  \notag
\end{align}
Here \(R=R(a,j,T)\) and \(\phi=\phi(a,j,T)\) are as in Lemma~\ref{lem-about-R-and-phi}.

\section{Stable manifolds of the conical ends}
\label{sec-stable-mfds-of-cone-ends}

\subsection{Definition of $\mc W$, $\mc W_{\rm ex}$, and $\mc W_{\rm sh}$}
Metrics with conical ends correspond to solutions of the soliton flow for which $x_\alpha\to n-1+\xi_{\alpha,\infty}>0$ as $s\to\infty$.  Here we study such solutions for $s\gg 0$.

We define the \emph{stable set of conical ends} to be the set of all $\fp \in \R^6$ whose forward trajectory \(\{g^t\fp \mid t\geq0\}\) under the soliton flow satisfies
\begin{equation}
  \sup |\xi_\alpha| < \infty,\qquad
  \sup |y_\alpha| < \infty,\qquad
  \sup \left| \frac\gamma s\right| <\infty.
\end{equation}
We denote this set by $\mc W_{\rm sh}$ for \emph{shrinkers} $(\lambda=-1)$ and $\mc W_{\rm ex}$ for \emph{expanders} $(\lambda=+1)$, and we write $\mc W$ to refer to either $\mc W_{\rm sh}$ or $\mc W_{\rm ex}$.

The sets $\mc W$ are invariant under the soliton flow $g^t$, and they are nonempty because they always contain the $\rfes$ orbit.

For any $\epsilon>0$ and $\ell>0$, we define
\[
  \mc W^{\epsilon,\ell} \isdef \bigl\{(\xi_\alpha, y_\alpha, \gamma, s)\in \mc W ~:~ s > \ell, |\xi_\alpha| < \epsilon, |y_\alpha| < \epsilon, |\gamma|<\epsilon s\bigr\}.
\]
It seems very likely that $\mc W$ is a smooth submanifold of $\R^6$.  In this paper, we will only need to know that the part of $\mc W$ that is close to the $\rfes$ is smooth, which is what we show in this section:

\begin{theorem}\label{thm-conical-stable-set-is-smooth}
  For small enough $\epsilon>0$ and large enough $\ell > 0$, $\mc W^{\epsilon,\ell}$ is a smooth submanifold of $\R^6$.  More precisely:
  \begin{enumerate}
  \item $\mc W_{\rm ex}^{\epsilon,\ell}$ is an open subset of $\R^6$; and
  \item $\mc W_{\rm sh}^{\epsilon,\ell}$ is a 4-dimensional submanifold of $\R^6$.
  \end{enumerate}
  The tangent space to $\mc W$ at the $\rfes$ orbit is spanned by solutions of the linearized flow $\dd g^t$ along the $\rfes$ that are $\cO(s)$ as $s\to\infty$.

\end{theorem}

Our proof proceeds by rewriting the system~\eqref{eq-BlowUpSystem-Polynomial} of differential equations as a fixed-point problem involving integral equations, and then using the Implicit Function Theorem to conclude smoothness of $\mc W$ near the $\rfes$.  We will begin by showing that orbits in $\mc W$ actually have more precise asymptotics for $s\to\infty$; but before we do that, we record some elementary estimates involving Gaussian integrals that will show up a few times.

\begin{lemma} \label{lem-Gaussian-integral-estimate} If $\nu:[\ell, \infty)\to\R$ is bounded and measurable, then
  \[
    \left|\int_s^\infty e^{(s^2-\varsigma^2)/2} \nu(\varsigma)\,\dd\varsigma\right| \leq \frac{1}{s} \sup_{\varsigma\geq \ell} |\nu(\varsigma)|,
  \]
  and
  \[
    \left|\int_\ell^s e^{(\varsigma^2-s^2)/2}\nu(\varsigma)\,\dd\varsigma\right| \leq \frac{2}{s} \sup_{\varsigma\geq\ell}|\nu(\varsigma)|.
  \]
\end{lemma}
\begin{proof}
  To prove the first inequality, we note that for $\varsigma\geq s\geq \ell$, we have
  \[
    (\varsigma^2-s^2)/2 = (\varsigma-s)(\varsigma+s)/2\geq s(\varsigma-s),
  \]
  so that if $M=\sup_{s\geq \ell}|\nu(s)|$, then
  \[
    \left|\int_s^\infty e^{(s^2-\varsigma^2)/2} \nu(\varsigma)\,\dd\varsigma\right|
    \leq \int_s^\infty e^{-s(\varsigma-s)} M\,\dd\varsigma = \frac{M}{s}.
  \]
 
  For the second inequality, we argue that for $0 < \varsigma < s$, one has
  \[(s^2-\varsigma^2)/2 = (s-\varsigma)(s+\varsigma)/2 \geq (s-\varsigma)s/2,\] so that
  \[
    \left|\int_\ell^s e^{(\varsigma^2-s^2)/2}\nu(\varsigma)\,\dd\varsigma\right|
    \leq M \int_\ell^s e^{(\varsigma^2-s^2)/2}\,\dd\varsigma \leq M\int_\ell^s e^{(\varsigma-s)s/2}\,\dd\varsigma \leq \frac 2s M,
  \]
  as claimed.
\end{proof}

\begin{lemma} \label{lem-better-asymptotics} If $(\xi_\alpha, y_\alpha, \gamma)$ is a solution whose graph lies in $\mc W$, then as $s\to\infty$, we have
  \[
    \xi_\alpha = \xi_{\alpha,\infty} + \cO(1/s) ,\qquad y_\alpha = \mc O(1/s), \qquad \gamma = K_\infty s + \cO(1/s),
  \]
  for certain constants $\xi_{\alpha,\infty}$ and $K_\infty$.
\end{lemma}

\begin{proof}
  We rewrite the soliton system~ \eqref{eq-BlowUpSystem-Polynomial} as integral equations.  To integrate the equation for $y_\alpha$, we first rewrite it by using $\Gamma=\gamma-n$, obtaining
  \begin{equation}
    s\frac{\dd y_\alpha}{\dd s} + \lambda s^2 y_\alpha
    = \xi_\alpha + (\gamma-n+1)y_\alpha + \gamma.
  \end{equation}
  We integrate this equation using the integrating factor $s^{-1}e^{\lambda s^2/2}$.  We have to choose the domain of integration according to the sign of $\lambda$.  In the case of \emph{expanders}, where $\lambda=+1$, we assume initial values $y_\alpha(\ell) = y_{\alpha,\ell}$ are given at $s=\ell$, and find for $s\geq \ell$ that
  \begin{subequations}
    \begin{equation}
      \label{eq-y-alpha-integral-expanders}
      y_\alpha(s) = e^{(\ell^2-s^2)/2}y_{\alpha,\ell}
      +
      \int_\ell^s e^{(\varsigma^2-s^2)/2} \Bigl\{
      \xi_\alpha(\varsigma)
      + \big(\gamma(\varsigma)-n+1\big)y_\alpha(\varsigma)
      +\gamma(\varsigma)\Bigr\}\,
      \frac{\dd\varsigma}{\varsigma}.
    \end{equation}
    In the case of \emph{shrinkers}, where $\lambda=-1$, we cannot specify the values of $y_\alpha(\ell)$.  Instead we get
    \begin{equation}
      \label{eq-y-alpha-integral-shrinkers}
      y_\alpha(s) = - \int_s^\infty e^{-(\varsigma^2-s^2)/2} \Bigl\{
      \xi_\alpha(\varsigma)
      + \big(\gamma(\varsigma)-n+1\big)y_\alpha(\varsigma)
      +\gamma(\varsigma)
      \Bigr\}\,
      \frac{\dd\varsigma}{\varsigma}.
    \end{equation}
  \end{subequations}
  If $\xi_\alpha=\cO(s)$, $\gamma = \cO(s)$, and $y_\alpha = \cO(1)$ as $s\to\infty$, then we have
  \[
    \frac{\xi_\alpha + (\gamma-n+1)y_\alpha+\gamma}{s} = \cO(1) \qquad (s\to\infty).
  \]
  In either case,~\eqref{eq-y-alpha-integral-expanders} or \eqref{eq-y-alpha-integral-shrinkers} combined with Lemma~\ref{lem-Gaussian-integral-estimate} implies that $y_\alpha = \cO(1/s)$ as $s\to\infty$.

  Next we consider $\xi_\alpha$.  The equation
  \[
    s\frac{\dd x_\alpha}{\dd s} = -2x_\alpha y_\alpha
  \]
  for $x_\alpha(s)$ can be integrated as follows:
  \[
    x_\alpha(s) = x_{\alpha\infty} \exp \Bigl( -2\int_s^\infty y_\alpha(\varsigma)\, \frac{\dd\varsigma}{\varsigma} \Bigr).
  \]
  For \(\xi_\alpha = x_\alpha -(n-1)\), this implies that
  \begin{equation}
    \label{eq-xialpha-integral}
    \xi_\alpha(s) =
    \xi_{\alpha,\infty}
    e^{ -2\int_s^\infty y_\alpha(\varsigma)\,\frac{\dd\varsigma}{\varsigma} }
    +(n-1)
    \Bigl\{e^{ -2\int_s^\infty y_\alpha(\varsigma)\,\frac{\dd\varsigma}{\varsigma} } -1 \Bigr\}.
  \end{equation}
  Since we have just shown that $y_\alpha = \cO(1/s)$, we can estimate the integral and conclude that
  \[
    \xi_\alpha = \xi_{\alpha,\infty} e^{\cO(1/s)}+(n-1)\{e^{\cO(1/s)} - 1\} = \xi_{\alpha\infty} + \cO(1/s),
  \]
  as claimed.

  Finally, we consider $\gamma$.  Again using $\gamma=\Gamma+n$, we rewrite equation \eqref{eq-PG} as
  \[
    s\frac{\dd\gamma}{\ds}- \gamma = -n + \tsum_\alpha p_\alpha (1+2y_\alpha + y_\alpha^2) = \tsum_\alpha p_\alpha\bigl\{2y_\alpha+y_\alpha^2 \bigr\}.
  \]
  Then we multiply with the integrating factor $s^{-2}$ and integrate, obtaining
  \begin{equation}
    \label{eq-gamma-integral}
    \gamma(s)
    = K_\infty s - s\int_s^\infty {\textstyle\sum_\alpha p_\alpha}
    \bigl\{ 2 y_\alpha(\varsigma) + y_\alpha(\varsigma)^2 \bigr\} \,
    \frac{\dd\varsigma}{\varsigma^2}.
  \end{equation}
  Here $K_\infty$ is a constant of integration, and we again use $y_\alpha = \cO(1/s)$ to justify convergence of the integral and to conclude that $\gamma = K_\infty s + \cO(1/s)$.
\end{proof}

\begin{lemma}
  For each \(N\geq 0\) there exist constants \(\xi_{\alpha j}\), \(y_{\alpha j}\), \(K_j\) such that 
  \begin{subequations}\label{eq-xi-y-g-expansion}
  \begin{align}
    \xi_\alpha(s)&= \xi_{\alpha,\infty}
                   +\frac{\xi_{\alpha1}}{s}+\frac{\xi_{\alpha2}}{s^2}+\cdots 
                   +\frac{\xi_{\alpha N}}{s^N} + \mc O(s^{-N-1})\\
    y_\alpha(s)&= \frac{y_{\alpha1}}{s} + \frac{y_{\alpha2}}{s^2} + \cdots + \frac{y_{\alpha N}}{s^N}+\mc O(s^{-N-1})\\
    \gamma(s)&=K_\infty s + \frac{K_1}{s} + \cdots + \frac{K_N}{s^N} + \mc O(s^{-N-1})
  \end{align}
  \end{subequations}
  for \(s\to\infty\).  These expansions may be differentiated indefinitely.

  The coefficients \(\xi_{\alpha j}\), \(y_{\alpha j}\), and \(K_j\) can be computed inductively.  In particular, one finds \(y_{\alpha1}=\lambda K_\infty\)
\end{lemma}
\begin{proof}
  Lemma~\ref{lem-better-asymptotics} implies the case \(N=0\).  To obtain the expansions for larger \(N\) apply induction to the integral equations \eqref{eq-y-alpha-integral-expanders}, \eqref{eq-y-alpha-integral-shrinkers}, \eqref{eq-xialpha-integral}, \eqref{eq-gamma-integral}.

  The differential equations~\eqref{eq-xi-y-gamma-system} imply similar expansions for the derivatives.  Once one has obtained expansions for the functions \(\xi_\alpha, y_\alpha, \gamma\) and their derivatives it follows that the expansions for the derivatives are those obtained by formally differentiating the expansions~\eqref{eq-xi-y-g-expansion}.
\end{proof}

\subsection{Proof of Theorem~\ref{thm-conical-stable-set-is-smooth}}
For expanders (\(\lambda=+1\)), we have the following integral equations for \((\xi_\alpha, y_\alpha, \gamma)\):
\begin{subequations}\label{eq-ell-integral-expanders}
  \begin{align}
    \xi_\alpha(s) &= \xi_{\alpha,\ell}
                    e^{ -2\int_\ell^s y_\alpha(\varsigma)\frac{\dd\varsigma}{\varsigma} }
                    + (n-1)\left\{
                    e^{ -2\int_\ell^s y_\alpha(\varsigma)\frac{\dd\varsigma}{\varsigma} }
                    -1\right\},
    \\
    y_\alpha(s)&=  e^{(\ell^2-s^2)/2}y_{\alpha,\ell}
                 + \int_\ell^s e^{(\varsigma^2-s^2)/2} \Bigl\{
                 \xi_\alpha(\varsigma) + \big(\gamma(\varsigma)-n+1\big)y_\alpha(\varsigma)
                 + \gamma(\varsigma)\Bigr\}\,
                 \frac{\dd\varsigma}{\varsigma}, \\
    \gamma(s)
                  &= \gamma_\ell \frac s\ell + s\int_\ell^s {\textstyle\sum_\alpha p_\alpha}
                    \bigl\{ 2 y_\alpha(\varsigma) + y_\alpha(\varsigma)^2 \bigr\} \,
                    \frac{\dd\varsigma}{\varsigma^2},
  \end{align}
\end{subequations}
where \(\xi_{\alpha,\ell} = \xi_\alpha(\ell)\), \(y_{\alpha,\ell}=y_\alpha(\ell)\), and \(\gamma_\ell = \gamma(\ell)\) are the values of the solution at \(s=\ell\).

For shrinkers (\(\lambda=-1\)), we have
\begin{subequations}\label{eq-ell-integral-shrinkers}
  \begin{align}
    \xi_\alpha(s) &= \xi_{\alpha,\ell}
                    e^{ -2\int_\ell^s y_\alpha(\varsigma)\frac{\dd\varsigma}{\varsigma} }
                    + (n-1)\left\{
                    e^{ -2\int_\ell^s y_\alpha(\varsigma)\frac{\dd\varsigma}{\varsigma} }
                    -1\right\},\\
    y_\alpha(s)&=
                 \int_s^\infty e^{(s^2 - \varsigma^2)/2} \Bigl\{
                 \xi_\alpha(\varsigma)
                 + \big(\gamma(\varsigma)-n+1\big)y_\alpha(\varsigma)
                 +\gamma(\varsigma)
                 \Bigr\}\,
                 \frac{\dd\varsigma}{\varsigma}, \\
    \gamma(s)
                  &= \gamma_\ell \frac s\ell + s\int_\ell^s {\textstyle\sum_\alpha p_\alpha}
                    \bigl\{ 2 y_\alpha(\varsigma) + y_\alpha(\varsigma)^2 \bigr\} \,
                    \frac{\dd\varsigma}{\varsigma^2}.
  \end{align}
\end{subequations}
The orbits of the soliton flow that model expander solitons with conical ends are solutions of~\eqref{eq-ell-integral-expanders} in which $\xi_{\alpha,\ell}$, $y_{\alpha,\ell}$, and $\gamma_\ell$ are five free parameters.  Shrinkers with conical ends are modeled by solutions of \eqref{eq-ell-integral-shrinkers}, and for these we only have three free parameters, namely $x_{\alpha,\ell}$ and $\gamma_\ell$.

To use the Implicit Function Theorem, we introduce a family of norms defined for functions \(f:[\ell,\infty) \to \R\).  Let \(r\in\R\) be any constant.  Then we define
\begin{equation}
  \label{eq-solspace-norms-defined}
  \|f\|_r = \sup_{s\geq \ell} s^{r}|f(s)|.
\end{equation}

The $\rfes$ soliton given by $\xi_\alpha(s) = 0$, $y_\alpha(s)=0$, and $\gamma(s)=0$ is a solution of these integral equations.  In order to apply the Implicit Function Theorem, we introduce the function space $\solspace$ of all continuous $ (\xi_\alpha, y_\alpha, \gamma) : [\ell,\infty) \to\R^5$ for which the norm
\[
  \|( \xi_\alpha, y_\alpha, \gamma)\|_\solspace \isdef \max \bigl\{\|\xi_\alpha\|_0, \|y_\alpha\|_{1/2}, \|\gamma\|_{-1}\colon \alpha = 1,2\bigr\}
\]
is finite, and for which the limits
\[
  \xi_{\alpha\infty} = \lim_{s\to\infty} \xi_\alpha(s) \quad \text{ and }\quad K_\infty = \lim_{s\to\infty} \frac{\gamma(s)}{s}
\]
exist.  With this norm, $\solspace$ is a Banach space.

The integral equations~(\ref{eq-ell-integral-expanders}abc) for expanders are the fixed-point equations for the map
\begin{multline}
  \mf F_{\rm ex}(\xi_{\alpha,\ell}, y_{\alpha,\ell}, \gamma_\ell \mid \xi_\alpha, y_\alpha, \gamma) = \Biggl( \xi_{\alpha,\ell} e^{ -2\int_\ell^s y_\alpha\,\frac{\dd\varsigma}{\varsigma}}
  + (n-1) \bigl\{e^{ -2\int_\ell^s y_\alpha\,\frac{\dd\varsigma}{\varsigma}} -1\bigr\},\\
  e^{(\ell^2-s^2)/2}y_{\alpha,\ell} + \int_\ell^s e^{(\varsigma^2-s^2)/2} \bigl\{ \xi_\alpha + (\gamma-n+1)y_\alpha + \gamma \bigr\}\,
  \frac{\dd\varsigma}{\varsigma}, \\
  \gamma_\ell \frac s\ell + s \int_\ell^s { \textstyle\sum_\alpha p_\alpha} \bigl\{ 2 y_\alpha + y_\alpha^2 \bigr\} \,\frac{\dd\varsigma}{\varsigma^2} \Biggr),
\end{multline}
while the equations~(\ref{eq-ell-integral-shrinkers}abc) that describe shrinkers are the fixed point equations for the map
\begin{multline}
  \mf F_{\rm sh}(\xi_{\alpha,\ell}, \gamma_\ell \mid x_\alpha, y_\alpha, \gamma) = \Biggl( \xi_{\alpha,\ell} e^{ -2\int_\ell^s y_\alpha\,\frac{\dd\varsigma}{\varsigma}}
  + (n-1) \bigl\{e^{ -2\int_\ell^s y_\alpha\,\frac{\dd\varsigma}{\varsigma}} -1\bigr\},\\
  - \int_s^\infty e^{(s^2 - \varsigma^2)/2} \bigl\{ \xi_\alpha + (\gamma-n+1)y_\alpha + \gamma \bigr\}\,
  \frac{\dd\varsigma}{\varsigma}, \\
  \gamma_\ell \frac s\ell + s\int_\ell^s {\textstyle\sum_\alpha p_\alpha} \bigl\{ 2 y_\alpha + y_\alpha^2 \bigr\} \,\frac{\dd\varsigma}{\varsigma^2} \Biggr).
\end{multline}

Using Lemma~\ref{lem-Gaussian-integral-estimate}, it is straightforward to verify that $\mf F_{\rm ex}:\R^5\times \solspace \to \solspace$ and $\mf F_{\rm sh}:\R^3\times \solspace\to \solspace$ are well defined.  $\mf F_{\rm ex}$ is linear in the parameters $(\xi_{\alpha,\ell}, y_{\alpha,\ell}, \gamma_\ell)$; $\mf F_{\rm sh}$ is linear in $(\xi_{\alpha,\ell}, \gamma_\ell)$; and both are real analytic in $(\xi_\alpha, y_\alpha, \gamma) \in \solspace$.  Hence the Implicit Function Theorem applies, provided we can show that the appropriate derivatives are invertible.  We define $\mf L_{\rm ex},\,\mf L_{\rm sh} : \solspace \to \solspace$ to be the partial Fr\'echet derivatives of $\mf F_{\rm ex}$ and $\mf F_{\rm sh}$, respectively, with respect to the $\solspace$ variables, \emph{i.e.,}
\[
  \mf L_{\rm ex}\cdot (\delta \xi_\alpha, \delta y_\alpha, \delta\gamma) = \left.  \frac {\dd}{\dd \varepsilon} \mf F_{\rm ex}\bigl( 0, 0, 0 \mid \varepsilon \delta \xi_\alpha, \varepsilon \delta y_\alpha, \varepsilon\delta\gamma\bigr)\right|_{\varepsilon=0}
\]
and
\[
  \mf L_{\rm sh}\cdot (\delta \xi_\alpha, \delta y_\alpha, \delta\gamma) = \left.  \frac {\dd}{\dd \varepsilon} \mf F_{\rm sh}\bigl( 0, 0,0 \mid \varepsilon \delta \xi_\alpha, \varepsilon \delta y_\alpha, \varepsilon\delta\gamma\bigr)\right|_{\varepsilon=0}.
\]
Concretely, these linearizations are given by
\begin{align*}
  \mf L_{\rm ex}\cdot (\delta \xi_\alpha, \delta y_\alpha, \delta\gamma)
  & =
    \Bigl(
    \mf I_{1\alpha}[\delta \xi_\alpha, \delta y_\alpha, \delta\gamma],
    \mf I_{2\alpha}^+[\delta \xi_\alpha, \delta y_\alpha, \delta\gamma],
    \mf I_3[\delta \xi_\alpha, \delta y_\alpha, \delta\gamma]
    \Bigr), \\
  \mf L_{\rm sh}\cdot (\delta \xi_\alpha, \delta y_\alpha, \delta\gamma)
  & =
    \Bigl(
    \mf I_{1\alpha}[\delta \xi_\alpha, \delta y_\alpha, \delta\gamma],
    \mf I_{2\alpha}^-[\delta \xi_\alpha, \delta y_\alpha, \delta\gamma],
    \mf I_3[\delta \xi_\alpha, \delta y_\alpha, \delta\gamma]
    \Bigr),
\end{align*}
where the integral operators are
\begin{align*}
  \mf I_{1\alpha}[\delta \xi_\alpha, \delta y_\alpha, \delta\gamma]
  & =  2(n-1)\int_\ell^s \delta y_\alpha\,\frac{\dd\varsigma}{\varsigma},  \\
  \mf I_{2\alpha}^+[\delta \xi_\alpha, \delta y_\alpha, \delta\gamma]
  & =\int_\ell^s e^{(\varsigma^2-s^2)/2}
    \bigl\{\delta \xi_\alpha - (n-1)\delta y_\alpha + \delta\gamma\bigr\}\, \frac{\dd\varsigma}{\varsigma}, \\
  \mf I_{2\alpha}^-[\delta \xi_\alpha, \delta y_\alpha, \delta\gamma]
  & =-\int_s^\infty e^{(s^2 - \varsigma^2)/2}
    \bigl\{\delta \xi_\alpha - (n-1)\delta y_\alpha + \delta\gamma\bigr\}\,
    \frac{\dd\varsigma}{\varsigma},  \\
  \mf I_3[\delta \xi_\alpha, \delta y_\alpha, \delta\gamma]
  & = 2s\int_\ell^s {\textstyle\sum_\alpha p_\alpha}\delta y_\alpha\,\frac{\dd\varsigma}{\varsigma^2}.
\end{align*}

\begin{lemma}\label{lem-L-contraction} The operator norms of the Fr\'echet
  derivatives are bounded by
  \[
    \|\mc L_{\rm ex}\|_{\solspace\to\solspace} \leq \frac{4n}{\sqrt\ell} \qquad \text{and}\qquad \|\mc L_{\rm sh}\|_{\solspace\to\solspace} \leq \frac{4n}{\sqrt\ell}.
  \]
  Hence if $\ell > 16n^2$, then the operators $\mf L_{\rm ex}:\solspace\to\solspace$ and $\mf L_{\rm sh}:\solspace \to \solspace$ are contractions.
\end{lemma}

\begin{proof} We estimate $\|\mf L_{\rm ex}\|_{\solspace\to\solspace}$.  Let $(\delta \xi_\alpha, \delta y_\alpha, \delta\gamma) \in \solspace$ be given, and set $ M = \|(\delta \xi_\alpha, \delta y_\alpha, \delta\gamma)\|_{\mb X}$, so that for all $s\geq \ell$, one has
  \[
    |\delta \xi_\alpha(s)|\leq M, \qquad |\delta y_\alpha(s)| \leq \frac {M}{\sqrt s}, \qquad |\delta\gamma(s)|\leq Ms.
  \]
  Then $w_{1\alpha} = \mf I_{1\alpha}[\delta\xi_\alpha, \delta y_\alpha, \delta\gamma]$ satisfies
  \[
    |w_{1\alpha}(s)| \leq 2(n-1)M\int_\ell^\infty \frac{\dd\varsigma}{\varsigma^{3/2}} = \frac{4(n-1)}{\sqrt{\ell}} M.
  \]
  For $w_{2\alpha}^+(s) = \mf I_{2\alpha}^+[\delta\xi_\alpha, \delta y_\alpha, \delta\gamma]$, we have
  \[
    |w_{2\alpha}^+(s)|\leq M \int_\ell^s e^{(\varsigma^2-s^2)/2} \bigl\{1 + (n-1) \varsigma^{-1/2} + \varsigma \bigr\}\, \frac{\dd\varsigma}{\varsigma}.
  \]
  Since $\varsigma\geq s\geq \ell\geq 1$, we have $1+(n-1)\varsigma^{-1/2}\leq n\leq n\varsigma$.  So by Lemma~\ref{lem-Gaussian-integral-estimate}, we find that
  \[
    |w_{2\alpha}^+(s)|\leq (n+1)M \int_\ell^s e^{(\varsigma^2-s^2)/2}\, \dd\varsigma \leq \frac{2(n+1)}{s} M.
  \]
  Hence for all $s\geq \ell$, we have
  \[
    \sqrt{s}|w_{2\alpha}^+(s)|\leq \frac{2(n+1)}{\sqrt\ell}M.
  \]
  For $w_{2\alpha}^-(s) = \mf I_{2\alpha}^-[\delta\xi_\alpha, \delta y_\alpha, \delta\gamma]$, we have a similar estimate,
  \begin{align*}
    |w_{2\alpha}^-(s)|
    &\leq M \int_s^\infty e^{(s^2 - \varsigma^2)/2}
      \bigl\{1 + (n-1) \varsigma^{-1/2} + \varsigma \bigr\}\,
      \frac{\dd\varsigma}{\varsigma}\\
    &\leq (n+1)M \int_s^\infty e^{(s^2-\varsigma^2)/2}\, \dd\varsigma \\
    &\leq \frac{n+1}{s} M.
  \end{align*}
  Thus for all $s\geq \ell$, we have
  \[
    \sqrt s \, |w_{2\alpha}^-(s)| \leq \frac{n+1}{\sqrt\ell} M.
  \]
  Finally, we consider $w_3(s) = \mf I_3[\delta\xi_\alpha, \delta y_\alpha, \delta\gamma]$, which satisfies
  \[
    \frac1s|w_3(s)| \leq 2 \int_\ell^\infty {\textstyle \sum_\alpha p_\alpha} \frac{M\dd\varsigma}{\varsigma^{5/2}} = \frac{4n}{3\ell^{3/2}}M
  \]
  for all $s\geq \ell$.

  Together these estimates imply that
  \[
    \|\mf L_{\rm ex}\|_{\solspace\to\solspace} \leq \max \Bigl\{\frac{4(n-1)}{\sqrt\ell}, \frac{2(n+1)}{\sqrt\ell}, \frac{4n}{\ell^{3/2}} \Bigr\} \leq \frac{4n}{\sqrt{\ell}}.
  \]
  The same estimate holds for $\mf L_{\rm sh}$.  It follows that both $\mf L_{\rm ex}$ and $\mf L_{\rm sh}$ are contractions if $\ell > 16n^2$.
\end{proof}

\begin{lemma} \label{lem-stable-mfd-cones} Assume $\ell > 16n^2$.

  For all $(\xi_{\alpha,\ell}, y_{\alpha,\ell}, \gamma_\ell) \in\R^5$ in a small neighborhood of \((0,0,0,0,0)\), there is a unique orbit \((\xi_\alpha, y_\alpha,\gamma)\in\solspace\) of the soliton flow with $\lambda=+1$ for which
  \[
    \xi_\alpha(\ell) = \xi_{\alpha,\ell}, \quad y_\alpha(\ell) = y_{\alpha,\ell}, \quad \gamma(\ell) = \gamma_\ell.
  \]
  The solution $(\xi_\alpha, y_\alpha, \gamma) \in \solspace$ is a real analytic function of the parameters $(\xi_{\alpha,\ell}, y_{\alpha,\ell}, \gamma_\ell)$.

  For all $(\xi_{\alpha,\ell}, \gamma_\ell) \in\R^3$ in a small neighborhood of \((0,0,0)\), there is a unique orbit \((\xi_\alpha, y_\alpha,\gamma)\in\solspace\) of the soliton flow with $\lambda=-1$ for which
  \[
    \xi_\alpha(\ell) = \xi_{\alpha,\ell}, \qquad \gamma(\ell) = \gamma_\ell .
  \]
  The solutions $(\xi_\alpha, y_\alpha, \gamma) \in \solspace$ again depend analytically on the parameters $(\xi_{\alpha,\ell}, \gamma_\ell)$.
\end{lemma}

\begin{proof}
  In the first case, the solutions in question are solutions of
  \[
    (\xi_\alpha, y_\alpha, \gamma) = \mf F_{\rm ex}(\xi_{\alpha,\ell}, y_{\alpha,\ell}, \gamma_\ell \mid \xi_\alpha, y_\alpha, \gamma).
  \]
  We have one solution, the \(\rfes\), given by \((\xi_\alpha, y_\alpha, \gamma) = (0, 0, 0, 0, 0)\).  Then the Implicit Function Theorem on Banach Spaces, together with Lemma~\ref{lem-L-contraction}, immediately implies the existence, uniqueness, and smooth dependence of solutions for small nonzero values of the parameters.

  The other case, where \(\lambda=-1\), deals with solutions of
  \[
    (\xi_\alpha, y_\alpha, \gamma) = \mf F_{\rm sh}(\xi_{\alpha,\ell}, \gamma_\ell \mid \xi_\alpha, y_\alpha, \gamma).
  \]
  We again have one solution, the \(\rfes\), given by \((\xi_\alpha, y_\alpha, \gamma) = (0, 0, 0, 0, 0)\).  As in the case of expanders, the Implicit Function Theorem again applies.
\end{proof}

\section{Expanding solitons}

\subsection{Notation}
In this section, we consider expanders and thus assume that \(\lambda=+1\).

It will be more natural to regard \((\xi_\alpha(s), y_\alpha(s), \gamma(s))\) as an \(\R^5\)-valued function of \(s\).  We will use lower-case boldface letters for vectors in \(\R^5\) and also for \(\R^5\)-valued functions.  The solutions provided by Lemma~\ref{lem-stable-mfd-cones} are functions \(\bs x:[\ell,\infty) \to\R^5\) that belong to the Banach space \(\solspace\), with
\[
  \bs x(s) = (\xi_\alpha(s), y_\alpha(s), \gamma(s)).
\]
For any \(\bs x\in \solspace\), the limits
\[
  \xi_\alpha^\infty(\bs x) \isdef \lim_{s\to\infty} \xi_\alpha(s) \quad\text{ and }\quad K^\infty(\bs x) \isdef \lim_{s\to\infty} \frac{\gamma(s)}{s}
\]
are by definition of \(\solspace\) well-defined.  Both \(\xi_\alpha^\infty:\solspace\to\R\) and \(K^\infty : \solspace\to\R\) are bounded linear functionals.

Lemma~\ref{lem-stable-mfd-cones} provides a map \(\bs x:\cU\to\solspace\), where \(\cU\subset\R^5\) is a small neighborhood of the origin, and where \(\bs x(\bs q)\in\solspace\) is the solution of the expander soliton equations that passes through \((\ell, \bs q)\in\R^6\).  We denote the value of \(\bs x(\bs q)\in\solspace\) at \(s\in[\ell,\infty]\) by
\[
  \bigl(\bs x(\bs q)\bigr)(s) = \bs x(s;\bs q).
\]
With this notation, we then have
\[
  \bs x(\ell;\bs q) = \bs q.
\]
When we need the components of \(\bs x\), we will also write
\[
  \bs x(s;\bs q) = \bigl( \xi_\alpha(s;\bs q), y_\alpha(s;\bs q), \gamma(s;\bs q) \bigr).
\]

\subsection{Asymptotic slopes of expanding solitons}
The quantities we are interested in are
\[
  \xi_\alpha^\infty\bigl(\bs x(\bs q)\bigr) = \lim_{s\to\infty}\xi_\alpha(s;\bs q),
\]
in which \(\bs q\) comes from the unstable manifold of the \(\gf\).  More precisely, if
\[
  g^{\log\ell/a}\bigl(\fq(j,T)\bigr) = \bigl(\ell, \bs q(j,T)\bigr),
\]
then we will show the following:

\begin{lemma}
  One has
  \begin{equation}
    \label{eq-xi-alpha-infty-expanders}
    \lim_{s\to\infty} \xi_\alpha(s;\bs q(j,T))
    = \xi_\alpha^\infty\bigl(\bs x(\bs q(j,T))\bigr)
    =c_2^+j + \cO\bigl(j^2+\epsilon(j,T)\bigr),
  \end{equation}
  with \(c_2^+\) as in~\eqref{eq-wu-expansion-at-a}.
\end{lemma}
\begin{proof}
  We found in Lemma~\ref{lem-Phi-at-ell} that
  \[
    \bs q(j,T) = c_2^+ j \bs\Phi_2^+(\ell) + \cO(j^2) + \bs\epsilon(j,T).
  \]
  We substitute this in \(\bs x\) and expand using the Fr\'echet derivative of \(\bs x:\cU\to\solspace\) to obtain
  \begin{align*}
    \bs x(\bs q(j,T))
    &= \dd\bs x(0)\cdot\bs q(j,T) + \cO\bigl(\|\bs q(j,T)\|^2\bigr)\\
    &= c_2^+j\,\dd\bs x(0)\cdot\bs\Phi_2^+ + \cO\bigl(j^2+\epsilon(j,T)\bigr).
  \end{align*}

  For any vector \(\bs v\in\R^5\), the function \(\bs w = \dd \bs x(0)\cdot\bs v\in\solspace\) is given by
  \[
    \bs w = \frac{\dd}{\dd \theta} \bs x(\theta\bs v)\Big|_{\theta=0}.
  \]
  Because \(\bs x(\theta\bs v)\) is the solution of the soliton equations with \(\bs x(\ell; \theta\bs v) = \theta\bs v\), it follows that \(\bs w(s)\) is a solution of the linearized equation around the \(\rfes\) with \(\bs w(\ell) = \bs v\).  Since \(\bs \Phi_2^+(s)\) is such a solution, we have
  \[
    \left(\dd \bs x(0)\cdot\bs\Phi_2^+\right)(s) = \bs \Phi_2^+(s).
  \]
  Therefore,
  \[
    \xi_\alpha^\infty\left(\dd \bs x(0)\cdot\bs\Phi_2^+\right) = \lim_{s\to\infty}\xi_\alpha\bigl(\bs \Phi_2^+(s)\bigr) = 1,
  \]
  as we see from the asymptotics of \(\Phi_2^+(s)\) at \(s=\infty\) derived in \S~\ref{Phi-big-s}.  The conclusion is then that~\eqref{eq-xi-alpha-infty-expanders} does indeed hold.
\end{proof}

\subsection{The difference of the asymptotic slopes}
Expression \eqref{eq-xi-alpha-infty-expanders} provides the dependence of the asymptotic slopes \(\xi_\alpha(\infty; j,T)\) on \(j\).  Unfortunately, the \(T\) dependence in the expressions we have thus far is contained in the error terms.  To exhibit the \(T\) dependence of \(\xi_\alpha(\infty; j,T)\), we must consider the difference \(\xi_1-\xi_2 = x_{12}(\ell;j,T)\).

\begin{lemma}
  Assume \(\|\xi_\alpha\|_0, \|y_\alpha\|_{1/2}, \|\gamma\|_{-1} \leq \epsilon\).  Then
  \begin{equation}
    \label{eq-x12-approximation-expanders}
    x_{12}(\infty;j,T)
    = Re^{-AT}\cos(\Omega T+\phi) + Re^{-AT}\epsilon(\ell,j,T),  \\
  \end{equation}
  where \(R\) and \(\phi\) are nearly constant, in the precise sense given by Lemma~\ref{lem-about-R-and-phi}.
\end{lemma}
\begin{proof}
  The following arguments allow us to estimate the difference \(x_{12}(\infty)-x_{12}(\ell)\).  We begin by subtracting equations~(\ref{eq-ell-integral-expanders}a) for \(\xi_1\) and \(\xi_2\), obtaining
  \begin{align}\label{eq-x12-integral}
    x_{12}(s)
    & = x_{12\ell}\,
      e^{ -2\int_\ell^s y_1(\varsigma)\,\frac{\dd\varsigma}{\varsigma}}
      + (n-1+\xi_{2\ell})\left\{
      e^{-2\int_\ell^s y_1(\varsigma)\,\frac{\dd\varsigma}{\varsigma}}
      - e^{-2\int_\ell^s y_2(\varsigma)\,\frac{\dd\varsigma}{\varsigma}}
      \right\}\\
    & = e^{ -2\int_\ell^s y_1(\varsigma)\,\frac{\dd\varsigma}{\varsigma}}
      \Bigl\{
      x_{12\ell}
      +(n-1+\xi_{2\ell}) \left( 1-
      e^{2\int_\ell^s y_{12}(\varsigma)\,\frac{\dd\varsigma}{\varsigma}}
      \right)
      \Bigr\}.
      \notag
  \end{align}
  Then we let \(s\to\infty\) to see that
  \[
    x_{12}(\infty) = e^{ -2\int_\ell^\infty y_1( s)\,\frac{\dd s}{ s}} \Bigl\{ x_{12\ell} + (n-1+\xi_{2\ell}) \left( 1- e^{2\int_\ell^\infty y_{12}( s)\,\frac{\dd s}{ s}} \right) \Bigr\}.
  \]
  In the following estimates, we use the facts that \(\|y_\alpha\|_{1/2}\leq \epsilon\) and \(\|\xi_\alpha\|_0\leq \epsilon<1\), which imply in particular that
  \[
    \int_\ell^\infty|y_\alpha(\varsigma)|\frac{\dd\varsigma}{\varsigma} \leq \|y_\alpha\|_{1/2} \int_\ell^\infty \frac{\dd\varsigma}{\varsigma^{3/2}} \leq \frac{2}{\sqrt\ell} \|y_\alpha\|_{1/2} \leq \frac{2\epsilon}{\sqrt\ell}.
  \]
  We will also frequently use the calculus inequality \(|e^x-1|\leq (e-1)|x|\leq 2|x|\) for all \(|x|\leq 1\).  Finally, we will regularly use the assumption that \(\ell\) is sufficiently large, for example, to ensure that \(8\epsilon/\sqrt\ell < 1\) (which will hold if \(\ell>64\)).

  We obtain
  \begin{align*}
    |x_{12}(\infty) - x_{12}(\ell)|\leq&
                                         \left|
                                         e^{2\int_\ell^\infty|y_1(\varsigma)|\frac{\dd\varsigma}{\varsigma}}-1
                                         \right|
                                         \,|x_{12}(\ell)| \\
                                       &\quad +  (n-1+\xi_{2\ell})
                                         e^{2\int_\ell^\infty|y_1(\varsigma)|\frac{\dd\varsigma}{\varsigma}}
                                         \left| 1-
                                         e^{2\int_\ell^\infty y_{12}( s)\,\frac{\dd s}{ s}}
                                         \right|
    \\
    \leq&  \frac{8\epsilon}{\sqrt\ell} |x_{12}(\ell)|
          + n \Bigl(1+\frac{8\epsilon}{\sqrt\ell}\Bigr)
          \cdot 4\int_\ell^\infty |y_{12}(\varsigma)|\frac{\dd\varsigma}{\varsigma},
  \end{align*}
  and hence
  \begin{equation}
    \label{eq-x12-ell-to-infty-estimate}
    |x_{12}(\infty) - x_{12}(\ell)|
    \leq  \frac{8\epsilon}{\sqrt\ell} |x_{12}(\ell)|
    + \frac{16n}{\sqrt\ell}\|y_{12}\|_{1/2}.
  \end{equation}
  It also follows from \eqref{eq-x12-integral} by a very similar computation that
  \begin{equation}
    \label{eq-x12-bound}
    \|x_{12}\|_0 \leq
    \Bigl(1+\frac{8\epsilon}{\sqrt\ell}\Bigr) |x_{12}(\ell)| + \frac{16n}{\sqrt\ell} \|y_{12}\|_{1/2}
  \end{equation}
  holds.  We note that the constant \(\lambda\) never entered this part of our derivation, so that estimate~\eqref{eq-x12-bound} holds not only here but also for solutions of the shrinker equations, in which \(\lambda=-1\) rather than \(\lambda=+1\).

  To get an analogous estimate for \(y_{12}\), we subtract equations~(\ref{eq-ell-integral-expanders}b) with $\alpha=1,2$, which leads to
  \[
    y_{12}(s) = e^{(\ell^2-s^2)/2} y_{12}(\ell) + \int_\ell^s e^{(\varsigma^2-s^2)/2} \bigl\{ x_{12}(\varsigma) + (\gamma(\varsigma)-n+1)y_{12}(\varsigma) \bigr\}\,\frac{\dd\varsigma}{\varsigma},
  \]
  and thus
  \begin{multline*}
    \sqrt s|y_{12}(s)| \leq
    \sqrt s e^{(\ell^2-s^2)/2} |y_{12}(\ell)| \\
    +\sqrt s \int_\ell^s e^{(\varsigma^2-s^2)/2} \Bigl\{ \frac{\|x_{12}\|_0}{\varsigma} + \frac{n-1}{\varsigma^{3/2}}\|y_{12}\|_{1/2} + \|\gamma\|_{-1}\, \|y_{12}\|_{1/2}\frac{1}{\sqrt\varsigma} \Bigr\} \,\dd\varsigma.
  \end{multline*}
  Using \(\|\gamma\|_{-1} \leq \epsilon\), and also the fact that \(s\mapsto \sqrt s e^{-s^2/2}\) is decreasing for \(s\geq \frac12\sqrt2\), one finds that
  \begin{equation}
    \label{eq-y12-bound}
    \|y_{12}\|_{1/2}\leq
    \sqrt\ell |y_{12}(\ell)| + \frac{2}{\ell\sqrt\ell}\|x_{12}\|_0
    + \frac{2(n-1)}{\ell^2}\|y_{12}\|_{1/2}
    + \frac{2\epsilon}{\ell}\|y_{12}\|_{1/2}.
  \end{equation}
  The bounds \eqref{eq-x12-bound} and \eqref{eq-y12-bound} together imply the following improved bounds:
  \begin{subequations}
    \label{eq-xy12-bounds}
    \begin{align}
      \|x_{12}\|_0 &\leq 2|x_{12}(\ell)| + 32n |y_{12}(\ell)|,    \\
      \|y_{12}\|_{1/2} &\leq \frac{16}{\ell\sqrt\ell}|x_{12}(\ell)|
                         + 2 \sqrt\ell |y_{12}(\ell)|.
    \end{align}
  \end{subequations}

  We combine these estimates with \eqref{eq-x12-ell-to-infty-estimate} to obtain, after some algebra,
  \begin{equation}
    \label{eq-x12-ell-to-infty-estimate-final}
    |x_{12}(\infty)-x_{12}(\ell)| 
    \leq
    \left(\frac{8\epsilon}{\sqrt\ell} + \frac{16}{\ell\sqrt\ell}\right)|x_{12}(\ell)| + 32n|y_{12}(\ell)|.
  \end{equation}
  This bound holds for all orbits close to the \(\rfes\) in the region \(s\geq \ell\).  To apply the bound to the orbits coming out of the \(\gf\), we recall that in \S~\ref{sec:s-goes-to-ell}, we found that if \(x_\alpha, y_\alpha, \gamma\) are given by \(\fq(\ell;j,T)\), then we have, in the case of expanders,
  \begin{align*}
    \mat x_{12}(\ell;j,T) \\
    y_{12}(\ell;j,T) \rix 
    &= Re^{-AT}
      \cos(\Omega T+\phi)
      \mat 1+\delta(\ell) \\ \delta(\ell) \rix
    \\
    &\qquad\qquad+ Re^{-AT}
      \sin(\Omega T+\phi)
      \ell^{-n+1}e^{-\ell^2/2}
      \mat \delta(\ell) \\
    C_\chi^+ +\delta(\ell)\rix.
  \end{align*}
  Here \(R= R_a (1+\epsilon_3(j,T))\) and \(\phi=\phi_a + \epsilon_4(j,T)\) are as in Lemma~\ref{lem-about-R-and-phi}, and in particular are close to the constants \(R_a, \phi_a\) that only depend on the parameter \(a\) that determines the size of the isolating block \(Q_a\).  The generic error terms \(\delta(\ell)\) all are functions of \(\ell\) that are bounded by \(\delta(\ell) = \cO(\ell^{-2})\).  Taking into account that for large \(\ell\), we have \(\ell^{-n+1}e^{-\ell^2/2}\ll \ell^{-2}\), we conclude that
  \begin{align*}
    x_{12}(\ell;j,T)
    & = Re^{-AT}\Bigl(\cos(\Omega T+\phi) + \epsilon(\ell,j,T)\Bigr),  \\
    |y_{12}(\ell;j,T)|
    & \leq C\ell^{-2}Re^{-AT}.
  \end{align*}
  Finally, we combine this with~\eqref{eq-x12-ell-to-infty-estimate-final} to get~\eqref{eq-x12-approximation-expanders}.
\end{proof}

To complete the existence proof of expanding solitons with prescribed conical ends, we consider the map
\[
  \Xi_\infty : (j,T) \mapsto \bigl(\xi(\infty;j,T), x_{12}(\infty;j,T)\bigr),
\]
whose domain is the rectangle
\[
  \mf R_{\iota m} \isdef [-\iota,\iota]\times\left[\frac{m\pi - \phi_a}{\Omega}, \frac{(m+1)\pi - \phi_a}{\Omega}\right] \subset\R^2.
\]
Recall that \(\xi= \sum_\alpha \frac{p_\alpha}{n}\xi_\alpha\).

\begin{lemma}\label{lem-Xi-degree-1-expanders}
  If \(\ell>0\) is large enough, \(a>0\) small enough, and \(\iota>0\) small enough, then for all large enough \(m\in\N\), the image \(\Xi_\infty(\mf R_{\iota m})\) contains an open neighborhood of the origin in \(\R^2\).
\end{lemma}
\begin{proof}
  The rectangle \(\mf R_{\iota m}\) has four sides.  Two of these are given by \(\Omega T + \phi_a = k\pi\) with \(k=m\) or \(k=m+1\).  The other two are given by \(j=\pm\iota\).

  On the sides where \(\Omega T+\phi_a = k\pi\), we have
  \begin{align*}
    x_{12}(\ell;j,T)&= R(j,T)e^{-AT}\bigl(
                      \cos(k\pi+\phi(j,T)-\phi_a) + \epsilon_5(j,T)
                      \bigr)\\
                    &=R(j,T)e^{-AT} \bigl(
                      (-1)^k \cos(\phi(j,T)-\phi_a) + \epsilon_5(j,T) \bigr) \\
                    &= (-1)^kR(j,T)e^{-AT} \bigl(1+\epsilon_3(a,j,T)+\epsilon_5(j,T)\bigr). 
  \end{align*}
  By Lemma~\ref{lem-about-R-and-phi}, we can choose \(a>0\) so small that \(|\epsilon_3(a,j,T)|<\frac14\) for all \(|j|\leq \iota\) and \(T\geq T_{a,\ell}\).  If \(T\) is sufficiently large, we also have \(|\epsilon_5(j,T)| < \frac14\).  So we see that \(x_{12}(\ell;j,T)\) has the same sign as \((-1)^k\).  In particular, \(x_{12}\) is positive on one of the sides where \(T\) is constant and negative on the other.

  For the average \(\xi(j,T)\) of the asymptotic slopes, we see from~\eqref{eq-xi-alpha-infty-expanders} that
  \[
    \xi(\infty; j,T) = \sum_\alpha \frac{p_\alpha}{n}\xi_\alpha(\infty; j,T) = c_2^+ j + \cO\bigl(j^2 + \epsilon(j,T)\bigr).
  \]
  We assume now that \(\iota\) is so small and \(T\) so large that the error term \(\cO(\iota^2+\epsilon(\iota,T))\) is smaller than \(c_2^+\iota\).  Then on the edges of \(\mf R_{\iota m}\) where \(j=\pm\iota\), we have \(\xi>0\) if \(j=+\iota\) and \(\xi<0\) if \(j=-\iota\).

  These observations imply that \(\Xi_\infty|\partial \mf R_{\iota m}\) maps into \(\R^2\setminus\{0\}\) with winding number \(+1\)
  and therefore that \(\Xi_\infty(\mf R_{\iota m})\) does indeed contain an open neighborhood of the origin.
\end{proof}

This completes the proof of Theorem~\ref{main-A}.

\section{Shrinking solitons}
We now turn to the case of shrinking solitons and so set \(\lambda=-1\).  In this setting, the orbits under the soliton flow \(g^t\) starting from most points \(\fp = (\xi_\alpha, y_\alpha, \gamma, \ell) \) near the \(\rfes\) do not continue for all \(t\geq 0\) and thus do not generate complete metrics with conical ends.  We saw in \S~\ref{sec-stable-mfds-of-cone-ends} that the set of points \(\fp\) that do generate complete metrics form a smooth submanifold \(\Upsilon\) of \(\R^5\times\{\ell\}\) of codimension two. (See below for the precise definition of $\Upsilon$.)  On the other hand, the orbits coming from \(\gf\) form a two dimensional submanifold of \(\R^5\times\{\ell\}\).  In this section, we construct shrinking solitons by finding intersections of \(\wug\) and \(\Upsilon\).

\subsection{Boundary conditions at \(s=\ell\) for conical ends}
Lemma~\eqref{lem-stable-mfd-cones} provides us with a real analytic map \(\bs x:\mc U \to \solspace\) from a small neighborhood \(\mc U\subset \R^3\) of the origin into the solution space \(\solspace\) such that for each \((\xi_{1,\ell}, \xi_{2,\ell}, \gamma_\ell)\in\mc U\), the function \(\bs x(s; \xi_{1,\ell},\xi_{2,\ell},\gamma_\ell)\) is a solution of the soliton equations with initial values at \(s=\ell\) given by
\[
  \bs x(\ell; \xi_{1,\ell}, \xi_{2,\ell}, \gamma_\ell) = (\xi_{1,\ell}, \xi_{2,\ell}, Y_1(\xi_{1,\ell}, \xi_{2,\ell}, \gamma_\ell), Y_2(\xi_{1,\ell}, \xi_{2,\ell}, \gamma_\ell), \gamma_\ell).
\]
Moreover, these are the only initial values in \(\mc U\) that correspond to some solution \(\bs x\in\solspace\) of the soliton equations with the prescribed values at \(s=\ell\) and with \(\|\bs x\|_\solspace \leq \epsilon\).  We write
\[
  \Upsilon = \{\bs x(\ell;\xi_{\alpha,\ell}, \gamma_\ell) : (\xi_{\alpha,\ell}, \gamma_\ell)\in\mc U\}
\]
for the set of initial values at \(s=\ell\) that generate a conical end.  We have just argued that \(\Upsilon\subset \R^5\) is a real analytic submanifold that contains the origin.  It follows directly from the construction of \(\Upsilon\) that it is a graph over the \((\xi_1, \xi_2, \gamma)\) subspace of \(\R^5\), \emph{i.e.,} that it is given by equations of the form
\begin{equation}
  \label{eq-shrinker-conical-end}
  y_\alpha = Y_\alpha(\xi_{1,\ell}, \xi_{2,\ell}, \gamma).
\end{equation}

To construct shrinking solitons, we must now see which points
\[
  \bs q(\ell; j,T) = (\xi_\alpha(\ell;j,T), y_\alpha(\ell;j,T), \gamma(\ell;j,T))
\]
coming from the unstable manifold of the \(\gf\) satisfy~\eqref{eq-shrinker-conical-end}.

\subsection{Estimating \(Y_\alpha\)}
One could approximate \(Y_\alpha(\xi_\alpha, \gamma)\) for small \(\xi_\alpha,\gamma\) by using the variational equation and analyzing the asymptotics of the fundamental solutions we have already obtained.  Alternatively, one can use the integral equations \eqref{eq-ell-integral-shrinkers} that describe \(\bs x(\xi_{\alpha,\ell},\gamma_\ell)\).  We use the latter approach to prove the following:

\begin{lemma}
  If \(\|\xi_\alpha\|_0, \|y_\alpha\|_{1/2}, \|\gamma\|_{-1}\leq\epsilon\) with \(\epsilon\) small and \(\ell\) large enough, then
  \begin{equation}
    \label{eq-Y-alpha-shrinker-estimate}
    \max_\alpha |Y_\alpha(\xi_{1,\ell},\xi_{2,\ell},\gamma_\ell)| 
    \leq \frac{4}{\ell^2}\max_\alpha |\xi_{\alpha,\ell}|
    + \frac{5}{\ell^2}|\gamma_\ell|.
  \end{equation}
\end{lemma}
\begin{proof}
  The assumption that \(\|y_\alpha\|_{1/2}\leq\epsilon\) implies that
  \[
    \int_\ell^\infty |y_\alpha(s)|\frac{\dd s}{s} \leq \frac{2}{\sqrt\ell}\|y_\alpha\|_{1/2}\leq \frac{2\epsilon}{\sqrt\ell}.
  \]
  We will also use the fact that for any real number \(|x|\leq 1\), one has \(|e^x-1|\leq (e-1)|x|\leq 2|x|\).  The integral equation (\ref{eq-ell-integral-shrinkers}a) implies that
  \[
    \|\xi_\alpha\|_0 \leq \Bigl(1+\frac{8\epsilon}{\sqrt\ell}\Bigr)|\xi_{\alpha,\ell}| + \frac{4(n-1)}{\sqrt\ell}\|y_\alpha\|_{1/2}.
  \]
  For \(y_\alpha\), we conclude from the integral equation (\ref{eq-ell-integral-shrinkers}b) that
  \begin{align*}
    \|y_\alpha\|_{1/2}
    &\leq \sup_{s\geq\ell} \sqrt s
      \int_s^\infty e^{(s^2-\varsigma^2)/2}\bigl\{|\xi_\alpha| + (n-1)|y_\alpha| + (1+y_\alpha)|\gamma|\bigr\}
      \,\frac{\dd \varsigma}{\varsigma} \\
    &\leq \sup_{s\geq\ell} \sqrt s
      \int_s^\infty e^{(s^2-\varsigma^2)/2}\Bigl\{
      \frac1\varsigma\|\xi_\alpha\|_0 
      + \frac{n-1}{\varsigma^{3/2}}\|y_\alpha\|_{1/2} 
      + 2\varsigma \|\gamma\|_{-1}\Bigr\}\,\dd \varsigma \\
    &\leq \frac{1}{\ell^{3/2}}\|\xi_\alpha\|_0
      + \frac{n-1}{\ell^2}\|y_\alpha\|_{1/2}
      + \frac{2}{\sqrt\ell}\|\gamma\|_{-1}.
  \end{align*}
  Finally, for \(\gamma\), we have
  \begin{align*}
    \|\gamma\|_{-1}
    &\leq \frac{|\gamma_\ell|}{\ell}
      + \int_\ell^\infty \sum_\beta p_\beta |y_\beta|(2+ y_\beta)\,\frac{\dd\varsigma}{\varsigma^2}\\
    &\leq  \frac{|\gamma_\ell|}{\ell}
      + 3n \max_\beta \int_\ell^\infty \|y_\beta\|_{1/2}\,\frac{\dd\varsigma}{\varsigma^{5/2}}\\
    &\leq  \frac{|\gamma_\ell|}{\ell}
      + \frac{2n}{\ell^{3/2}} \max_{\beta} \|y_\beta\|_{1/2}.  
  \end{align*}
  By combining the inequalities for \(\|y_\alpha\|_{1/2}\) and \(\|\gamma\|_{-1}\), we get
  \begin{align*}
    \max_\alpha\|y_\alpha\|_{1/2}
    &\leq  \frac{1}{\ell^{3/2}}\max_\alpha\|\xi_\alpha\|_0
      + \frac{n-1}{\ell^2}\max_\alpha\|y_\alpha\|_{1/2}
      + \frac{2|\gamma_\ell|}{\ell^{3/2}}
      + \frac{4n}{\ell^2} \max_\alpha \|y_\alpha\|_{1/2} \\
    & = \frac{1}{\ell^{3/2}}
      \bigl\{\max_\alpha\|\xi_\alpha\|_0 + 2|\gamma_\ell|\bigr\}
      + \frac{5n-1}{\ell^2} \max_\alpha\|y_\alpha\|_{1/2}.
  \end{align*}
  For sufficiently large \(\ell\), we have \(\frac{5n-1}{\ell^2} <\frac12\), and hence
  \[
    \max_\alpha \|y_\alpha\|_{1/2} \leq \frac{2}{\ell^{3/2}} \max_\alpha \|\xi_\alpha\|_0 + \frac{4}{\ell^{3/2}}|\gamma_\ell|.
  \]
  Applying this to our inequality for \(\|\xi_\alpha\|_0\), we arrive at
  \[
    \max_\alpha\|\xi_\alpha\|_0 \leq \Bigl(1+\frac{8\epsilon}{\sqrt\ell}\Bigr)|\xi_{\alpha,\ell}| + \frac{8(n-1)}{\ell^2} \max_\alpha\|\xi_\alpha\|_0 +\frac{16(n-1)}{\ell^2}|\gamma_\ell|.
  \]
  For sufficiently large \(\ell\), we can again remove the term with \(\|\xi_\alpha\|_0\) on the right and estimate
  \begin{equation}
    \label{eq-xialpha-shrinker-norm-bound}
    \max_\alpha \|\xi_\alpha\|_0 
    \leq 2\max_\alpha|\xi_{\alpha,\ell}| + \frac{32n}{\ell^2}|\gamma_\ell|.
  \end{equation}
  Going back to \(\|y_\alpha\|_{1/2}\), we conclude that for large enough \(\ell\), one has
  \[
    \max_\alpha \|y_\alpha\|_{1/2} \leq \frac{4}{\ell^{3/2}} \max_\alpha|\xi_{\alpha,\ell}| + \frac{5}{\ell^{3/2}}|\gamma_\ell|.
  \]
  Because \(|Y_\alpha| \leq \ell^{-1/2}\|y_\alpha\|_{1/2}\), this implies \eqref{eq-Y-alpha-shrinker-estimate}.
\end{proof}

\subsection{Estimating the difference \(Y_1-Y_2\)}
The approximations we have for \(Y_1\) and \(Y_2\) are identical, and differ only in the error terms.  As in the case of expanders, we can find a better estimate for the difference \(Y_{12} = Y_1-Y_2\) in terms of the difference \(\xi_{1,\ell} - \xi_{2,\ell}\).

\begin{lemma}
  \label{lem-Y12-shrinker-estimate}
  If \(\|\xi_\alpha\|_0, \|y_\alpha\|_{1/2}, \|\gamma\|_{-1}\leq\epsilon\), with \(\epsilon\) small and \(\ell\) large enough, then
  \begin{equation}
    \label{eq-Y12-shrinker-estimate}
    |Y_1(\xi_{1,\ell},\xi_{2,\ell},\gamma_\ell) -
    Y_2(\xi_{1,\ell},\xi_{2,\ell},\gamma_\ell)| 
    \leq \frac{4}{\ell^2} |\xi_{1,\ell} - \xi_{2,\ell}|.
  \end{equation}
\end{lemma}
\begin{proof}

In \eqref{eq-x12-bound}, we found an initial estimate for \(\|x_{12}\|_0\), namely
that for \(\ell\) large enough, one has
\[
  \|x_{12}\|_0 \leq \Bigl(1+\frac{8\epsilon}{\sqrt\ell}\Bigr) |x_{12}(\ell)| + \frac{16n}{\sqrt\ell} \|y_{12}\|_{1/2}.
\]
In the case of shrinkers, we subtract equations~(\ref{eq-ell-integral-shrinkers})
for \(y_1, y_2,\) which yields the integral equation \begin{equation}
  \label{eq-y12-integral_sh} y_{12}(s) = -\int_s^\infty e^{-(\varsigma^2-s^2)/2}
  \bigl\{ x_{12}(\varsigma) + (\gamma(\varsigma)-n+1)y_{12}(\varsigma)
  \bigr\}\,\frac{\dd\varsigma}{\varsigma}.
\end{equation}
Assuming as above that \(\|\xi_\alpha\|_0, \|y_\alpha\|_{1/2}, \|\gamma\|_{-1} <
\epsilon\), this integral equation for \(y_{12}\) implies that
\begin{align*}
  |y_{12}(s)|
  &\leq \int_s^\infty e^{(s^2-\varsigma^2)/2} 
  \Bigl\{\|x_{12}\|_0 + (n-1+\epsilon\varsigma)\|y_{12}\|_{1/2} \varsigma^{-1/2}\Bigr\}
  \, \frac{\dd\varsigma}{\varsigma} \\
  &\leq \frac{1}{s^2}\|x_{12}\|_0 + \frac{n-1}{s^{5/2}}\|y_{12}\|_{1/2}
  + \frac{\epsilon}{s^{3/2}} \|y_{12}\|_{1/2}.
\end{align*}
Hence,
\[
  \|y_{12}\|_{1/2} = \sup_{s\geq \ell} \sqrt s|y_{12}(s)| \leq \frac{1}{\ell^{3/2}}\|x_{12}\|_0 +\frac{n-1}{\ell^2} \|y_{12}\|_{1/2} +\frac{\epsilon}{\ell} \|y_{12}\|_{1/2}.
\]
If \(\ell\) is large enough to ensure that \((n-1)\ell^{-3/2} + \epsilon\ell^{-1} \leq \frac12\), then we get
\[
  \|y_{12}\|_{1/2} \leq \frac{2}{\ell^{3/2}} \|x_{12}\|_{0}.
\]
We already have \eqref{eq-x12-bound}, which combined with our new estimate for \(\|y_{12}\|_{1/2}\) implies that
\[
  \|x_{12}\|_{0} \leq \Bigl(1+\frac{8\epsilon}{\sqrt\ell}\Bigr)|x_{12}(\ell)| + \frac{32n}{\ell^2} \|x_{12}\|_0.
\]
For large enough \(\ell\), this leads to
\[
  \|x_{12}\|_0 \leq 2|x_{12}(\ell)|,
\]
and hence also yields
\[
  \|y_{12}\|_{1/2} \leq \frac{4}{\ell^{3/2}}|x_{12}(\ell)|.
\]
Because \(|Y_{12}(\xi_{1,\ell}, \xi_{2,\ell}, \gamma_\ell)| \leq \ell^{-1/2}\|y_{12}\|_{1/2}\), estimate~\eqref{eq-Y12-shrinker-estimate} follows.
\end{proof}

\subsection{Matching \(\fq(\ell; j,T)\) and \(\Upsilon\)}
We now consider the map
\[
  \Xi(j,T) = \mat
  y(\ell; j, T) - Y(\xi_\alpha(\ell; j, T), \gamma(\ell; j, T)) \\[1ex]
  y_{12}(\ell; j, T) - Y_{12}(\xi_\alpha(\ell; j, T), \gamma(\ell; j, T)) \rix.
\]
Since \(y = \frac{p_1}{n}y_1+\frac{p_2}{n}y_2\) and \(y_{12} = y_1-y_2\), while the functions \(Y_1\) and \(Y_2\) satisfy similar relations, the equations
\begin{align*}
  y_1(\ell;j,T) &= Y_1\big(\xi_\alpha(\ell;j,T), \gamma(\ell;j,T)\big),
  \\
  y_2(\ell;j,T) &= Y_2\big(\xi_\alpha(\ell;j,T), \gamma(\ell;j,T)\big),
\end{align*}
that determine which points \(\fq(\ell;j,T)\) belong to \(\Upsilon\) are equivalent to the equation \(\Xi(j,T) = 0\).

To evaluate the terms in this map, we revisit the approximations we found for \(\fq(\ell;j,T)\) in the context of shrinking solitons.  It follows from Lemma~\ref{lem-Phi-at-ell} that
\[
  \Phi(\ell;j,T) = c_2^- j\Phi_2^-(\ell) + j^2B(j) + E(j,T),
\]
where \(B(j)\) is bounded as \(j\to0\), and \(E(j,T) \to 0\) uniformly in \(j\) as \(T\to\infty\).  Using the asymptotic expansion~\eqref{eq-Phi2-shrink-at-infty} of \(\Phi_2^-(\ell)\) for large values of \(\ell\), we find that
\begin{subequations}
  \label{eq-xi-y-gamma-at-ell-shrinkers}
  \begin{align}
    \xi(\ell;j, T)& = C_\xi(\ell) j \ell^{-n-1}e^{\ell^2/2}
                    +b_\xi(j)j^2 + \epsilon_\xi(j,T), \\
    y(\ell;j, T) & = C_y(\ell)j \ell^{-n+1}e^{\ell^2/2}
                   +b_y(j)j^2 + \epsilon_y(j,T),  \\
    \gamma(\ell;j, T) &= C_\gamma(\ell) j \ell^{-n-1}e^{\ell^2/2}
                        +b_\gamma(j)j^2 + \epsilon_\gamma(j,T).
  \end{align}
\end{subequations}
Here, \(b_\xi,b_y, b_\gamma\) and \(\epsilon_\xi,\epsilon_y,\epsilon_\gamma\) are the components of \(B(j)\) and \(E(j,T)\), respectively, while \(C_\xi(\ell)\), \(C_y(\ell)\), and \(C_\gamma(\ell)\) are the coefficients in the expansion~\eqref{eq-Phi2-shrink-at-infty} of \(\Phi_2^-(\ell)\).  For large \(\ell\), they satisfy
\[
  C_\xi(\ell) = C_{\xi,\infty} + \cO(\ell^{-2}), \qquad C_y(\ell) = C_{y,\infty} + \cO(\ell^{-2}), \qquad C_\gamma(\ell) = C_{\gamma,\infty} + \cO(\ell^{-2}),
\]
in which none of the constants \(C_{\xi,\infty}\), \(C_{y,\infty}\), and \(C_{\gamma,\infty}\) vanish.

From equation \eqref{eq-x12-y12-at-ell} for the difference variables, we know that at \(\fq(\ell;j,T)\), one has
\begin{subequations}
  \label{eq-x12-y12-at-ell-shrinkers}
  \begin{align}
    x_{12} &= C R(j,T)e^{-AT}\ell^{-n-1}e^{\ell^2/2} 
             \bigl\{\sin(\Omega T+\phi)  + \epsilon(\ell,j,T)\bigr\},\\
    y_{12} &=  C R(j,T)e^{-AT}\ell^{-n+1}e^{\ell^2/2} 
             \bigl\{\sin(\Omega T+\phi)  + \epsilon(\ell,j,T)\bigr\}.
  \end{align}
\end{subequations}
As in the case of expanding solitons, we consider a sequence of rectangles and compute the degree of the map \(\Xi\) on these rectangles.  Specifically, we consider the domains
\[
  \mf R_{\iota m}^* = [-\iota, \iota] \times \left[ \frac{(m-\frac12)\pi-\phi_a}{\Omega}, \frac{(m+\frac12)\pi-\phi_a}{\Omega} \right] .
\]

\begin{lemma}\label{lem-Xi-degree-1-shrinkers}
  If \(\iota>0\) is small enough, then \(\Xi\) maps \(\partial\mf R_{\iota m}^*\) into \(\R^2\setminus\{(0,0)\}\) with winding number \(+1\) for each large enough \(m\in\N\).
\end{lemma}
\begin{proof}
  We first consider the sides of \(\mf R_{\iota m}^*\) on which \(|j|=\iota\).  Since we assume \(\iota\) is small enough, we may assume that
  \begin{equation}
    \label{eq-Bjj-small}
    |B(\iota)\iota| \leq
    \frac12 \min \{|C_{\xi,\infty}|, |C_{y,\infty}|, |C_{\gamma,\infty}|\} \ell^{-n-1}e^{\ell^2/2}.
  \end{equation}
  By \eqref{eq-xi-y-gamma-at-ell-shrinkers}, it then follows that for all \(j\in[-\iota,\iota]\), one has
  \begin{equation}
    \label{eq-yljT-lower-bound}
    |y(\ell; j ,T)| \geq
    \frac12 |C_{y,\infty}j|  \ell^{-n+1}e^{\ell^2/2}
    -\max_{|j|\leq\iota}\epsilon_y(j ,T).
  \end{equation}
  For \(j=\pm\iota\), we find that
  \[
    |y(\ell;\pm\iota,T)| \geq \frac14 |C_{y,\infty}j| \ell^{-n+1}e^{\ell^2/2}
  \]
  if \(T\) is large enough, because \(\epsilon_y(j, T)\to 0 \) uniformly in \(j\in[-\iota,\iota]\) as \(T\to\infty\).

  It also follows from our expression (\ref{eq-xi-y-gamma-at-ell-shrinkers}b) for \(y(\ell;j,T)\) that \(y(\ell;\pm\iota, T)\) have opposite signs once \(T\) is large enough, assuming that \(\iota>0\) is so small that \eqref{eq-Bjj-small} holds.

  According to \eqref{eq-Y-alpha-shrinker-estimate}, we further have
  \[
    \left|Y\left(\xi_\alpha(\ell;j,T), \gamma(\ell;j,T)\right)\right| \leq \frac{C}{\ell^2} \left(\max_\alpha |\xi_\alpha(\ell;j,T)| + |\gamma(\ell;j,T)|\right).
  \]
  Using assumption~\eqref{eq-Bjj-small} again, we find that
  \[
    \Big|Y\Bigl(\xi_\alpha(\ell; \pm\iota,T), \gamma(\ell; \pm\iota,T)\Bigr)\Big| \leq C \ell^{-n-3}e^{\ell^2/2}\iota.
  \]
  The constants \(C\) and \(C_{y,\infty}\) do not depend on \(\ell\), \(\iota\), or \(T\).  So we may assume that \(\ell\) is so large that \(C\ell^{-2} \leq \frac14|C_{y,\infty}|\).  This lets us conclude that for \(\xi_\alpha = \xi_\alpha(\ell; \pm\iota,T)\) and \(\gamma = \gamma(\ell; \pm\iota,T)\), one has
  \begin{equation}
    \label{eq-y-Y-matching-shrinker}
    |y(\ell; \pm\iota, T) - Y(\xi_\alpha, \gamma)| \geq
    \frac14 |C_{y,\infty}| \iota  \ell^{-n+1}e^{\ell^2/2} -\epsilon_y(\pm\iota ,T)
    >0
  \end{equation}
  for all large enough \(T\).  Since \(|Y|\) is smaller than \(|y(\ell;\iota, T)|\) if \(T\) is large enough, the signs of \(y(\ell;\pm\iota,T)\) and \(y(\ell;\pm\iota, T) - Y(\xi_\alpha, \gamma)\) coincide, and hence the terms \(y(\ell;+\iota, T) - Y(\xi_\alpha, \gamma)\) and \(y(\ell;-\iota, T) - Y(\xi_\alpha, \gamma)\) have opposite signs.

  We now consider the sides of \(\partial\mf R_{\iota m}^*\) where \(T=T_m^\pm\) is constant, \emph{i.e.,} where
  \[
    \Omega T_m^\pm=\left(m\pm \frac12\right)\pi - \phi_a.
  \]
  Here we have, because of \eqref{eq-x12-y12-at-ell-shrinkers},
  \begin{align*}
    x_{12}(\ell;j,T_m^\pm) &= C R(j,T_m^\pm)e^{AT_m^\pm}\ell^{-n-1}e^{\ell^2/2} 
                             \bigl\{\pm(-1)^m\cos(\phi(j,T_m^\pm)-\phi_a)  + \epsilon(\ell,j,T_m^\pm)\bigr\},\\
    y_{12}(\ell;j,T_m^\pm) &=  C R(j,T_m^\pm)e^{AT_m^\pm}\ell^{-n+1}e^{\ell^2/2} 
                             \bigl\{\pm(-1)^m\cos(\phi(j,T_m^\pm)-\phi_a)  + \epsilon(\ell,j,T_m^\pm)\bigr\}.
  \end{align*}
  We recall that in \S~\ref{sec-apply-lambda} we chose \(a>0\) so small that \(\cos \bigl(\phi(j,T_m^\pm) - \phi_a\bigr) \geq \frac12\) for all \(j\in[-\iota, \iota]\) and for all large enough \(m\in\N\).  We have also shown in Lemma~\ref{lem-Y12-shrinker-estimate} that
  \[
    \left|Y_{12}(\xi_\alpha(\ell;j,T_m^\pm), \gamma(\ell;j,T_m^\pm))\right| \leq \frac{4}{\ell^2} \left|x_{12}(\ell;j,T_m^\pm)\right|,
  \]
  so that we have
  \begin{multline*}
    \left| y_{12}(\ell;j,T_m^\pm)- Y_{12}(\xi_\alpha(\ell;j,T_m^\pm), \gamma(\ell; j,T_m^\pm)) \right|
    \geq |y_{12}(\ell;j,T_m^\pm)| - \frac{4}{\ell^2} |x_{12}(\ell;j,T_m^\pm)|\\
    \geq CR(j,T_m^\pm)e^{AT_m^\pm}\ell^{-n+1}e^{\ell^2/2}\Bigl\{\frac12 - \frac{C}{\ell^2} - \epsilon(\ell; j,T_m^\pm)\Bigr\}.
  \end{multline*}
  Since \(\epsilon(\ell;j,T_m^\pm)\to0\) as \(m\to\infty\), and since we may assume that \(\frac12 - \frac{C}{\ell^2} \geq \frac14\), it follows that for all large enough \(m\), one has
  \[
    \left| y_{12}(\ell;j,T_m^\pm)- Y_{12}(\xi_\alpha(\ell;j,T_m^\pm), \gamma(\ell; j,T_m^\pm)) \right| \geq \frac18 CR(j,T_m^\pm)e^{AT_m^\pm}\ell^{-n+1}e^{\ell^2/2} > 0
  \]
  for \(j\in[-\iota,\iota]\), while \(y_{12}(\ell;j,T_m^+)-Y_{12}(\ell;j,T_m^+)\) and \(y_{12}(\ell;j,T_m^-)-Y_{12}(\ell;j,T_m^-)\) have opposite signs.

  It follows from these considerations that \(\Xi(j,T) \neq 0\) for all \((j,T)\in\partial\mf R_{\iota m}^*\).  Moreover, by checking the signs of the coordinates of \(\Xi(j,T)\) on opposite sides of the rectangle \(\partial\mf R_{\iota m}^*\), one sees that \(\Xi\) has winding number \(1\).  The final conclusion is that there exist \((j_m,T_m) \in \mf R_{\iota m}^*\) such that \(\Xi(j_m,T_m) = 0\), and thus that the orbit \(\{g^t\fq(j_m, T_m) \mid t\in\R\}\) is complete and yields a soliton metric with a good fill compactification at \(s=0\) and a conical end at \(s=\infty\).
\end{proof}

At this point we have proved the existence claim in Theorem~\ref{main-B}, and we only have to prove that the asymptotic apertures \(x_{m,\alpha}^-(\infty)\) converge to those of the Ricci-flat cone as \(m\to\infty\).  This is the content of the following Lemma.

\begin{lemma}
  We have
  \[
    \lim_{m\to\infty} j_m = 0, \qquad \lim_{m\to\infty} T_m = \infty,
  \]
  and hence
  \[
    \lim_{m\to\infty} \xi_\alpha(\infty; j_m, T_m) = 0.
  \]
\end{lemma}
\begin{proof}
  By definition of \(\mf R_{\iota m}^*\), we have \(T_m = \frac m\Omega + \cO(1)\) as \(m\to\infty\), so \(T_m\to\infty\).  Since \(|j_m|\leq \iota\), we may assume that, after passing to a subsequence, \(j_m\to j_*\) for some \(j_*\in[-\iota,\iota]\).  We will show that \(j_*=0\) is the only possible limit and hence that the whole sequence \(j_m\) converges to zero.

  It follows from~\eqref{eq-x12-y12-at-ell-shrinkers} that \(x_{12}(\ell; j_m, T_m) \to 0\) and \(y_{12}(\ell; j_m, T_m) \to 0\).

  For the averaged variables, we have
  \[
    \xi(\ell; j_m, T_m) \to \xi(\ell;, j_*, \infty),\quad y(\ell; j_m, T_m) \to y(\ell;, j_*, \infty),\quad \gamma(\ell; j_m, T_m) \to \gamma(\ell;, j_*, \infty).
  \]
  For all \(m\), we have
  \[
    y_\alpha(\ell; j_m, T_m) = Y_\alpha(\xi_1(\ell; j_m, T_m), \xi_2(\ell; j_m, T_m), \gamma(\ell; j_m, T_m)).
  \]
  So averaging over \(\alpha\) and taking the limit as \(m\to\infty\), we get
  \[
    y(\ell; j_*, \infty) = Y\bigl(\xi(\ell; j_*, \infty), \xi(\ell; j_*, \infty), \gamma(\ell; j_*, \infty)\bigr).
  \]
  Passing to the limit \(T\to\infty\) in~\eqref{eq-yljT-lower-bound}, we get
  \begin{equation}
    \label{eq-yljinfty-upperbound}
    |y(\ell; j_*, \infty)| \geq \frac12|C_y(0)|\ell^{-n+1} e^{\ell^2/2} |j_*|.
  \end{equation}
  On the other hand, it follows from~\eqref{eq-Y-alpha-shrinker-estimate} that
  \begin{equation}
    \label{eq-Yljinfty-upperbound}
    \big|Y\bigl(\xi(\ell; j_*, \infty), \xi(\ell; j_*, \infty),
    \gamma(\ell; j_*, \infty)\bigr)\big|
    \leq \frac{4}{\ell^2}|\xi(\ell; j_*, \infty)| 
    + \frac{5}{\ell^2} |\gamma(\ell; j_*, \infty)|.
  \end{equation}
  To bound the right-hand side, we let \(T\to\infty\) in \eqref{eq-xi-y-gamma-at-ell-shrinkers}.  As \(T\to\infty\), both \(\epsilon_\xi(j, T)\) and \(\epsilon_\gamma(j,T)\) vanish, so that
  \begin{align*}
    |\xi(\ell; j_*, \infty)| &\leq|C_\xi(\ell)|\ell^{-n-1}e^{\ell^2/2}|j_*|
                               + |b_\xi(j_*)|\,|j_*|^2, \\ 
    |\gamma(\ell; j_*, \infty)| &\leq|C_\gamma(\ell)|\ell^{-n-1}e^{\ell^2/2}|j_*|
                                  + |b_\gamma(j_*)|\,|j_*|^2.
  \end{align*}
  The hypothesis~\eqref{eq-Bjj-small} combined with the facts that \(C_\xi(\ell)\to C_{\xi,\infty}\) and \(C_\gamma(\ell)\to C_{\gamma,\infty}\) for \(\ell\to\infty\) allow us to bound the right-hand sides here as follows:
  \begin{align*}
    |\xi(\ell; j_*, \infty)| &\leq 2|C_{\xi,\infty}|\ell^{-n-1}e^{\ell^2/2}|j_*|, \\ 
    |\gamma(\ell; j_*, \infty)| &\leq 2|C_{\gamma,\infty}|\ell^{-n-1}e^{\ell^2/2}|j_*|.
  \end{align*}
  Applying this to~\eqref{eq-Yljinfty-upperbound}, we find that
  \[
    \big|Y\bigl(\xi(\ell; j_*, \infty), \xi(\ell; j_*, \infty), \gamma(\ell; j_*, \infty)\bigr)\big| \leq C\ell^{-n-3}e^{\ell^2/2}|j_*|.
  \]
  We combine this upper bound for \(|Y(\cdots)|\) with the lower bound~\eqref{eq-yljinfty-upperbound} we have for \(|y(\cdots)|\) to obtain
  \[
    \frac12|C_{y,\infty}|\ell^{-n+1} e^{\ell^2/2} |j_*| \leq C\ell^{-n-3}e^{\ell^2/2}|j_*|.
  \]
  The constants \(C\) and \(C_{y,\infty}\) do not depend on \(\ell\), so we may assume one more time that \(\ell\) is so large that
  \[
    \frac12|C_{y,\infty}|\ell^{-n+1} e^{\ell^2/2} > C\ell^{-n-3}e^{\ell^2/2},
  \]
  which then implies \(j_*=0\), as claimed.

  At this point, we have shown that \(T_m\to\infty\) and \(j_m\to0\).  It follows that \(\xi_\alpha(\ell;j_m, T_m)\to0\) and \(\gamma(\ell;j_m, T_m)\to0\) as \(m\to\infty\).  The upper bound~\eqref{eq-Y-alpha-shrinker-estimate} for \(Y_\alpha\) further implies that \(y_\alpha(\ell; j_m, T_m)\to 0\) as \(m\to\infty\).  Next, we apply the upper bound~\eqref{eq-xialpha-shrinker-norm-bound} to conclude that \(\xi_\alpha(s; j_m, T_m)\to0\) uniformly in \(s\geq \ell\), which proves the last claim in the Lemma, namely that \(\xi_\alpha(\infty; j_m, T_m) \to 0\) as \(m\to\infty\).
\end{proof}

\section{Constructing Ricci flow spacetimes}
\label{app-diffeo}
\subsection{Ricci flows from solitons}
We have constructed Ricci-solitons $(G, \mf X, \lambda)$ on the manifold $(0,\infty)\times \mc M$, where
$\mc M=\mc S^{p_1}\times \mc S^{p_2}$ and where the expansion factor is $\lambda\in\{\pm1\}$. 
The metric and soliton vector field are of the form
\[
  G = (\mr d s)^2 + s^2 g_{\mc M}(s)\qquad\text{and}\qquad \mf X = f(s)\frac{\partial}{\partial s},
\]
where
\[
  g_{\mc M}(s) =\frac{p_1-1}{x_1(s)}g_{\mc S^{p_1}}+ \frac{p_2-1}{x_2(s)}g_{\mc S^{p_2}},
\]
and where $x_\alpha(s)$ are solutions of the soliton equations~\eqref{eq-Pa}--\eqref{eq-Pd}.

Given such a soliton, there exists a family of diffeomorphisms $\phi_t:(0,\infty)\times\mc M \to (0,\infty)\times\mc M$ defined for $\lambda t>0$ and generated by the time dependent vector field $(2\lambda t)^{-1}\mf X$.  Thus $\frac{\dd}{\dd t}\phi_t(p) = (2\lambda t)^{-1}\mf X(\phi_t(p))$ for all $t>0$ if $\lambda=+1$ or $t<0$ if $\lambda=-1$; moreover, $\phi_{1/2\lambda}=\mr {id}_{\mc N}$.
Then the time-dependent family of metrics \(g(t)\) given by
\[
  g(t) = 2\lambda t\, \phi_{t}^*(G)
\]
evolves by Ricci flow.  See \cite[\S~1.1]{CCGGIIKLLN} with $\lambda = 2\varepsilon$.

\subsection{Properties of the soliton vector field}
The form of the soliton vector field $\mf X=f(s)\pd_s$ allows us to write the
diffeomorphisms $\phi_t$ as
\[
  \phi_t(s, \omega) = \left(S_{\theta(t)}(s), \omega\right), \qquad \theta(t)\stackrel{\rm def}={\frac1{2\lambda}\log 2\lambda t},
\]
where $\{S_\theta\mid \theta\in\R\}$ is the flow on $[0, \infty)$ obtained by solving
\begin{equation}
  \frac{\pd S_\theta(s)}{\pd\theta} =  f\left(S_\theta(s)\right)\quad\mbox{ with initial condition }\quad
  S_0(s)=s.
\end{equation}

\begin{lemma}~		\label{lem-f-properties}
  \begin{enumerate}
  \item $f$ is an odd real analytic function of $s$.  In particular, $f(0)=0$.
  \item For any \(N\in\N\) there exist $K_\infty, K_1, K_2, \dots, K_N \in\R$ such that
  \begin{equation}\label{eq-app-f-asymptotics}
    f(s) = -\lambda s + K_\infty + \frac{K_1}{s}+\cdots + \frac{K_N}{s^N} + \mc O(s^{-N-1}),\qquad (s\to\infty).
  \end{equation}
  The expansion may be differentiated indefinitely.
  
  \item There exists a largest $s_0\geq 0$ such that $f(s_0)=0$.
  \item The flow $S_\theta(s)$ is defined for all $s\geq 0$ and $\theta\in\R$.
  \item The interval $[0,s_0]$ is invariant under the flow $S_\theta$; \emph{i.e.,} $S_\theta([0,s_0])=[0,s_0]$ for all $\theta\in\R$.
  \item For each $s\in(s_0, \infty)$, one has $S_\theta(s)\to\infty$ as $\lambda\theta\to-\infty$.
  \end{enumerate}
\end{lemma}
 
\begin{proof}
Statement (1) follows from Lemma~\ref{lem-Wu-gf-analytic} on the unstable manifold of
the Good Fill fixed points and the relation $\gamma(s) = \Gamma(s) +n = sf(s) +
\lambda s^2 - \sum_\alpha p_\alpha y_\alpha$; see~\eqref{eq-define-Gamma}.

(2) follows from the asymptotic expansion~\eqref{eq-xi-y-g-expansion}  of
$\Gamma(s)=\gamma(s)-n$ for large $s$ and the relation between $\Gamma$ and $f$.

Statement (2) directly implies (3). It may well be that $s_0=0$ for the solitons that we construct here. 
However, we did not see an easy argument that would prove this.

(4) follows from the fact that $f(0)=0$ and that $f(s)$ grows linearly for large $s$.

(5) holds because $f(0)=f(s_0)=0$.

(6) follows from the asymptotic expansion in (2), which implies that $\lambda f(s)
>0$ for $s>s_0$.
\end{proof}

\subsection{Convergence of the evolving metrics as $t\to0$}

The evolving metrics \(g(t)\) on \((0,\infty)\times\mc M\) are given by
\begin{equation}\label{eq-app-evolving-metrics}
g(t) = 2\lambda t \bigl(S_{\theta(t)}'(s)\bigr)^2 (\mr d s)^2 
  + 2\lambda t \bigl(S_{\theta(t)}(s)\bigr)^2 g_{\mc M}(S_{\theta(t)}(s)) .
\end{equation}
The form of this metric simplifies if we introduce a different coordinate. For each $t$ with $\lambda t>0$,
 we define a diffeomorphism $\varrho_t:[0,\infty)\to [0,\infty)$ by setting
\[
  \varrho_t(s) = \sqrt{2\lambda t}\,S_{\theta(t)}(s).
\]
In terms of $r = \varrho_t(s)$, we have $\dd r = \sqrt{2\lambda t}S_{\theta(t)}'(s)\dd s$, and thus
\[
  g(t) = \bigl(\varrho_t\times\mathrm{id}_{\mc M}\bigr)^*\bigl(\bar g(t)\bigr),\quad \text{ where }\quad
  \bar g(t) =  (\dd r)^2 + r^2 g_{\mc M}\left(\frac{r}{\sqrt{2\lambda t}}\right).
\]
The metrics  $\bar g(t) $ are defined on \(\{(r, \omega_1,\omega_2)\in \R\times\mc S^{p_1}\times\mc S^{p_2} \mid r>0\}\).

Since $x_\alpha(s)\to \bar x_\alpha$ as $s\to\infty$, we have
$$
g_{\mc M}(s) \to \frac{p_1-1}{\bar x_1}g_{\mc S^{p_1}}+ \frac{p_2-1}{\bar x_2}g_{\mc S^{p_2}}, \qquad(s\to\infty),
$$
so that the metrics \(\bar g(t)\)  converge smoothly on \((0,\infty)\times\mc S^{p_1}\times\mc S^{p_2}\)
to a cone metric as in~\eqref{eq-cone-metric},
$$
\bar g(t) \to \bar G = (\dd r)^2
+ r^2 \left[\frac{p_1-1}{\bar x_1}g_{\mc S^{p_1}}
+ \frac{p_2-1}{\bar x_2}g_{\mc S^{p_2}} \right],
\qquad (\lambda t\searrow 0).
$$

To construct a global Ricci flow spacetime, we show that the metrics $g(t)$ themselves converge.  This will follow from the convergence of $\varrho_t(s)$ as $\lambda t\searrow 0$.

\begin{lemma}\label{lem-convergence-of-varrho}
  The limit
  \[
    \lim_{\lambda t\searrow0} \varrho_t(s) = \varrho_0(s)
  \]
  exists for all $s\geq0$.  The convergence is uniform for $s\geq 0$, and in $C^\infty_{\rm loc}$ for $s>s_0$.
  For all $s\in[0,s_0]$, we have $\varrho_t(s)\to0$.
\end{lemma}
\begin{proof}  
We have $\sqrt{2\lambda t} = e^{\lambda\theta}$, which implies that $\varrho_t(s) = e^{\lambda\theta} S_{\theta}(s)$.  

Since the smooth function $f$ is bounded on any compact interval $[0,s_1]$,  the large-$s$ asymptotic
expansion~\eqref{eq-app-f-asymptotics} for $f(s)$ implies that for all $s\geq 0$, we have
\[
\big\vert f(s)  + \lambda s \big\vert \leq C
\]
for some $C>0$.  Hence
\begin{equation*}
  \left|\frac{\dd}{\dd\theta}e^{\lambda\theta}S_\theta(s)\right|
  = e^{\lambda\theta}\bigl| \lambda S_\theta(s) + f(S_\theta(s)) \bigr|
  \leq C e^{\lambda\theta}.
\end{equation*}
Integrating this, we find that $e^{\lambda\theta}S_\theta(s)$ converges as $\lambda\theta\to-\infty$, and that if $e^{\lambda\theta}S_\theta(s) \to \varrho_0(s)$, then
\[
  \left| e^{\lambda\theta}S_\theta(s)- \varrho_0(s)\right| \leq Ce^{\lambda\theta}=C\sqrt{2\lambda t}
\]
for all $s \geq 0$ and $\lambda t>0$.  We have shown that $e^{\lambda\theta}S_\theta(s)$ converges uniformly.  

If $0\leq s\leq s_0$ then $0\leq S_\theta(s) \leq s_0$ so that
$e^{\lambda\theta}S_\theta(s)\to0$ uniformly. This implies that $\varrho_0(s)=0$ for
all $s\in[0, s_0]$.

Finally, we show that $\varrho_t(s)$ converges in $C^\infty_{\rm loc}$ in the region $(s_0, \infty)$.
Let $s>s_0$ be given.  Then $S_\theta(s) > s$ for all $\lambda\theta<0$.
Integrating the differential equation $\frac{\dd}{\dd \theta}S_\theta(s) = f(S_\theta(s))$, we find that
$$
\theta = \int_{s}^{S_\theta(s)} \frac{\dd \varsigma}{f(\varsigma)}.
$$
We differentiate with respect to $s$ to get
$$
\frac{S_\theta'(s)}{f(S_\theta(s))} - \frac{1}{f(s)} = 0 \quad\implies\quad
S_\theta'(s) = \frac{f(S_\theta(s))}{f(s)}.
$$
Hence
\begin{align*}
  \varrho_t'(s) &= e^{\lambda\theta}S_\theta'(s) =  \frac{e^{\lambda\theta}f(S_\theta(s))}{f(s)}\\
  &= -\lambda \frac{e^{\lambda\theta}S_\theta(s)}{f(s)}
  + \frac{e^{\lambda\theta}}{f(s)}\{\lambda S_\theta(s) + f(S_\theta(s))\}.
\end{align*}
Since $f(S)+\lambda S$ is uniformly bounded and $e^{\lambda\theta}S_\theta(s)$ converges uniformly as $\lambda\theta\to-\infty$,
we conclude that
\[
  \varrho_t'(s) \to -\lambda\frac{\varrho_0(s)}{f(s)}
\]
uniformly for $s\geq s_1$, for any $s_1>s_0$.  This proves local $C^1$ convergence of $\varrho_t(s)$.

To get local $C^2$ convergence, we differentiate again and use $S_\theta'(s) = f(S_\theta(s))/f(s)$ to find that
\[
  S_\theta''(s) = \frac{f'(S_\theta(s))S_\theta'(s)}{f(s)} - \frac{f(S_\theta(s))f'(s)}{f(s)^2}
  =\frac{f(S_\theta(s))}{f(s)^2} \left\{ f'(S_\theta(s)) - f'(s) \right\}.
\]
As $\lambda\theta\to-\infty$, we have $S_\theta(s)\to\infty$, and thus $f'(S_\theta(s)) \to -\lambda$ by part~(3) of Lemma~\ref{lem-f-properties}.  We already have shown that $e^{\lambda\theta}f(S_\theta(s)) \to -\lambda\varrho_0(s)$.  It follows that $\varrho_t''(s) = e^{\lambda\theta}S_\theta''(s)$ also converges uniformly as $\lambda t\searrow 0$ for $s\geq s_1$, for any $s_1>s_0$.

By repeating these arguments, one proves that all higher derivatives
\(\partial_s^j\varrho_t(s)\) also converge locally uniformly for \(s>s_0\).
\end{proof}

\subsection{Gluing solitons}
For a given shrinking soliton \((G_-, \mf X_-)\) with \(\mf X_-=f_-(s)\pd_s\), we
define an ancient spacetime \((\Omega_-, \mf t_-, \pdt, g_-)\), where
\[
  \Omega_- = (-\infty,0)\times\R^{p_1+1}\times\mc S^{p_2},
\]
and where \(\mf t_-\) is the projection on the first factor \((-\infty,0)\), with
\(\pdt\) determined by \(\pdt\mf t=1\) and the requirement that the projection of
\(\pdt\) on the second factor \(\R^{p_1+1}\times\mc S^{p_2}\) vanishes.  The metric
\(g_-(t)\) on \(\{t\}\times(\R^{p_1+1}\setminus\{0\})\times\mc S^{p_2}\) is given by
\(g(t)=\sqrt{-2t}\phi_{-,t}^*(G^-)\).  Since our soliton metric satisfies the Good
Fill boundary condition, these metrics extend to smooth metrics on the spacelike
timeslice $\mf t_-^{-1}(t)=\{t\}\times\R^{p_1+1}\times\mc S^{p_2}$ for each \(t<0\).

Let \(s_-\) be the largest \(s\in[0,\infty)\) with \(f_-(s_-) = 0\), and define
\[
  \mc E_-=\left(\R^{p_1+1}\setminus\mc B_{s_-}^{p_1+1}\right) \times \mc S^{p_2}
  = (s_-,\infty)\times\mc S^{p_1}\times\mc S^{p_2}.
\]
Then Lemma~\ref{lem-convergence-of-varrho} implies that the metrics \(g_-(t)\) converge
smoothly to a metric \(g_-(0)\) on \(\mc E_-\) as \(t\nearrow0\).  Moreover, there is a
smooth isometry \(\Phi_-\) from \(\bigl(\mc E_-, g_-(0)\bigr)\) to the cone metric
\(\bar G\) on \((0,\infty)\times\mc S^{p_1}\times\mc S^{p_2}\) given
by~\eqref{eq-cone-metric}.  It follows that the ancient spacetime
\((\Omega_-, \mf t_-, \pdt, g_-(t))\) extends to a Ricci flow spacetime-with-boundary
\((\bar\Omega_-, \mf t_-, \pdt, g_-(t))\), in which
\[
  \bar\Omega_- = \Omega_-\cup \pd\Omega_-, \qquad
  \pd \Omega_- = \{0\}\times \mc E_- .
\]

We repeat the same construction for our expanding soliton \((G_+, \mf X_+)\),
this time obtaining a future spacetime \((\Omega_+, \mf t_+, \pdt, g_+)\), with
\( \Omega_+ = (0,\infty) \times \R^{p_1+1} \times \mc S^{p_2} \), in which \(\mf
t_+\) is again projection on the first component, and \(\pdt\) is determined by
the same conditions as in the case of shrinkers.  Let \(s_+\) be the largest
\(s\geq 0\) for which \(f_+(s)=0\), and define \(\mc E_+=(s_+, \infty)\times\mc
S^{p_1}\times\mc S^{p_2}\).  The metrics \(g_+(t) =
\sqrt{2t}\,\phi_{+,t}^*(G_+)\) converge smoothly on \(\mc E_+\) as
\(t\searrow0\), which allows us to extend the spacetime to include a boundary
\(\pd\Omega_+ = \{0\}\times \mc E_+\).  We again have a smooth isometry
\(\Phi_+\) from \((\mc E_+, g_+(0))\) to a cone metric.

Since we have chosen the two solitons so that their asymptotic cones have the
same apertures, it follows that the isometry \(\Phi_+\) is with the same
cone~\eqref{eq-cone-metric} as the isometry \(\Phi_-\).  Therefore one can glue
the two spacetimes \((\Omega_\pm, \mf t_\pm, \pdt, g_\pm)\) into a larger
spacetime \(\mc M = \bar\Omega_-\sqcup_{\Phi_\pm}\bar\Omega_+\) by identifying
the two boundary components \(\pd\Omega_\pm\) via the isometries \(\Phi_\pm\),
and by choosing the differentiable structure on \(\bar\Omega_+ \sqcup_{\Phi_\pm}
\bar\Omega_+\) so that the vector field \(\pdt\) is smooth across the common
boundary at \(\mf t=0\).  The two metrics \(g_\pm\) together induce a smooth
quadratic form \(\hat g(t)\) on the tangent space to each spacelike timeslice
\(\mf t_\pm^{-1}(t)\).  This metric $\mf g$ is a smooth solution of Ricci flow
on the (compact, incomplete) manifold $\mc M$ in the precise sense that $\mc
L_{\pdt} \hat g = -2\Rc[\hat g]$.  Hence $\mc M$ is a Ricci flow spacetime.

\subsection{Solutions with changing topology}
A minor variation on the preceding construction leads us to Ricci flow
spacetimes whose time slices change topology at \(\mf t=0\).  Namely, one can
exchange the order of the factors $\mc S^{p_1}\times\mc S^{p_2}$ for $t>0$ while
keeping the same cone aperture to obtain the change in topology claimed in the
Main Theorem, and also Theorem~C.   This completes our construction of Ricci
flow spacetimes.

\subsection{Maximality of the glued spacetime}
Theorem~\ref{main-C} and our Main Theorem claim that the spacetimes we construct are maximal.
To show this, suppose \((\hat{\mc M}, \mf t, \pdt, \mf g)\) is an extension of one of
the Ricci flow spacetimes  \((\mc M, \mf t, \pdt, \mf g)\) we construct above.
Consider the spacetime metric \(\mf g=(\dd\mf t)^2 + \hat g\), which is
canonically associated to a Ricci flow spacetime, and let \(\hat{\mc M}^*\)
and \(\mc M^*\) be the metric space completions of \(\hat{\mc M}\) and \(\mc
M\) with respect to the spacetime metric \(\mf g\).  Since \(\mc M\subset
\hat{\mc M}\), we get a natural inclusion \(\mc M^* \subset \hat{\mc M}^*\).
It follows from our construction that \(\mc M^*\) is the one-point
compactification of \(\mc M\) that adds the vertex \(V\) of the cone at time \(\mf
t=0\). 

If \(\mc M\subsetneqq\hat{\mc M}\), then the closure of \(\mc M\) in
\(\hat{\mc M}\) contains at least one point \(v\notin\mc M\).  This point must be a limit
in \(\hat{\mc M}\) of a sequence of points \(v_i\in\mc M\).  The sequence
\(v_i\) is a Cauchy sequence for the spacetime metric \(\mf g\) on \(\hat{\mc
M}\), and hence also a Cauchy sequence in \(\mc M\).  Therefore \(v_i\) either
converges to a point in \(\mc M\), which is impossible because then \(v=\lim
v_i\) would belong to \(\mc M\), or the sequence converges to the vertex \(V\)
in the metric completion \(\mc M^*\).  The sectional curvatures of the metric
\(\mf g\) on the time slice \(\mf t^{-1}(0)\) are unbounded near the vertex \(V\)
(see Appendix~\ref{app-cones}).  It follows that the sectional curvatures of
\(\hat{\mc M}\) are also unbounded near the point \(v\), which contradicts the
assumption that \(\tilde{\mc M}\) is a smooth Ricci flow spacetime.

\appendix

\section{Doubly-warped product geometries}
\label{DWP-geometries}
We observe that the Ricci flow spacetimes that we have constructed above are
maximal in the sense that 

\subsection{The general case}

It is well known (see, \emph{e.g.,}~\cite{Pet16}) that all curvatures of a
doubly-warped product metric
\[
  g = (\mathrm ds)^2 + \varphi_1^2\,g_{\mc S^{p_1}} + \varphi_2^2\,g_{\mc S^{p_2}}
\]
on $\mb R_+\times\mc S^{p_1}\times\mc S^{p_2}$ are convex linear combinations of
the five sectional curvatures\footnote{Our notation here is as follows: the
first parameter denotes the spherical factor(s) involved, while the second
indicates the highest-order derivative that appears.}
\[
  \kappa_{\alpha,1}=\frac{1-\varphi_{\alpha,s}^2}{\varphi_\alpha^2},\quad 
  \kappa_{\alpha,2}=-\frac{\varphi_{\alpha,ss}}{\varphi_\alpha},\quad 
  \kappa_{12,1}=-\frac{\varphi_{1,s} \varphi_{2,s}}{\varphi_1\varphi_2}.
\]
We note that the functions $\kappa_{\alpha,1}$ ($\alpha\in\{1,2\}$) are the
sectional curvatures of orthonormal planes tangent to $\mc S^{p_\alpha}$.  The
functions $\kappa_{\alpha,2}$ are the sectional curvatures of orthonormal planes
spanned by $\frac{\partial}{\partial s}$ and vectors tangent to $\mc
S^{p_\alpha}$. And $\kappa_{12,1}$ is the sectional curvature of a plane spanned
by one vector tangent to $\mc S^{p_1}$ and one tangent to $\mc S^{p_2}$.

It follows easily that the Ricci tensor of such a metric is given by
\begin{align*}
  \Rc&=\big\{p_1\kappa_{1,2} +p_2\kappa_{2,2}\big\}(\mr ds)^2\\
     &+\big\{\kappa_{1,2}+ (p_1-1)\kappa_{1,1}+p_2\kappa_{12,1}\big\}\varphi_1^2g_{\mc S^{p_1}}\\
     &+\big\{\kappa_{2,2}+(p_2-1)\kappa_{2,1}+p_1\kappa_{12,1}\big\}\varphi_2^2g_{\mc S^{p_2}}.
\end{align*}

Now let $\mf X = f(s)\,\frac{\partial}{\partial s}$ denote the gradient of a potential
function $F(s)$.  Applying the general formula
\[
  (\mc L_{\mf X}g)(V_1,V_2) = \mf X\big\{g(V_1,V_2)\big\}+g(\nabla_{V_1}\mf X,V_2)+g(V_1,\nabla_{V_2}\mf X)
\]
for the Lie derivative of a covariant 2-tensor to this special case, one sees that
\[
  \mc L_{\mf X} g = 2f_s(\mr d s)^2+2f\varphi_1\varphi_{1,s}\,g_{\mc S^{p_1}}+2f\varphi_2\varphi_{2,s}\,g_{\mc S^{p_2}}.
\]

In the main body of this work, we apply the equations above to the soliton condition
\[
  -2\Rc[g] = 2\lambda g + \mc L_{\mf X}g,
\]
were $\lambda\in\{-1,0,+1\}$ controls the rescaling of the soliton.\footnote{Compare to
  equation~(1.8) of~\cite{CCGGIIKLLN}, with $\epsilon=2\lambda$.}

\subsection{Cones}
\label{app-cones}
In the case of a cone metric, 
\[
  \varphi_1= s\,\sqrt{\frac{p_1-1}{\bar x_1}}
  \qquad\mbox{and}\qquad
  \varphi_2= s\, \sqrt{\frac{p_2-1}{\bar x_2}}
\]
with asymptotic apertures $\bar x_1,\,\bar x_2$, the sectional curvature $\kappa_{12,1}$ is given by
\[
  \kappa_{12,1} = -\frac{1}{s^2}.
\]
In particular, the norm of the curvature tensor of the cone becomes unbounded as $s\searrow0$.


\section{A representation as a mechanical system on $\R^3$}
\label{MechanicalSystem}
As a curiosity, we observe that our assumption that $p_\alpha\geq2$ allows us to define
\[
  u_\alpha = \log\vp_\alpha-\tfrac12 \ln(p_\alpha-1),\quad\alpha\in\{1,2\}, \qquad\text{ and }\qquad v = f - p_1\frac{\mr du_1}{\mr ds} - p_2\frac{\mr du_2}{\mr ds}.
\]
In these variables, the differential equations~\eqref{eq-LittleSolitonSystem} become
\begin{subequations}
  \label{eq-Mechanical}
  \begin{align}
    \ddot u_\alpha &= v \dot u_\alpha +  e^{-2u_\alpha}+\lambda,\qquad \alpha\in\{1,2\},\\
    \dot v &= p_1\dot u_1^2 + p_2\dot u_2^2 - \lambda.
  \end{align}%
\end{subequations}
These equations can be interpreted as a mechanical system in which \(s\) is ``time'' and where (in this section only) we write \(s\)-derivatives as fluxions.  In this interpretation, \(u_1\) and \(u_2\) are the coordinates of two unit-mass particles on the real line that are each subject to a force field given by \(F(u) = e^{-2u}+\lambda\), and whose motion is subject to friction with friction coefficient \(v\).  The only unusual aspect of this system from the point of view of mechanics is that the friction coefficient \(v\) can be either positive or negative, and that it is itself a function of time that satisfies an \textsc{ode}.

The derivation of the equations for \(u_1, u_2, v\) from \eqref{eq-LittleSolitonSystem} is a simple calculus exercise.  Even though it appears simpler than the original equations \eqref{eq-LittleSolitonSystem}, we will not use the mechanical system \eqref{eq-Mechanical} in this paper.  It does however make several cameo appearances.  For example, the Ivey invariant for the stationary soliton flow can be interpreted as the energy dissipation in the mechanical system.  Indeed, if \(\lambda=0\), then any solution of~\eqref{eq-Mechanical} satisfies
\[
  \frac{\mr d}{\mr ds}\left\{ \frac{p_1}{2} \bigl(\dot u_1^2 + e^{-2u_1} \bigr) + \frac{p_2}{2} \bigl(\dot u_2^2 + e^{-2u_1} \bigr) \right\} = v\dot v = \frac{\mr d\frac12 v^2}{\mr ds},
\]
which implies that the quantity \(I = p_1(\dot u_1^2+e^{-2u_1}) + p_2(\dot u_2^2+e^{-2u_2}) - v^2\) is preserved along solutions of~\eqref{eq-Mechanical}.  Similarly, the non-obvious Lyapunov function \(W\) in Gastel and Kronz' construction of the B\"ohm soliton (see \S~\ref{thm-BGK}) can be interpreted as Kinetic\(+\)Potential Energy for a renormalized version of the mechanical system~\eqref{eq-Mechanical}.

\section{An estimate for orbits near a hyperbolic fixed point}
\label{AnalysisAppendix}

\subsection{A model nonlinear system}
Consider a system
\begin{equation}\label{eq-hyperbolic-system}
  x_-' = -A_- x_- + B_-(x)x, \qquad x_+' = A_+ x_+ + B_+(x)x,
\end{equation}
where $x = (x_-, x_+) \in \R^{k_-}\times \R^{k_+}$, where $A_- : \R^{k_-} \to \R^{k_-} $, $A_+ : \R^{k_+}\to \R^{k_+} $ are constant linear maps, and where $B_\pm$ are smooth functions on some neighborhood of the origin in $\R^{k_-+k_+} $ such that $B_-(x) $ is a linear map from $\R^{k_-+k_+} $ to $\R^{k_-}$ and $B_+(x) $ is a linear map from $\R^{k_-+k_+} $ to $\R^{k_+}$.  We assume furthermore that $B_\pm (0) = 0$.

The origin $(0,0)$ is a fixed point for our system \eqref{eq-hyperbolic-system}.  The linearization of this system at the origin has the matrix
\[
  \begin{bmatrix}
    -A_- & 0 \\ 0 & A_+
  \end{bmatrix}.
\]
We make one more assumption, namely that the eigenvalues of both $A_\pm$ all have strictly positive real parts.

\subsection{An Analysis Lemma}

\label{lem-AnalysisLemma}\itshape
There is a constant $C\in \R$ that only depends on the matrices $A_\pm$ and the nonlinear functions $B_\pm$, such that for all $T>0$ and for any solution $x:[0, T]\to \R^{k_-+k_+}$ of~\eqref{eq-hyperbolic-system} such that $\sup_{0\leq t\leq T} \|x(t)\|$ is sufficiently small, one has
\[
  \|x(t)\|\leq C\bigl( e^{-\epsilon t} + e^{-\epsilon(T-t)} \bigr)\sup_{0\leq t\leq T} \|x(t)\|
\]
and
\[
  \int_0^T \|x(t)\|\, \dd t \leq C \sup_{0\leq t\leq T} \|x(t)\|.
\]
\upshape\medskip

\begin{proof}
  Briefly, we use a Gronwall-type argument to establish an exponential upper bound for $\|x(t)\|$ in the interval $[0,T]$, and then integrate this upper bound to get the claimed estimate.

  There is a $\delta > 0 $ such that all eigenvalues $ \mu$ of $A_+$ and $A_-$ satisfy $\Re \mu \geq \delta$.  Furthermore, there is a constant $C_A>0$ such that
  \[
    \|e^{-tA_\pm}\| \leq C_A e^{-\delta t}
  \]
  holds for all $t\geq 0$.  Applying the variation of constants formula to the system \eqref{eq-hyperbolic-system}, we find that on the interval $[0, T]$, both $x_+$ and $x_-$ are given by
  \begin{align*}
    x_-(t) & = e^{-t A_-}x_-(0) + \int_0^t e^{-(t-s)A_-} B_-(x(s))x(s) \,\dd s,     \\
    x_+(t) & = e^{-(T-t)A_+} x_+(T) - \int_t^T e^{-(s-t)A_+} B_+(x(s))x(s)\, \dd s.
  \end{align*}
  Since $B_\pm(0) = 0$, there is a constant $C_B>0$ such that $\| B_\pm(x)\| \leq C_B\|x\|$ holds for all sufficiently small $x$.  Thus we get
  \begin{align*}
    \|x_-(t)\| & \leq C_A e^{-\delta t}\|x_-(0)\| + C_AC_B \int_0^t e^{-\delta(t-s)} \|x(s)\|^2 \,\dd s,       \\
    \|x_+(t)\| & \leq C_A e^{-\delta(T-t)} \| x_+(T) \| + C_AC_B \int_t^T e^{-\delta(s-t)} \|x(s)\|^2 \,\dd s.
  \end{align*}
  For $K$ to be fixed below, we may assume that $\|x(s)\| \leq K$ for all $s\in [0,T]$, whence we get
  \begin{align*}
    \|x_-(t)\| & \leq C_AKe^{-\delta t} + C_AC_BK \int_0^t e^{-\delta(t-s)} \|x(s)\| \,\dd s,    \\
    \|x_+(t)\| & \leq C_AKe^{-\delta(T-t)} + C_AC_BK \int_t^T e^{-\delta(s-t)} \|x(s)\| \,\dd s.
  \end{align*}
  After adding these two inequalities, we find there exists $C_0=C_0(C_A,C_B)$ such that that for all $t\in [0, T]$, one has
  \begin{align}
    \|x(t)\| & \leq \|x_-(t)\| + \|x_+(t)\| \notag                    \\
    \label{eq-x-gronwall}
             & \leq C_0K \bigl(e^{-\delta t} + e^{-\delta(T-t)}\bigr)
               + C_0K \int _0^T e^{-\delta |s-t|} \|x(s)\|\,\dd s.
  \end{align}

  If we define
  \[
    \rho(t) = \frac{1}{2\delta} \int _0^T e^{-\delta |s-t|} \|x(s)\|\,\dd s,
  \]
  then since $G(t) = (2\delta)^{-1}e^{-|t|}$ is the Green's function for $-\frac{\dd^2}{\dd t^2} + \delta^2$, the quantity $\rho$ satisfies
  \[
    -\rho ''(t) + \delta^2 \rho(t) = \|x(t)\|
  \]
  for all $t\in(0, T)$.  We can therefore rewrite the integral inequality~\eqref{eq-x-gronwall} as
  \begin{equation}
    -\rho ''(t) + \bigl(\delta^2 - 2\delta C_0K \bigr)\rho(t)
    \leq C_0 \bigl(e^{-\delta t} + e^{-\delta (T-t)}\bigr),
    \qquad (0<t<T).
    \label{eq-phi-diff-inequality}
  \end{equation}
  Moreover, we have the boundary conditions
  \begin{equation}
    \rho(0) =  \frac{1}{2\delta}\int_0^Te^{-\delta t}\|x(t)\|\,\dd t \leq K,\quad\mbox{and}\quad
    \rho(T) \leq \frac K\delta.
    \label{eq-phi-boundary-condition}
  \end{equation}

  For any constant $M$, the function $\bar\rho(t) = M\bigl(e^{-\epsilon t}+e^{-\epsilon(T-t)}\bigr)$ satisfies $\bar\rho\,'' = \epsilon^2\bar \rho $, so that
  \[
    -\bar\rho\,''(t) +\bigl(\delta^2 - 2\delta C_0K\bigr)\bar\rho = \bigl(\delta^2 - 2\delta C_0K - \epsilon^2\bigr) \bar\rho.
  \]
  If we now choose $K$ small enough that $2C_0K < \frac12 \delta$ and then choose $\epsilon = \frac12 \delta$, we have
  \[
    \delta^2 - 2\delta CK - \epsilon^2 > \tfrac12 \delta^2 - \epsilon^2 = \tfrac 14 \delta^2.
  \]
  So the function $\bar\rho$ becomes a supersolution of the boundary-value problem~(\ref{eq-phi-diff-inequality}, \ref{eq-phi-boundary-condition}) provided that $M=C_1K$, where $C_1$ is a constant that depends on $\delta,\epsilon$, and $C_0$.  By the maximum principle, we conclude that $\rho\leq \bar\rho$, and thus that
  \[
    \rho(t) \leq C_1 K \bigl( e^{-\epsilon t} + e^{-\epsilon(T-t)} \bigr).
  \]
  Applying this to \eqref{eq-x-gronwall} and using the fact that $\epsilon < \delta$ implies that
  \[
    e^{-\delta t} + e^{-\delta(T-t)} \leq e^{-\epsilon t} + e^{-\epsilon(T-t)}.
  \]
  Since we may assume that $K\leq1$, we find that
  \[
    \|x(t)\| \leq (C_2K + C_2K^2) \bigl( e^{-\epsilon t} + e^{-\epsilon(T-t)} \bigr) \leq C_3K \bigl( e^{-\epsilon t} + e^{-\epsilon(T-t)} \bigr).
  \]

  We complete the proof by integrating over $[0,T]$, obtaining
  \[
    \int_0^T \|x(t)\|\dd t \leq \frac{2C_3}{\epsilon} K = C_4 K,
  \]
  where the constant $C$ only depends on $\delta$, $C_A$, and $C_B$, but not on $T$.
\end{proof}

\section{Asymptotics of \texorpdfstring{\(\chi(s)\)}{chi(s)} as \texorpdfstring{\(s\to\infty\)}{s->infty} }
\label{appendix-chi}
Asymptotic expansions for the solutions \(\chi\) of \eqref{eq-chi-s} are well documented and can be derived in a number of ways.  Here we indicate one possible real-variable approach.

If \(\chi:\R\to\R\) is a solution of \eqref{eq-chi-s}, namely
\[
  \chi_{ss} + \left(\frac ns+\lambda s\right) \chi_s + \frac{2(n-1)}{s^2} \chi = 0,
\]
then the function
\[
  Z(s) = \frac{\chi_s(s)}{s\chi(s)}
\]
satisfies
\begin{equation}
  \label{eq-appendix-Z-ode}
  \mc G(s,Z)  \isdef 
  \frac{1}{s}\frac{\dd Z}{\dd s} 
  =
  \frac{2(n-1)}{s^4} - \frac{n+1}{s^2}Z - \lambda Z - Z^2.
\end{equation}
As \(s\to\infty\), this equation becomes
\[
  \frac{\dd Z}{\dd (s^2/2)} = - Z(Z+\lambda) + \cO(s^{-2}),
\]
which has two constant (approximate) solutions, \(Z=0\) and \(Z=-\lambda\).

Direct substitution reveals that \(Z_4(s) = 2(n-1)/s^4\) satisfies
\[
  \frac{1}{s}\frac{\dd Z_4}{\dd s} - \mc G(s, Z_4(s)) = \cO(s^{-6}), \qquad (s\to\infty),
\]
while
\[
  Z_4^{\pm}(s) \isdef \frac{2(n-1)\pm 1}{s^4}
\]
satisfies
\[
  \frac{1}{s}\frac{\dd Z_4^\pm}{\dd s} - \mc G(s, Z_4^\pm(s)) = \pm \frac{\lambda}{s^4} + \cO(s^{-6}), \qquad (s\to\infty).
\]
If \(\lambda>0\), then this implies that \(Z_4^-(s) < Z_4^+(s)\) are lower and upper barriers for the \textsc{ode} \eqref{eq-appendix-Z-ode}, and therefore that there is a solution \(Z(s)\) with \(Z_4^-(s) \leq Z(s) \leq Z_4^+(s)\) for large \(s\).  In fact, if \(s_0\gg 1\), then any solution \(Z(s)\) of \eqref{eq-appendix-Z-ode} that satisfies \(Z_4^-(s_0) \leq Z(s_0) \leq Z_4^+(s_0)\) will continue to satisfy \(Z_4^-(s) \leq Z(s) \leq Z_4^+(s)\) for all \(s\geq s_0\).

If \(\lambda<0\), then \(Z_4^-(s)\) is an upper barrier, and \(Z_4^+(s)\) is a lower barrier.  Since \(Z_4^-(s) < Z_4^+(s)\) for all \(s\geq s_0\) if \(s_0\) is large enough, we can apply a Wa\.zewski argument and conclude that there exists at least one \(Z^*\in\big(Z_4^-(s_0) , Z_4^+(s_0)\big)\) such that the solution of~\eqref{eq-appendix-Z-ode} with \(Z(s_0) = Z^*\) satisfies \(Z_4^-(s) \leq Z(s) \leq Z_4^+(s)\) for all \(s\geq s_0\).

In either case, the conclusion is that there exists a solution \(Z(s)\) of~\eqref{eq-appendix-Z-ode} with
\[
  Z(s) = \bigl(2(n-1)+\cO(1)\bigr) s^{-4}, \qquad (s\to\infty).
\]
By repeating this argument, one finds that for any \(m\in\N\), there exists a solution that satisfies the expansion
\[
  Z(s) = \frac{A_4}{s^4} + \frac{A_6}{s^6} + \frac{A_8}{s^8} + \cdots + \frac{A_{2m}}{s^{2m}} + \cO\bigl(s^{-2m-2}\bigr), \qquad (s\to\infty).
\]
The coefficients \(A_{2j}\) can be computed inductively by substituting the formal expansion; one finds for example that \(A_4 = 2\lambda(n-1)\).

Integration then shows that \(\chi\) satisfies
\[
  \chi(s) = e^ {\int Z(s) s\dd s} = e^{C - \lambda \frac{n-1}{s^2} + \cO(s^{-4})} = e^C \Bigl\{1 - \lambda \frac{n-1}{s^2} + \cO(s^{-4})\Bigr\}, \qquad (s\to\infty).
\]

As we noted above, there exists another solution with \(Z(s)= -\lambda + o(1)\) for large \(s\).  Similar reasoning then leads to an expansion of the form
\[
  Z(s) = -\lambda + \frac{n+1}{s^2} + \frac{B_4}{s^4} + \cdots,
\]
which after integration leads to
\[
  \chi(s) = e^{-\lambda s^2/2} s^{-n-1} \bigl\{1 + \cO(s^{-2})\bigr\}, \qquad (s\to\infty).
\]

\end{document}